\documentclass[10pt]{amsart}
\usepackage{amsxtra,amscd,amsthm,amsmath,amssymb}
\usepackage{longtable} 
\usepackage{enumerate, tensor}
\usepackage{verbatim}
\usepackage{graphicx}
\usepackage{xspace}
\usepackage{ifthen}
\usepackage{tabularx}

\usepackage{color}
\usepackage{mathrsfs}

\usepackage{amsmath,xparse}

\NewDocumentCommand{\onorm}{sO{}m}{%
  {\IfBooleanTF{#1}
    {\onormaux{\left|}{\right|}{#3}}
    {\onormaux{#2|}{#2|}{#3}}}
}
\makeatletter
\newcommand{\onormaux}[3]{\mathpalette\onormaux@i{{#1}{#2}{#3}}}
\newcommand{\onormaux@i}[2]{\onormaux@ii#1#2}
\newcommand{\onormaux@ii}[4]{%
  \sbox\z@{$\m@th#1#2#4#3$}%
  \sbox\tw@{$\m@th\|$}%
  \mathopen{\hbox to\wd\tw@{\hss\vrule height \ht\z@ depth \dp\z@ width .3\wd\tw@\hss}}%
  #4
  \mathclose{\hbox to\wd\tw@{\hss\vrule height \ht\z@ depth \dp\z@ width .3\wd\tw@\hss}}%
}
\makeatother

\usepackage{amsthm}

\theoremstyle{plain}
\numberwithin{equation}{subsection}
\newtheorem{thm}[equation]{Theorem}

\newtheorem{prop}[equation]{Proposition}
\newtheorem{lemma}[equation]{Lemma}
\newtheorem{Definition}[equation]{Definition}

\newtheorem{Remark}[equation]{Remark}

\newtheorem{cor}[equation]{Corollary}

\theoremstyle{remark}
\newtheorem{para}[equation]{\bf}
\theoremstyle{plain}
\renewcommand{\subsubsection}{\addtocounter{equation}{1}{\vskip 8pt \noindent\bf\arabic{section}.\arabic{subsection}.\arabic{equation}.}}
\theoremstyle{definition}

\newcommand{\quash}[1]{}  
\newcommand{\nc}{\newcommand}
\nc{\on}{\operatorname}

\newcommand{\lps}{[\![}
\newcommand{\rps}{]\!]}
\newcommand{\llps}{(\!(}
\newcommand{\lrps}{)\!)}
\newcommand{\rk}{{\rm rk}}

\renewcommand{\phi}{\varphi}

\newcommand{\frakm}{{\mathfrak m}}

\newcommand{\calA}{{\mathcal A}}

\newcommand{\calC}{{\mathcal C}}
\newcommand{\calD}{{\mathcal D}}

\newcommand{\calH}{{\mathcal H}}

\newcommand{\calK}{{\mathcal K}}

\newcommand{\calT}{{\mathcal T}}

\newcommand{\sO}{{\mathscr O}}

\nc{\al}{{\alpha}} \nc{\be}{{\beta}}

\nc{\ve}{{\varepsilon}} \nc{\Ga}{{\Gamma}}

\nc{\La}{{\Lambda}}

\newcommand{\rK}{{\mathrm K}}
\newcommand{\rI}{{\mathrm I}}

\newcommand{\St}{{\mathrm {St}}}
\newcommand{\SL}{{\mathrm {SL}}}

\def\0{\circ}
\def\x{{\bf x}}

\newcommand{\C}{{\mathbb C}}
\newcommand{\R}{{\mathbb R}}
\newcommand{\Q}{{\mathbb Q}}

\newcommand{\ep}{\epsilon}
\newcommand{\et}{{\text{\rm \'et}}}

\newcommand{\Hom}{{\rm Hom}}

\newcommand{\Z}{{\mathbb Z}}

\newcommand{\ti}{\tilde}
\newcommand{\Spec}{{\rm Spec \, } }
\newcommand{\Spf}{{\rm Spf } }

 \renewcommand{\O}{{\mathcal O}}

\newcommand{\GL}{{\rm GL}}

\newcommand{\und}{\underline}

\newcommand{\rM}{{\rm M}}

\newcommand{\rH}{{\mathrm H}}

\def\thfill{\null\nobreak\hfill}

\def\endproof{\thfill\vbox{\hrule
  \hbox{\vrule\hbox to 5pt{\vbox to 5pt{\vfil}\hfil}\vrule}\hrule}}

\newcommand{\Gal}{{\rm Gal}}

\def\BF{{\mathbb F}}
\def\sigm{{\tilde\sigma}}

\def\fl{{\mathfrak l}}
\def\cy{{(\zeta_{\ell^\infty})}}

\newcommand{\CS}{{\mathrm {Vol}}}

\newcommand{\sF}{{\mathscr F}}

\newcommand{\Gk}{{G_k}}

\newcommand{\intprod}{\mathbin{\mathpalette\dointprod\relax}}
\newcommand{\dointprod}[2]{%
  \raisebox{\depth}{\scalebox{1}[-1]{$#1\lnot$}}}

\begin{document}

\title[Volume and symplectic structure]{Volume and symplectic structure\\ for $\ell$-adic local systems
}

\author[G. Pappas]{Georgios Pappas}
\thanks{Partially supported by  NSF grants  \#DMS-1701619,   \#DMS-2100743,  and the Bell Companies Fellowship Fund through the Institute for Advanced Study.}

\begin{abstract}
We introduce a notion of volume for an $\ell$-adic local system over an algebraic curve and, under some conditions, give a symplectic form on the rigid analytic deformation space of the corresponding geometric local system. These constructions can be viewed as arithmetic analogues of the volume and the Chern-Simons invariants of a representation of the fundamental group of a $3$-manifold which fibers over the circle and of the symplectic form on the character varieties of a Riemann surface. We show that the absolute Galois group acts on the deformation space by conformal symplectomorphisms which extend to an $\ell$-adic analytic flow. We also prove that the locus of  
local systems which are arithmetic over a cyclotomic extension is the critical set of a collection of rigid analytic functions. The vanishing cycles of these functions give additional invariants.

\end{abstract}

\address{Dept. of
Mathematics\\
Michigan State
Univ.\\
E. Lansing\\
MI 48824\\
USA} \email{pappasg@msu.edu}

\date{\today}

\maketitle

\vspace{-1ex}

\tableofcontents
\vspace{-2ex}

 \section*{Introduction}
 
 In this paper, we introduce certain constructions for \'etale $\Z_\ell$-local systems 
 (i.e. lisse $\Z_\ell$-sheaves) on proper algebraic curves  
  defined a field of characteristic different from $\ell$.  In particular, using an $\ell$-adic regulator, we define 
 a notion of $\ell$-adic volume.  We also give a symplectic form on the (formal) deformation space of a modular representation of the geometric \'etale fundamental group of the curve. (In what follows, we will use the essentially equivalent language of $\ell$-adic representations of the \'etale fundamental group.) 
 
 Our definitions can be viewed as giving analogues of constructions in the symplectic theory of character varieties of a Riemann surface and of the volume and the Chern-Simons invariants of representations of the fundamental group of a $3$-manifold which fibers over $S^1$. 
 
 Let us recall some of these classical constructions, very briefly. We start with the symplectic structure on the character varieties of the fundamental group $\Gamma=\pi_1(\Sigma)$ of a (closed oriented) topological surface $\Sigma$. To fix ideas we consider the space 
 \[
 X_G(\Gamma)={\rm {Hom}}(\Gamma, G)/G
 \]
parametrizing equivalence classes of representations of $\Gamma$ with values in a connected real reductive group $G$; there are versions for complex reductive groups. (Here, we are being intentionally vague about the precise meaning of the quotient; what is clear is that it is taken for the conjugation action on the target.)
Suppose that $\rho: \Gamma\to G$ is a representation which gives a point $[\rho]\in X_G(\Gamma)$.
The tangent space $T_{[\rho]}$ of $X_G(\Gamma)$ at 
 $[\rho]$ can be identified with ${\rm H}^1(\Gamma, {\rm Ad}_{\rho})$, where 
 ${\rm Ad}_{\rho}$ is the Lie algebra of $G$ with the adjoint action.
 Consider the composition
\[
{\rm H}^1(\Gamma, {\rm Ad}_{\rho})\times {\rm H}^1(\Gamma, {\rm Ad}_{\rho})\xrightarrow{\cup} {\rm H}^2(\Gamma, {\rm Ad}_{\rho}\otimes_{\R}{\rm Ad}_{\rho} )\xrightarrow{B} {\rm H}^2(\Gamma, \R)\cong \R,
\]
where $B$ is induced by the Killing form and the last isomorphism is given by Poincare duality. This defines a non-degenerate alternating form 
\[
T_{[\rho]}\otimes_\R T_{[\rho]}\to \R,
\]
i.e. $\omega_{[\rho]} \in \wedge ^2 T_{[\rho]}^*$. By varying $\rho$ 
we obtain a $2$-form $\omega$ over $X_G(\Gamma)$. Goldman \cite{Goldman} shows that this form is closed, i.e. $d\omega=0$, and so it gives a symplectic structure (at least over the space of ``good" $\rho$'s which is a manifold).
Note here that the mapping class group of the surface $\Sigma$ acts naturally on the character variety by symplectomorphisms, i.e. maps that respect Goldman's symplectic form. In turn, this action relates to many fascinating mathematical structures.
 
 Next, we discuss the notion of a volume of a representation. Here, again to fix ideas, we start with a (closed oriented) smooth $3$-manifold $M$ and take $\Gamma_0=\pi_1(M)$. Let $X=G/H$ be a contractible $G$-homogenous space of dimension $3$ and 
 choose a $G$-invariant volume form $\omega_X$ on $X$. A representation $\rho: \Gamma_0=\pi_1(M)\to G$ gives a flat $X$-bundle space $\pi: \tilde M\to M$ with $G$-action. The volume form $\omega_X$ naturally induces a $3$-form $\omega^\rho_X$ on $\tilde M$. Take $s: M\to \tilde M$ to be a differentiable section of $\pi$ and set
 \[
 {\rm Vol}(\rho)=\int_{M} s^*\omega^\rho_X\in \R
 \]
which can be seen to be independent of the choice of section $s$ (\cite{Goldman2}). The map $\rho\mapsto {\rm Vol}(\rho)$
 gives an interesting real-valued function on the space  $X_G(\Gamma_0)$. 
 
An important special case is when $M$ is hyperbolic, 
$G={\rm PSL}_2(\C)$, $X={\mathbb H}^3=\C\times \R_{>0}$ (hyperbolic $3$-space), $\omega_X$ is the standard volume form on ${\mathbb H}^3$,
 and $\rho_{\rm hyp}: \pi_1(M)\to {\rm PSL}_2(\C)={\rm Isom}^+({\mathbb H}^3)$ is the representation associated to the hyperbolic structure of $M={\mathbb H}^3/\Gamma_0$. Then, 
 ${\rm Vol}(\rho_{\rm hyp})={\rm Vol}(M)$, the hyperbolic volume of $M$. 
 The Chern-Simons invariant ${\rm {CS}}(M)$  of  $M$ is also related. For this, compose the map 
 \[
  {\rm H}_3(M, \Z)={\rm H}_3(\pi_1(M), \Z)\xrightarrow{\rho_{\rm hyp}} \rH_3({\rm PSL}_2(\C), \Z)
  \]
  with the ``regulator"
 \[
 {\mathfrak R}: \rH_3({\rm PSL}_2(\C), \Z)\to \C/\pi^2\Z.
 \]
 The product of $-i$ with the image of the fundamental class $[M]$ under this composition is the ``complex volume" 
 \[
 {\rm Vol}_\C(\rho_{\rm hyp})={\rm Vol}_\C(M)={\rm Vol}(M)+i 2\pi^2{\rm {CS}}(M)
 \]
 (\cite{Walter}); this can also be given by an integral over $M$. 
 
 A straightforward generalization of this construction leads to the definition of a complex volume 
 ${\rm Vol}_\C(\rho)$ for representations $\rho: \pi_1(M)\to \SL_n(\C)$ (see, for example,  \cite{Stavros}.) This uses the regulator maps (universal Cheeger-Chern-Simons classes)
 \[
  {\mathfrak R}_n: \rH_3(\SL_n(\C), \Z)\to \C/(2\pi i)^2\Z.
  \]
 (See also \cite{DHainZucker} and \cite{Goncharov}.)

  We now return to the arithmetic set-up of local systems over algebraic curves. Recall that we are considering formal deformations of a modular (modulo $\ell$) representation. We show that when the modular representation is the restriction of a representation of the arithmetic \'etale fundamental group, the absolute Galois group acts on the deformation space by ``conformal symplectomorphisms''
(i.e. scaling the symplectic form) which extend to an $\ell$-adic analytic flow. This gives an analogue of the action of the mapping class group on the character variety by symplectomorphisms we mentioned above. We show that if the curve is defined over a field $k$, the action of a Galois automorphism  that fixes the field extension $k\cy$ generated by all $\ell$-power roots of unity, is ``Hamiltonian''. We use this to express the set of deformed representations that extend to a representation of a larger fundamental group over $k\cy$ as the intersection of the critical loci for a set of rigid analytic functions $V_\sigma$, where $\sigma$ ranges over ${\rm Gal}(k^{\rm sep}/k\cy)$. The Milnor fibers and vanishing cycles of $V_\sigma$ provide interesting constructions.
 
 Let us now explain this in more detail. Let $\ell$ be a prime which we assume is odd, for simplicity. Suppose that $X$ is a smooth geometrically connected proper curve over a field $k$ of characteristic prime to $\ell$. The properness of the curve is quite important for most of the constructions.
 We fix an algebraic closure $\bar k$. Denote by $G_k=\Gal(k^{\rm sep}/k)$ the Galois group where $k^{\rm sep}$ is the separable closure of $k$ in $\bar k$, by $k{\cy}=\cup_{n } k(\zeta_{\ell^n})$  the subfield of $k^{\rm sep}$ generated over $k$ by all the $\ell^n$-th roots of unity and by  $\chi_{\rm cycl}: \Gk\to \Z_\ell^*$ the cyclotomic character.      Fix a $\bar k$-point $\bar x$ of $X$, and consider the \'etale fundamental groups
 which fit in the canonical exact sequence
\[
1\to \pi_1(X\times_k \bar k, \bar x)\to \pi_1(X, \bar x)\to G_k\to 1.
\] 
 For simplicity, we set $\bar X=X\times_k \bar k$ and omit the base point $\bar x$. 
 Let $\sF$ be an \'etale $\Z_\ell$-local system of rank $d>1$ over $X$. 
 The  local system $\sF$ corresponds to a continuous representation $\rho: \pi_1(X)\to \GL_d(\Z_\ell)$.

  The $\ell$-adic volume 
  $\CS({\sF})$ of $\mathscr F$ is, by definition, a continuous cohomology class
\[
\CS({\sF})\in \rH^1(k, \Q_\ell(-1)).
\]
Here, as usual, $\Q_\ell(n)=\Q_\ell\otimes_{\Z_\ell}  \chi_{\rm cycl}^{n}$ is
the $n$-th Tate twist. 

Note that if $k$ is a finite field or a finite extension of $\Q_p$ with $p\neq \ell$, then we have
 $ \rH^1(k, \Q_\ell(-1))=(0)$.  If $k$ is a finite extension of $\Q_\ell$, then 
 \[
 \rH^1(k, \Q_\ell(-1))\simeq \Q_\ell^{[k:\Q_\ell]}.
 \]
 If $k$ is a number field with $r_1$ real and $r_2$ complex places, then 
assuming a conjecture of Schneider \cite{Schn}, we have 
\[
\rH^1(k, \Q_\ell(-1))\simeq \Q_\ell^{r_1+r_2}.
\]
In fact, using the restriction-inflation exact sequence, we can give $\CS({\sF})$ as a continuous homomorphism 
 \[
\CS({\sF}): {\rm Gal}( k^{\rm sep}/k{\cy})\to \Q_\ell(-1)
 \]
which is equivariant for the action of ${\rm Gal}(k{\cy}/k)$. 

To define $\CS(\sF)$ we use a continuous $3$-cocycle 
\[
{\mathfrak r}_{\Z_\ell}: \Z_\ell\lps \GL_d(\Z_\ell)^3\rps\xrightarrow{\ }\Q_\ell
\]
that corresponds to the $\ell$-adic Borel regulator \cite{HubKings}. The quickest way is probably as follows
(but see also \S \ref{441}): Using the Leray-Serre spectral sequence we obtain a homomorphism 
  \[
 \rH^3(\pi_1(X),\Q_\ell/\Z_\ell)\to \rH^3_{\et}(X,\Q_\ell/\Z_\ell)\to \rH^1(k, \rH^2_{\et}(\bar X, \Q_\ell/\Z_\ell))=\rH^1(k, \Q_\ell/\Z_\ell(-1)).
 \]
 Taking Pontryagin duals gives 
 \[
 \rH^1(k, \Q_\ell/\Z_\ell(-1))^*=\rH_1(G_k, \Z_\ell(1))\to  \rH_3(\pi_1(X),\Z_\ell).
 \]
 Now compose this with the map given by $\rho$ and the $\ell$-adic regulator to obtain
 \[
 \rH_1(G_k, \Z_\ell(1))\to \rH_3(\pi_1(X),\Z_\ell)\xrightarrow{\rH_3(\rho)}\rH_3(\GL_d(\Z_\ell), \Z_\ell)\xrightarrow{{\mathfrak r}_{\Z_\ell}}\Q_\ell.
 \]
 This, by the universal coefficient theorem, gives $\CS(\sF)\in \rH^1(k, \Q_\ell(-1))$, up to sign.    
 In fact, we give an ``explicit'' continuous $1$-cocycle that represents $\CS(\sF)$ by 
 a construction inspired by classical Chern-Simons theory \cite{Freed}:
  
 Let $c$ be a fundamental $2$-cycle in $\Z_\ell\lps\pi_1(\bar X)^2\rps$. Lift $\sigma\in G_k$  to $\ti \sigma\in \pi_1(X)$ 
 and consider the (unique up to boundaries) $3$-chain 
 $\delta(\ti\sigma, c)\in \Z_\ell\lps\pi_1(\bar X)^3\rps$ with boundary
 \[
 \partial (\delta(\ti\sigma, c))=\ti\sigma\cdot c\cdot \ti\sigma^{-1}- \chi_{\rm cycl}(\sigma)\cdot c.
 \]
 Also, let $F_{\rho(\ti\sigma)}(\rho(c))$ be the (continuous) $3$-chain for $\GL_d(\Z_\ell)$ which gives the ``canonical'' boundary for the $2$-cycle $\rho(\ti\sigma)\cdot \rho(c)\cdot \rho(\ti\sigma)^{-1}-\rho(c)$.
 We set
 \[
 B(\sigma):=  {\mathfrak r}_{\Z_\ell} [\rho(\delta(\ti\sigma, c))- F_{\rho(\ti\sigma)}(\rho(c))]\in \Q_\ell.
 \]
 The map $G_k\to \Q_\ell(-1)$ given by $\sigma\mapsto \chi_{\rm cycl}(\sigma)^{-1}B(\sigma)$ is a continuous $1$-cocycle
 which is independent of choices up to coboundaries and whose class is $\CS(\sF)$.
 
 This explicit construction is more flexible and can be applied to continuous representations $\bar\rho: \pi_1(\bar X)\to \GL_d(A)$, where $A$ is a more general $\ell$-adic ring. In fact, we do not need that $\bar\rho$ extends to $\pi_1(X)$ but only that it has the following property: There is continuous homomorphism $\phi:  G_k\to {\rm Aut}(A)$ such that, for every $\sigma\in G_k$, there is a lift $\ti\sigma\in \pi_1(X)$, and a matrix $h_{\ti\sigma}\in \GL_d(A)$, with
  \[
 \bar\rho(\ti\sigma\cdot\gamma\cdot \ti\sigma^{-1})=h_{\ti \sigma}\cdot \phi(\sigma)(\bar\rho(\gamma))\cdot h_{\ti \sigma}^{-1},\quad \forall \gamma\in \pi_1(\bar X).
 \] 
 (In the above case, $A=\Z_\ell$, $\phi(\sigma)={\rm id}$, and $h_{\ti \sigma}=\rho(\ti\sigma)$.) We again obtain a (continuous) class 
 \[
 \CS_{\rho, \phi}\in \rH^1(k, \sO(\calD)(-1))
 \]
  which is independent of choices. In this, $\sO(\calD)$ 
 is the ring of analytic functions on the rigid generic fiber $\calD={\rm Spf}(A)[1/\ell]$ of $A$, with $G_k$-action given by $\phi$. 
 
 In particular, with some more work (see \S \ref{ss:lifts}, and especially Proposition \ref{Mackeyprop}), 
 we find that this construction applies to the case that $A$ is the universal formal deformation ring of an absolutely irreducible representation $\bar\rho_0: \pi_1(\bar X)\to \GL_d({\mathbb F}_\ell)$ which is the restriction of  a continuous $ \rho_0: \pi_1(X)\to \GL_d({\mathbb F}_\ell)$. Then the Galois group $G_k$ acts on $A$ and the action, by its definition, satisfies the condition above. In this case, the ring $A$ is (non-canonically) a formal power series ring $A\simeq W({\mathbb F}_\ell)\lps x_1,\ldots, x_r\rps $ and $\sO(\calD)$ is the ring of rigid analytic functions on 
 the open unit $\ell$-adic polydisk.

  Suppose now that $\ell$ is prime to $d$. We show that, in the above case of a universal formal deformation with  determinant fixed
 to be a given  character $\ep:\pi_1(\bar X)\to \Z_\ell^*$,
   the ring $A$ carries a canonical 
 ``symplectic structure". This is reminiscent of the canonical symplectic structure on the  character varieties of the fundamental groups of surfaces \cite{Goldman}. Here, it is given by a continuous non-degenerate $2$-form $\omega\in \wedge^2\Omega^{\rm ct}_{A/W}$ which is closed, i.e. $d\omega=0$. 
 
 Let us explain our construction of $\omega$. 
By definition, 
\[
\Omega^{\rm ct}_{A/W}=\varprojlim\nolimits_n \Omega_{A/\frakm^n/W},
\]
 where $\frakm$ is the maximal ideal of $A$. Consider the map 
 \[
d\log: \rK^{\rm ct}_2(A)=\varprojlim\nolimits_n \rK_2(A/\frakm^n) \to \wedge^2\Omega^{\rm ct}_{A/W}=\varprojlim\nolimits_n \Omega_{A/\frakm^n/W}
\]
obtained as the limit of
\[
 d\log_{(n)}: \rK_2(A/\frakm^n)\to   \Omega_{A/\frakm^n/W}.
\]
We first define a (finer) invariant 
\[
\kappa=\varprojlim_n \kappa_n
\]
 of the universal formal deformation  $\rho_A: \pi_1(\bar X)\to \GL_d(A)$ with values in the limit $\rK^{\rm ct}_2(A)=\varprojlim_n \rK_2(A/\frakm^n)$. For $n\geq 1$, $\kappa_n$ is the  image of $1$ under the   
 composition 
 \[
\Z_\ell(1)\xrightarrow{{\rm tr}^*} \rH_2(\pi_1(\bar X),\Z_\ell)\to \rH_2(\SL_{d+1}(A/\frakm^n), \Z_\ell)\xrightarrow{\sim} \rK_2(A/\frak m^n).
 \]
 Here the second map is  induced by $\rho\oplus \ep^{-1}$, and the third is the isomorphism
 given by stability and the Steinberg sequence
 \[
 1\to \rK_2(A/\frakm^n)\to {\rm St}(A/\frakm^n)\to \SL(A/\frakm^n)\to 1.
 \]
 Then the $2$-form $\omega\in \wedge^2\Omega^{\rm ct}_{A/W}$
is, by definition, the image 
\[
\omega=d\log(\kappa). 
\]
The closedness of $\omega$ follows immediately since all the $2$-forms in the image of $d\log$ are closed. 
We show that the form $\omega$  is also given via cup product and
Poincare duality, just as in the construction of Goldman's form above (see Theorem \ref{prop:sympl}), and that is non-degenerate. This is done by examining the tangent space of the Steinberg extension using some classical work of van der Kallen.  

In fact, this also provides an alternate argument for the closedness of Goldman's $2$-form \cite{Goldman} on character varieties. Showing that this form (which is defined using cup product and duality) is closed, and thus gives a symplectic structure, has a long and interesting history. The first proof, by Goldman, used a
gauge theoretic argument that goes back to Atiyah and Bott. A more direct proof using group cohomology was later given by Karshon \cite{Karshon}. Other authors gave different arguments that also extend to parabolic character varieties for surfaces with boundary, see for example \cite{Jeffrey}. The approach here differs substantially: We first define a $2$-form which is easily seen to be closed using ${\rm K}_2$, and then we show that it agrees with the more standard form constructed using cup product and duality. Let us mention here that Pantev-Toen-Vaqui\'e-Vezzosi  have given in \cite{PTVV} a general approach for constructing symplectic structures on similar spaces (stacks) which uses derived algebraic geometry. In fact, following this, the existence of the canonical symplectic structure on $\Spf(A)[1/\ell]$ was also shown, and in a greater generality, by Antonio \cite{Jorge}, by extending the results of  \cite{PTVV} to a rigid-analytic set-up. This uses, among other ingredients, the theory of derived rigid-analytic stacks developed in work of Porta-Yu \cite{PortaYu} (see also \cite{Pridham}). Our argument is a lot more straightforward and, in addition, gives the symplectic form over the formal scheme ${\rm Spf}(A)$. 
(However, the derived approach would be important for handling the cases in which the representation is not irreducible.)

It is not hard to see (cf. \cite{Deligne}), that the automorphisms $\phi(\sigma)$  of $A$ given by $\sigma\in G_k$, respect the form $\omega$ up to  Tate twist:
\[
\phi(\sigma)(\omega)=\chi^{-1}_{\rm cycl}(\sigma)  \omega.
\]
 (So they are ``conformal symplectomorphisms'' of a restricted type.) 
In particular, if $k$ is a finite field of order $q=p^f$, prime to $\ell$, and $\sigma$ is the geometric Frobenius ${\rm Frob}_q$, the corresponding automorphism $\phi=\phi({\rm Frob}_q)$ satisfies $\phi(\omega)=q\cdot \omega$. 
 
 The automorphism $\phi(\sigma)$ can be extended to give a ``flow": Using an argument of Poonen on interpolation of iterates, we show that we can write $\calD$ as an increasing union of affinoids 
 \[
 \calD=\bigcup\nolimits_{c\in {\mathbb N}} \bar\calD_c
 \]
  (each $\bar\calD_c$ isomorphic to a closed ball of radius $\ell^{-1/c}$) such that the following is true: 
 
 There is $N\geq 1$, and for each $c$, there is a rational number $\varepsilon(c)>0$, such that, for $\sigma\in G_k$, the action of 
 $\sigma^{nN}$ on $A$ interpolates to an $\ell$-adic analytic flow $\psi^t:=\sigma^{tN}$ on $\bar \calD_c$, defined for $|t|_\ell\leq \varepsilon(c)$, i.e. to a rigid analytic map
 \[
 \{t\ |\ |t|_\ell\leq \varepsilon(c)\}\times \bar\calD_c\to \bar\calD_c,\quad (t,\x)\mapsto \psi^t(\x),
 \]
 with $\psi^{t+t'}=\psi^t\cdot \psi^{t'}$. As $c\mapsto+\infty$, $\varepsilon(c)\mapsto 0$, and so we can think of this as a flow on $\calD$ which, as we approach the boundary, 
only exists for smaller and smaller times. (A similar construction is given by Litt in \cite{Litt} and, in the abelian case, by Esnault-Kerz \cite{EK1}.) We show that if $\chi_{\rm cycl}(\sigma)=1$, this flow is symplectic and in fact Hamiltonian, i.e. it preserves the level sets of an $\ell$-adic analytic function $V_\sigma\in \sO(A)$. More precisely, the flow $\sigma^{tN}$ gives a vector field $X_\sigma$ on $\calD$ whose contraction with $\omega$ is the exact $1$-form $dV_\sigma$. It follows that the critical points of the function $V_\sigma$ are fixed by the flow. We   use this to deduce that the intersection of the critical loci of $V_\sigma$ correspond to representations of $\pi_1(\bar X)$ that extend to $\pi_1(X\times_k k')$ for some finite extension $k'$ of $k\cy$.  The flow $\psi^t$ is an interesting feature of the rigid deformation space $\calD$ that we think deserves closer study. 
Versions of this flow construction has already been used in \cite{Litt}, \cite{EK1}, \cite{EK2}, to obtain results about the set of representations which extend to the arithmetic \'etale fundamental group. It remains to see if its symplectic nature, explained here, can provide additional information.

 The inspiration for these constructions comes from a wonderful idea of M. Kim \cite{KimCS1} (see also \cite{KimCS2}) who, guided by the folkore analogy between $3$-manifolds and rings of integers in number fields and between knots and primes, gave a construction of an arithmetic Chern-Simons invariant for finite gauge group. He also suggested (\cite{KimAGT}) to look for more general Chern-Simons type theories in number theory that resemble the corresponding theories in topology and 
mathematical physics. An important ingredient of classical Chern-Simons theory is the symplectic structure on the character variety of a closed orientable surface: When the surface is  the boundary of a $3$-manifold, the Chern-Simons construction gives a section of a line bundle over the character variety. The line bundle has a connection with curvature given by Goldman's symplectic form. One can try to imitate this construction in number theory by 
 regarding the $3$-manifold with  boundary as analogous to a ring of integers with a prime inverted. 
 
 In this paper, we have a different, simpler, analogy: 
 Our topological model is a closed $3$-manifold $M$ fibering over the circle $S^1$ with fiber a closed orientable surface $\Sigma$ of genus $\geq 1$ with fundamental group $\Gamma=\pi_1(\Sigma)$. The monodromy gives an element $\sigma$ of the mapping class group ${\rm Out}(\Gamma)$, so we can take $M$ to be the ``mapping torus" $\Sigma\times [0,1]/{\sim}$, where $(a,0)\sim (\ti\sigma(a),1)$ with $\ti\sigma:\Sigma\to \Sigma$ representing $\sigma$. There is an exact sequence
 \[
 1\to \Gamma\to \pi_1(M)\to \Z=\pi_1(S^1)\to 1
 \]
and conjugation by a lift of $1\in \Z= \pi_1(S^1)$ to $\pi_1(M)$ induces $\sigma\in {\rm Out}(\Gamma)$. A smooth projective
 curve $X$ over the finite field $k={\mathbb F}_q$ is the analogue of $M$;
 in the analogy, $\bar X$ corresponds to $\Sigma$ and the 
 outer action of Frobenius on $\pi_1(\bar X)$ to $\sigma$. The formalism extends to general fields $k$ with the Galois group $G_k$ replacing $\pi_1(S^1)$.
 The $\ell$-adic volume $\CS(\sF)$ of a local system  $\sF$ on $X$ corresponds to the (complex) volume
 of a representation of $\pi_1(M)$; this invariant includes the Chern-Simons invariant of 
 the representation. Note that a representation of $\Gamma$ gives a bundle with flat connection over $\Sigma$. This extends to a bundle with connection 
 on $M$ which corresponds to a representation of $\pi_1(M)$ if the connection is flat. Flatness occurs
at critical points of the Chern-Simons functional. So, in our picture, $V_\sigma$ 
 is an analogue of this functional. In fact, it is reasonable to speculate that the value $V_\sigma(\x)$  at
 a point $\x$ which corresponds to a representation $\rho$ of $\pi_1(X)$ relates to the $\ell$-adic volume $\CS(\rho)$;
 we have not been able to show such a statement.

 In topology, such constructions are often a first step
 in the development of various ``Floer-type" theories. It seems that the most relevant for our analogy are theories for non-compact complex groups like $\SL_2(\C)$, for which there is a more algebraic treatment. A modern viewpoint 
 for a particular version of these is, roughly, as follows: Since the character variety of $\Gamma$ has a (complex, or even algebraic) symplectic structure and $\sigma$ acts by a symplectomorphism, the fixed point locus of $\sigma$ (which are points extending to representations of $\pi_1(M)$) is an intersection of two complex Lagrangians. 
Hence, it acquires a $(-1)$-shifted symplectic structure in the sense of 
 \cite{PTVV}; this is the same as the shifted symplectic structure on the derived moduli stack of $\SL(2,\C)$-local systems over $M$ constructed in loc. cit.  By \cite{Joyce}, the fixed point locus with its shifted symplectic structure is locally the (derived) critical locus of a function and one can define Floer homology invariants of $M$ by using the vanishing cycles of this function, see \cite{AM}. There are similar constructions in Donaldson-Thomas theory (see, for example, \cite{Ber}). 
 Such a construction  can also be given in our set-up by using the potentials $V_\sigma$, see \S \ref{Milnor}. Passing to the realm of wild speculation,
 one might ponder the possibility of similar, Floer-type, constructions on spaces of representations of the Galois group of a number field or of a local $p$-adic field. We say nothing more about this here. We will, however, mention that the idea of viewing certain spaces of Galois representations as Lagrangian intersections was first explained by M. Kim in \cite[Sect. 10]{KimAGT}. 
 
Classically, the Chern-Simons invariant and the volume are hard to calculate directly for closed manifolds. They can also be defined for manifolds with boundary; combined with various ``surgery formulas" this greatly facilitates calculations. 
We currently lack examples of such calculations in our arithmetic set-up. We hope that  extending the theory to non-proper curves will lead to some explicit calculations and a better understanding of the invariants. Indeed, there should be such an extension, under some assumptions.
For example, we expect that there is a symplectic structure on 
the space of formal deformations of a representation of the fundamental group of a non-proper curve when the monodromy at the punctures is fixed up to conjugacy. 
We also expect that, in the case that $k$ is an $\ell$-adic field, the invariants $\CS(\sF)$ and $V_\sigma$ can be calculated using methods of $\ell$-adic Hodge theory. We hope to return to some of these topics in another paper.
 
 \medskip
 
  \noindent{\bf Acknowledgements:} We thank J. Antonio, H. Esnault, E. Kalfagianni, M. Kim, D. Litt and P. Sarnak, for useful discussions, comments and corrections. In particular, the idea of constructing a Hamiltonian vector field from the Galois action on the deformation space arose in conversations of the author with M. Kim. We also thank the IAS for support and the referee for constructive comments that improved the presentation.
  
  \medskip
  
 \noindent{\bf Notations:} Throughout the paper $\mathbb N$ denotes the non-negative integers and $\ell$ is a prime.
 We denote by ${\mathbb F}_\ell$ the finite field of $\ell$ elements and by $\Z_\ell$, resp. $\Q_\ell$,
 the $\ell$-adic integers, resp. $\ell$-adic numbers. We fix an algebraic closure $\bar\Q_\ell$ of $\Q_\ell$
 and we denote by $|\ |_\ell$ (or simply $|\ |$), resp. $v_\ell$, the $\ell$-adic absolute value, resp. $\ell$-adic valuation on $\bar\Q_\ell$, normalized so that $|\ell|_\ell=\ell^{-1}$, $v_\ell(\ell)=1$. We will denote by $\BF$ a field of characteristic $\ell$ which is algebraic over the prime field ${\mathbb F}_\ell$ and by $W(\BF)$, or simply $W$, the ring of Witt vectors with coefficients in $\BF$. If $k$ is a field of characteristic $\neq \ell$ we fix an algebraic closure $\bar k$. We denote by $k^{\rm sep}$ the separable closure of $k$ in $\bar k$, by $k\cy=\cup_{n\geq 1} k(\zeta_{\ell^n})$  the subfield of $k^{\rm sep}$ generated over $k$ by the (primitive) $\ell^n$-th roots of unity $\zeta_{\ell^n}$ and by  $\chi_{\rm cycl}: {\rm Gal}(k^{\rm sep}/k)\to \Z_\ell^\times$ the cyclotomic character defined by
 \[
 \sigma(\zeta_{\ell^n})=\zeta_{\ell^n}^{\chi_{\rm cycl}(\sigma)}
\]
for all $n\geq 1$. We set
$
\Gk={\rm Gal}(k^{\rm sep}/k).
$
Finally,  we will denote by $(\ )^*$ the Pontryagin dual, by $(\ )^\vee$ the linear dual, and by $(\ )^\times$ the units.

\medskip

\section{Preliminaries}\label{sectPre}
\medskip

We start by giving some elementary facts about $\ell$-adic convergence of power series and then recall 
constructions in the homology theory of (pro)-finite groups.

\subsection{Factorials}
For $a\in \mathbb N$ we can write its unique $\ell$-adic expansion $a=a_0+a_1\ell+\cdots +a_d\ell^d$, $0\leq a_i<\ell$.
Write $s_\ell(a)=a_0+\cdots +a_d$, resp. $d_\ell(a)=d+1$, for the sum, resp. the number of digits.  We obviously have
$|a^{-1}|_\ell\leq  \ell^{d_\ell(a)-1}$ and
the following identity is well-known:
\begin{equation}\label{factorial}
v_\ell(a!)=\sum_{i=1}^\infty\left[\frac{a}{\ell^i}\right]=\frac{a-s_\ell(a)}{\ell-1}.
\end{equation}
It follows that 
\begin{equation}\label{factorial2}
|a!|_\ell\geq \ell^{-a/(\ell-1)}=(\ell^{-1/(\ell-1)})^a.
\end{equation}

For ${\bf a}=(a_0,\ldots, a_n)\in {\mathbb N}^{n+1}$, we write $\onorm{\bf a}=a_0+\cdots +a_n$. 

\begin{lemma}\label{multiestimate}
For all ${\bf a}=(a_0,\ldots, a_n)\in {\mathbb N}^{n+1}$,
\[
|(\frac{a_0! a_1!\cdots a_n!}{(\onorm{\bf a}+n)!})|_\ell\leq |n!|^{-1}_\ell\cdot \ell^{(n+1)d_\ell(\onorm{\bf a}+n)}.
\]
\end{lemma}

\begin{proof}
By (\ref{factorial}),
\[
v_\ell((\frac{a_0! a_1!\cdots a_n!}{(\onorm{\bf a}+n)!}))=\frac{-n+s(\onorm{\bf a}+n)-\sum_{i=0}^n s(a_i)}{\ell-1}.
\]
For  $a$, $b\geq 1$, we have
\[
(s(a)+s(b))-(\ell-1) d_{\ell}(a+b)\leq s(a+b)\leq s(a)+s(b).
\]
This gives
\[
s(\onorm{\bf a}+n)-\sum_{i=0}^n s(a_i)\geq s(n)-(n+1)(\ell-1)d_\ell(\onorm{\bf a}+n).
\]
Hence,
\[
v_\ell((\frac{a_0! a_1!\cdots a_n!}{(\onorm{\bf a}+n)!}))\geq \frac{-n+s(n)  }{\ell-1}+(n+1)d_\ell(\onorm{\bf a}+n)=-v_\ell(n!)-(n+1)
d_\ell(\onorm{\bf a}+n)
\]
which gives the result.
\end{proof}

Fix $c, f\in \Q_{>0}$. We have $\lim_{x\to +\infty} (cd_\ell(x)-fx)=-\infty$. Set
\[
N(c, f)={\rm sup}_{x\in {\mathbb N}_{\geq 1}} (cd_\ell(x)-fx).
\]
The proof of the following is left to the reader.
\begin{lemma}\label{trivial}
 For each $c$, we have $\lim_{f\to +\infty}N(c,f)=-\infty$.\endproof
\end{lemma}

\subsection{$\ell$-adic convergence}

 Let $E\subset \bar\Q_\ell$ be a finite extension of $\Q_\ell$ with ring of integers $\O=\O_E$ and residue field $\BF$.
Then $\O$ is a finite $W(\BF)$-algebra with ramification index $e$. 
Let $\fl$ be a uniformizer of $\O$; then $ |\fl|_\ell=(1/\ell)^{1/e}$.
Let $R=\O\lps x_1,\ldots, x_m\rps$ be the local ring of formal power series with coefficients in $\O$ and maximal ideal
$\frakm=(\fl, x_1,\ldots, x_m)$.  
We will allow $m=0$  which gives $R=\O$.

For $\x=(x_1,\ldots, x_m)\in \bar\Q_\ell^m$, set $||\x ||={\rm sup}_i|x_i|_\ell$. 
For a multindex ${\bf i}=(i_1,\ldots, i_m)$, we use the notations $\onorm{\bf i}=i_1+\cdots +i_m$ and $\x^{\bf i}=x_1^{i_1}\cdots x_m^{i_m}$. For $r\in \ell^\Q$, $0<r\leq 1$, denote by
\[
D_r(m)=\{\x\ |\ ||\x ||<r\},\quad \bar D_r(m)=\{\x\ |\ ||\x||\leq r\},
\]
 the  rigid analytic open polydisk, resp. closed polydisk, of radius $r$ over $E$.
 (We omit $E$ from the notation).  We let
 \[
 \sO(\bar D_r(m))=\{\sum\nolimits_{{\bf i}\in {\mathbb N}^{m}} a_{\bf i}\x^{\bf i}\ |\ a_{\bf i}\in E, \lim_{\onorm{\bf i}\to \infty}|a_{\bf i}|_\ell r^{\onorm{\bf i}}= 0\},
 \]
 \[
 \quad \sO(D_r(m))=\bigcap_{r'<r}\sO(\bar D_{r'}(m)),
 \] 
for the $E$-algebra of rigid analytic functions on the open, resp. closed, polydisk.
For $f=\sum_{\bf i} a_{\bf i}\x^{\bf i}\in  \sO(\bar D_r(m))$, set 
 \[
 ||f||_r=\sup_{\bf i}|a_{\bf i}|_\ell r^{\onorm{\bf i}}=\sup_{\x\in \bar D_r(m)}||f(\x)||
 \]
  for the Gauss norm. The $E$-algebra $\sO(\bar D_r(m))$ is complete for $||\ ||_r$ and $\sO(D_r(m))$ is a Fr\'echet space
 for the family of norms $\{||\  ||_{r'}\}_{r'<r}$.
 For simplicity, we will write $D=D_1(m)$ when $m$ is understood, and often write
 $\sO(D)$ or simply $\sO$ instead of $\sO(D_1(m))$.
 
The following will be used in \S \ref{lregulator} and \S \ref{App}.

\begin{prop}\label{sequence}
a) Consider the formal power series in $E[[x_1,\ldots , x_m]]$ 
\[
F=\sum_{\bf a\in {\mathbb N}^k}\xi_{{\bf a}}\cdot G_{{\bf a}},
\]
with $\xi_{{\bf a}}\in E$,
$
|\xi_{{\bf a}}|_\ell\leq C_1\ell^{C_2d_\ell(\onorm{\bf a})}$, $G_{{\bf a}}\in \frakm^{B\onorm{\bf a}},
$ 
where  $C_1$, $C_2$, $B$ are positive constants. Then $F$  converges to a function in $\sO(D_1(m))$,
and for every $a\in \Q\cap (0,1/e]$ 
\begin{equation}\label{hypo1}
||F||_{(1/\ell)^a}\leq C_1\ell^{N(C_2, aB)}.
\end{equation}

 b) Suppose that $(F_n)$ is a sequence in $\sO(D_1(m))$ whose terms are power series 
 given as in part (a), i.e.
 \[
F_n=\sum_{\bf a\in {\mathbb N}^k}\xi_{{\bf a}, n}\cdot G_{{\bf a}, n}
\]
with $\xi_{{\bf a}, n}\in E$, $G_{{\bf a}, n}\in \frakm^{B(n)\onorm{\bf a}}$. Assume that
$
|\xi_{{\bf a}, n}|_\ell\leq C_1\ell^{C_2d_\ell(\onorm{\bf a})}$, where  $C_1$, $C_2$ are constants and $B(n)$
a function of $n$ with $\lim_{n\mapsto +\infty}B(n)=+\infty$.
 Then  the series $F=\sum_{n\geq 0} F_n$ converges in $\sO=\sO(D_1(m))$.
 \end{prop}
 
 \begin{proof} Observe that $G\in \frakm^k$ implies that $||G||_{(1/\ell)^a}\leq ((1/\ell)^{a})^k=\ell^{-ak}$,
 for all $a\in \Q\cap(0,1/e]$.
We obtain that 
 \[
|| \xi_{{\bf a}}\cdot G_{{\bf a}}||_{(1/\ell)^a}\leq C_1 \ell^{C_2d_\ell(\onorm{\bf a})-aB\onorm{\bf a}}.
 \]
Since $\lim_{\onorm{\bf a}\to \infty} C_2d_\ell(\onorm{\bf a})-aB\onorm{\bf a}=-\infty$, $F$ converges. The inequality
(\ref{hypo1}) follows from the definition of $N(c, f)$. This shows (a). Part (b) now follows by using (a),
  Lemma \ref{trivial}, and the Fr\'echet property of $\sO(D_1(m))$. 
 \end{proof}

\subsection{Homology of groups}\label{ss:homology}

Let $H$ be a (discrete) group and suppose that $C_\bullet(H)\to \Z\to 0$ is the bar resolution of $\Z$ by free
(left) $\Z[H]$-modules. Set $\bar C_\bullet(H)=\Z\otimes_{\Z[H]}C_\bullet(H)=(\bar C_j(H), \partial_j)$ for the corresponding complex which calculates the homology groups $\rH_\bullet(H,\Z)$. 
Then, $\bar C_n(H)$ is the free abelian group generated by 
 elements $[h_1|h_2|\cdots |h_n]$ and the boundary map
 \[
 \partial_n: \bar C_n(H)\to \bar C_{n-1}(H)
 \]
  is given by
the usual formula
\[
\partial_n([h_1|\cdots |h_n])=[h_2|\cdots |h_{n}]+
\]
\[
\ \ \ \ \ \ \ \ \ \ \ \ \ +\sum_{0<j<n}(-1)^j[h_1|\cdots |h_jh_{j+1}|\cdots |h_n]+(-1)^n[h_1|\cdots |h_{n-1}].
\]
Set 
\[
\bar C_{3,2}(H)=\tau_{[-3,-2]}\bar C_\bullet(H)[-2]
\]
for the complex in degrees $-1$ and $0$ obtained by shifting the truncation
\[
\bar C_3(H)/{\rm Im}(\partial_4)\xrightarrow{\partial_3} {\rm ker}(\partial_2)
\]
of $\bar C_\bullet(H)$.  
Its homology groups are
\[
\rH^{-1}(\bar C_{3,2}(H))=\rH_3(H, \Z),\quad \rH^{0}(\bar C_{3,2}(H))=\rH_2(H, \Z).
\]
For $h\in H$, denote by ${\rm inn}_h={}^h(\ ): H\to H$ the inner automorphism given by $x\mapsto {}^hx=hxh^{-1}$.
It induces chain homomorphisms
\[
{\rm inn}_h: C_\bullet(H)\to C_\bullet(H),\quad {\rm inn}_h: \bar C_\bullet(H)\to \bar C_\bullet(H).
\]
Note
\[
{\rm inn}_{h}([g_1|\cdots |g_n])=[hg_1h^{-1}|\cdots |hg_nh^{-1}].
\]
It is well-known that ${\rm inn}_h$ induces the identity on homology groups.
In fact, the formula
\[
F_h([g_1|\cdots |g_n])= \sum_{0\leq r\leq n} (-1)^r [g_1|\cdots |h^{-1}|hg_{r+1}h^{-1}|\cdots |hg_nh^{-1}].
\]
defines a graded chain map
\[
F_h: \bar C_{\bullet}(H)\to \bar C_{\bullet+1}(H)
\] 
such that, for all $c\in \bar C_\bullet(H)$,
\[
{\rm inn}_h(c)-c=F_{h}(\partial(c))+\partial (F_{h}(c)),
\]
i.e. giving a homotopy between ${\rm inn}_h$ and the identity. 
(cf. \cite{KimCS1} Appendix B by Noohi, Prop. 7.1.)
By loc. cit., Cor. 7.3, we have
\begin{equation}\label{Fid}
F_{hh'}=F_h\cdot {\rm inn}_{h'} +F_{h'}
\end{equation}
in $\bar C_{\bullet +1}(H)/{\rm Im}(\partial)$.
(In loc. cit., Noohi gives an explicit $F_{h, h'}: \bar C_\bullet(H)\to \bar C_{\bullet +2}(H)$
such that $F_{hh'}-F_{h'}-F_h\cdot {\rm inn}_{h'}=\partial F_{h, h'}$.)

\subsubsection{}\label{shufflepar}
Suppose that $H'\subset H$ is a subgroup and $h\in H$ is in the centralizer $\mathfrak Z_H(H')$ of $H'$
in $H$.
Then, for $h'_1,\ldots, h'_n\in H'$, 
\begin{equation}\label{shuffle}
F_{h}([h'_1|\cdots |h'_n])= \sum_{0\leq r\leq n} (-1)^r [h'_1|\cdots |h'_r|h^{-1}|h'_{r+1}|\cdots |h'_n].
\end{equation}
Furthermore, if $z\in Z_2(H')\subset Z_2(H)$ is a $2$-cycle, then homotopy identity gives $\partial_3 F_h(z)=0$.
Hence,  $F_h$ induces $[F_h]: \rH_2(H')\to \rH_3(H)$. The identity (\ref{Fid}) gives
 \[
[F_{h_1h_2}]=[F_{h_1}]+[F_{h_2}]
\]
 for $h_1$, $h_2$ centralizing $H'$. Therefore, we obtain a homomorphism
 \begin{equation}\label{mapshuffle}
 \rH_1(\mathfrak Z_H(H')) \otimes_{\Z} \rH_2(H')\to \rH_3(H); \quad \ (h', z)\mapsto [F_{h'}(z)].
 \end{equation}
We can now see that  this homomorphism agrees up to sign with the composition
\[
\rH_1(\mathfrak Z_H(H'))\otimes_{\Z} \rH_2(H')\xrightarrow{\nabla} \rH_3(\mathfrak Z_H(H')\times H')\xrightarrow {i} \rH_3(H)
\]
where the first map is obtained by the $\times$-product in group homology and the second 
is the natural map given by the group homomorphism $\mathfrak Z_H(H')\times H'\to H$, $(h, h')\mapsto hh'=h'h$.
Indeed, the $\times$-product  is given by the shuffle product and so 
the class of
\[
(i\cdot \nabla)([h]\otimes [h'_1|\cdots |h'_n])
\] 
is (up to sign)
the same as in formula (\ref{shuffle}). In view of this fact, we will set
\[
\nabla_{h, H'}=[F_h]: \rH_2(H')\to \rH_3(H), 
\]
for $h\in \mathfrak Z_H(H')$ 
and denote the map (\ref{mapshuffle}) by $\nabla_{\mathfrak Z_H(H'), H'}$.

\subsection{Homology of profinite groups}
Let now $H$ be a profinite  group. Set 
\[
\Z_\ell\lps H\rps=\varprojlim\nolimits_U \Z_\ell[H/U]
\]
for the complete $\Z_\ell$-group ring of $H$, where the limit is over finite index open normal subgroups $U\subset H$. 
 
 We can consider homology $\rH_i(H, \ -)$ with coefficients
in compact $\Z_\ell\lps H\rps$-modules (see \cite{Brumer}, \cite{Symonds}, \cite{NSW}). Recall that a $\Z_\ell\lps H\rps$-module is called compact
if it is given by the inverse limit of finite discrete $\ell$-power torsion discrete $H$-modules. The category 
of compact $\Z_\ell\lps H\rps$-modules has enough projectives. In fact, there is a standard profinite bar resolution 
(\cite{RibesZa})
$C_\bullet(H)_\ell\to \Z_\ell\to 0$ with terms
\[
C_n(H)_\ell=\Z_\ell\lps H^{n+1}\rps=\varprojlim\nolimits_{U, i} \Z/\ell^i\Z[(H/U)^{n+1}]
\]
and the standard differential.

We can now give the complexes
\[
\bar C_\bullet(H)_\ell= C_\bullet(H)_\ell\hat\otimes_{\Z_\ell\lps H\rps }\Z_\ell,\quad \bar C_{3,2}(H)_\ell=(\tau_{[-3,-2]}\bar C(H)_\ell)[-2]
\]
 similarly to before.  We have
\[
\rH^{-1}(\bar C_{3,2}(H)_\ell)=\rH_3(H, \Z_\ell),\quad \rH^{0}(\bar C_{3,2}(H)_\ell)=\rH_2(H, \Z_\ell).
\]
Similarly to the above, we have chain morphisms
\[
{\rm inn}_h: C_\bullet(H)_\ell\to  C_\bullet(H)_\ell,\quad {\rm inn}_h: \bar C_\bullet(H)_\ell\to \bar C_\bullet(H)_\ell,
\]
and a homotopy $
F_h: \bar C_{\bullet}(H)_\ell\to \bar C_{\bullet+1}(H)_\ell$
between ${\rm inn}_h$ and the identity which satisfies (\ref{Fid}) in $\bar C_{\bullet +1}(H)/{\rm Im}(\partial)_\ell$. The rest of the identities in the previous paragraph are also true.

\medskip

\section{$\rK_2$ invariants and $2$-forms}\label{sect2}

In this section, we recall the construction of ``universal" invariants of  representations with trivial determinant.
These take values in the second cohomology of the group with coefficients either in Milnor's $\rK_2$-group
or in (closed) K\"ahler $2$-forms of the ground ring. Using some old work of van der Kallen, we
reinterpret the evaluation of these invariants on the tangent space via a cup product in cohomology.
This allows us to show that a representation of a Poincare duality group in dimension $2$ 
with trivial determinant gives a natural {\sl closed} $2$-K\"ahler form over the ground ring. This construction provides an algebraic argument for the existence of Goldman's symplectic form on the 
character variety of the fundamental group of a closed Riemann surface (\cite{Goldman}). 
 
Let $A$ be a (commutative) local ring so that $\rK_1(A)=A^\times$ and $\SL(A)$
is generated by elementary matrices. We have the canonical Steinberg central extension (\cite{Milnor})
\[
1\to \rK_2(A)\to {\St}(A)\to \SL(A)\to 1.
\]
The group $\GL(A)$ acts on $\St(A)$ by conjugation in a way that lifts the standard conjugation action on $\SL(A)$
and the action fixes every element of $\rK_2(A)$
(\cite{Weibel}, Exerc. 1.13, Ch. III.)

\subsection{A $\rK_2$ invariant}

Suppose that $\Gamma$ is a discrete group and $\rho: \Gamma\to \SL(A)$ is a group homomorphism. For each $\gamma\in \Gamma$,
choose a lift $s(\rho(\gamma))\in \St(A)$ of $\rho(\gamma)$. Then 
\[
\kappa_\rho:\Gamma\times \Gamma\to \rK_2(A)
\]
 given by 
\[
\kappa_\rho(\gamma_1, \gamma_2):=s(\rho(\gamma_1\gamma_2))s(\rho(\gamma_2))^{-1}s(\rho(\gamma_1))^{-1}\in \rK_2(A)
\]
is a $2$-cocycle. The corresponding class 
\[
\kappa_\rho\in \rH^2(\Gamma, \rK_2(A))
\]
 is independent 
of the choice of lifts and depends only on the equivalence class of $\rho$ up to $\GL(A)$-conjugation. 
The class $\kappa_\rho$ is the pull-back via $\rho$ of a universal class in $\rH^2(\SL(A), \rK_2(A))$ defined by the Steinberg extension.

\begin{Remark}
{\rm A very similar construction is described in \cite[\S 15]{FoGo}.}
\end{Remark}

\subsubsection{} Recall that there is a group homomorphism 
\[
d\log: \rK_2(A)\to \Omega^2_A:=\Omega^2_{A/\Z}; \quad \{f, g\}\mapsto d\log(f)\wedge d\log(g)
\]
where $\{f, g\}$ is the Steinberg symbol of $f$, $g\in A^*$. Here, under our assumption that $A$ is local, $\rK_2(A)$ is generated by such symbols (\cite{Stein}). Since
$
d(f^{-1}df)=0
$, the image of $d\log$ lies in the subgroup of closed $2$-forms.
We denote by 
\[
\omega_\rho\in \rH^2(\Gamma, \Omega^2_{A})
\]
the image of $\kappa_\rho$ under the map induced by $d\log$.

\subsubsection{} \label{sss:def}
If ${\rm H}_2(\Gamma, \Z)\simeq \Z$ and $[a]\in {\rm H}_2(\Gamma, \Z)$ is a generator, we set
\[
\kappa_{[a], \rho}:=[a]\cap \kappa_\rho\in \rK_2(A).
\]
Alternatively, we can obtain this class by evaluating the homomorphism 
\[
{\rm H}_2(\Gamma, \Z)\xrightarrow{\rH_2(\rho)} \rH_2(\SL(A), \Z)=\rK_2(A)
\]
at $[a]$. We can now set
\[
\omega_{[a], \rho}:=d\log(\kappa_{[a], \rho}) \in \Omega^2_A.
\]
This construction applies, in particular, when $\Gamma$ is the fundamental group of a closed surface. Note that by its construction, the $2$-form $\omega_{[a], \rho}$ is closed. When the choice of $[a]$ is understood, we will omit it from the notation.

\subsection{The tangent of the Steinberg  extension}
Suppose now that $R$ is a local ring in which $2$ is invertible. Let $V$ be a finite free $R$-module of rank $n$. Let us consider the (local) $R$-algebra 
\[
A=R\times V={\mathrm {Sym}}^\bullet_R(V)/\mathfrak M^2
\]
with multiplication $(r, v)\cdot (r', v')=(rr', rv'+r'v)$. Set 
\[
(r, v)_0=r.
\]
Notice that 
\[
\Omega^2_{A/R}\cong \wedge^2 V
\]
 by $d(0, v)\wedge d(0, v')\mapsto v\wedge v'$ and so we have a group homomorphism
\[
\iota: \rK_2(A)\to \rK_2(R)\times \wedge^2V; \quad \{f, g\}\mapsto (\{  f_0,  g_0  \},  d\log(\{f, g\})).
\]

We can write
\[
\SL(A)\cong \SL(R)\ltimes \rM^0(V),
\]
where $\rM^0(V)= \varinjlim_{m }\rM^0_{m \times m}(V)$. Here, $\rM^0_{m\times m} $ denotes the $m\times m$ matrices with trace zero. In the above, $g=\gamma(1+m)\in \SL(A)$ maps to $(\gamma, m)$ 
 and the semi-direct product is for the action of $\SL(R)$ on $\rM^0(V)$ by conjugation. 
 
Define 
\[
{\rm Tr}_{\rm alt}: \rM^0(V)\times \rM^0(V)\to \wedge^2V
\]
 by
\[
{\rm Tr}_{\rm alt}(X, Y)=\frac{1}{2}({\rm Tr}(X\otimes Y)- {\rm Tr}(Y\otimes X))\in \wedge^2 V.
\]
Here, $X\otimes Y$ denotes the square matrix with entries in $V\otimes V$ which is obtained from $X$ and $Y$ (which have entries in $V$)
by replacing in the formula for the product of matrices the multiplication by the symbol $\otimes$.

 \quash{
Consider the Cartesian product
\[
S(A)=\St(R)\times \rM^0(V)\times \wedge^2V,
\]
on which we  define the operation
\[
(\gamma, m, \omega)\cdot (\gamma', m', \omega')=(\gamma\gamma',  \gamma'^{-1}m\gamma'+m', \omega+\omega'+{\rm Tr}_{\rm alt}(\gamma'^{-1}m\gamma', m')).
\]
We can see that this makes $S(A)$ into a group and that there is surjective group homomorphism
\[
S(A)\to \SL(A); \quad (\gamma, m, \omega)\mapsto (\gamma, m),
\]
whose kernel  is the central subgroup  $\rK_2(R)\times  \wedge^2V$ of $S(A)$.
Notice that 
\[
\Omega^2_{A/R}\cong \wedge^2 V
\]
 by $d(0, v)\wedge d(0, v')\mapsto v\wedge v'$ and so we have a group homomorphism
\[
\iota: \rK_2(A)\to \rK_2(R)\times \wedge^2V; \quad \{f, g\}\mapsto (\{  f_0,  g_0  \},  d\log(\{f, g\})).
\]
The universal property of the Steinberg extension, implies that there is  a unique group homomorphism 
\[
\tilde \iota: \St(A)\xrightarrow{\  } S(A)
\]
which extends $\iota$ and which lifts the identity on $\SL(A)$. }

\begin{prop}\label{vanderKallen} Consider the Cartesian product
\[
S(A)=\St(R)\times \rM^0(V)\times \wedge^2V,
\]
on which we  define the operation
\[
(\gamma, m, \omega)\cdot (\gamma', m', \omega')=(\gamma\gamma',  \gamma'^{-1}m\gamma'+m', \omega+\omega'+{\rm Tr}_{\rm alt}(\gamma'^{-1}m\gamma', m')).
\]

a) This operation makes $S(A)$ into a group and there is a surjective group homomorphism
\[
S(A)\to \SL(A); \quad (\gamma, m, \omega)\mapsto (\gamma, m),
\]
whose kernel  is the central subgroup  $\rK_2(R)\times  \wedge^2V$ of $S(A)$.

b) There is  a unique group homomorphism 
\[
\tilde \iota: \St(A)\xrightarrow{\  } S(A)
\]
which extends $\iota$ and which lifts the identity on $\SL(A)$. 

c) Assume in addition $\Omega^1_{R/\Z}=0$.
Then $\iota$ and $\tilde\iota$ are isomorphisms
\[
\iota: \rK_2(A)\xrightarrow{\ \sim\ }  \rK_2(R)\times \wedge^2V,\qquad \tilde\iota:  \St(A)\xrightarrow{\ \sim\ } S(A).
\]
 \end{prop}

\begin{proof} Part (a) is obtained by a straightforward calculation. Part (b) follows from (a) and the universal property of the Steinberg extension. 
Using the same universal property of the Steinberg extension, we see that to show (c)  is enough to show that $\iota: \rK_2(A)\to \rK_2(R)\times \wedge^2V$ is an isomorphism, assuming $\Omega^1_{R/\Z}=0$.
By \cite{vanderKallen} there is a functorial isomomorphism
\[
\partial: \rK_2(R[\varepsilon])\xrightarrow{\sim} \rK_2(R)\times \Omega^1_{R },
\]
where $R[\epsilon]=R[x]/(x^2)$ is the ring of dual numbers and $\epsilon=x\,{\rm mod}\, (x)^2$. Since $R$ is local, we can represent elements of $\rK_2(R)$ and $\rK_2(R[\epsilon])$ by Steinberg symbols. Then, the isomorphism is given using 
\[
d\log: \rK_2(R[\epsilon])\to \Omega^2_{R[\epsilon]}=\Omega^2_R\oplus \ep\Omega^2_R\oplus d\ep\wedge \Omega^1_R.
\]
Indeed, we have (see \cite[2.3]{Osipov}, or \cite{BlochTangent})
\[
\partial(\{f,  g\})= (\{f_0, g_0\}, (d\log)_2\{f,  g\} )
\]
where $d\ep\wedge (d\log)_2\{f,  g\}$ is the projection of $d\log(\{f, g\})$ on the last component above.  In particular, 
\[
d\ep\wedge (d\log)_2\{1+sr\ep, r\}=d\epsilon\wedge sdr, \qquad  \partial(\{1+sr\ep,  r\})=(0, sdr).
\]
 Suppose that $n=\rk_R V=1$, so $A\simeq R[\ep]$. Since $\Omega^1_{R/\Z}=(0)$,  $\wedge^2V=(0)$ and, by the above, $\iota$ is an isomorphism (both sides are $\rK_2(R)$). We now argue by induction on $n$. Set $A'=R\times V'$, with $\rk_R V'=n-1$ and basis $v'_1, \ldots , v'_{n-1}$, so that $A$ is a quotient of $A'[\ep]$ by $\ep\cdot  v'_i=0$. We have $\wedge^2V=\wedge^2(R\cdot \ep\oplus V')=\wedge^2V'\oplus (\ep\wedge V')$ and by the induction hypothesis
 \[
 \rK_2(A')\simeq \rK_2(R)\times \wedge^2 V', \quad \hbox{\rm so,}
 \]
\[
\rK_2(A'[\ep])=\rK_2(A')\times d\ep\wedge \Omega^1_{A'/\Z}\simeq \rK_2(R)\times \wedge^2V'\times d\ep \wedge V'\simeq \rK_2(R)\times \wedge^2V.
\]
Since $A'[\ep]^\times\to A^\times$ is surjective
and $\rK_2(A)$ is generated by Steinberg symbols, the group $\rK_2(A)$ is a quotient of $\rK_2(A'[\ep])$
and the composition
\[
\rK_2(A'[\ep])\to \rK_2(A)\xrightarrow{\iota}  \rK_2(R)\times \wedge^2V
\]
is the isomorphism above.
The claim that $\iota$ 
is an isomorphism follows.
\end{proof}

\subsection{Tangent space and pairings}
Suppose now that $A=R\times V$ is as in the previous paragraph and that
\[
\rho: \Gamma\to \SL(A)=\SL(R)\ltimes \rM^0(V)
\]
is a representation that lifts $\rho_0: \Gamma\to  \SL(R)$. Then we can write
\[
\rho(\gamma)=\rho_0(\gamma)(1+c(\gamma))
\]
where $c: \Gamma\to {\rm M}^0(V)$ is a $1$-cocycle with ${\rm M}^0(V)$
carrying the adjoint action $\gamma \cdot M=\rho_0(\gamma)^{-1}M\rho_0(\gamma) $.
We can consider the cup product
\[
c\cup c\in \rH^2(\Gamma, {\rm M}^0(V)\otimes_R {\rm M}^0(V) ).
\]
This is given by the $2$-cocycle
\[
(c\cup c)(\gamma_1, \gamma_2)=\rho_0(\gamma_2)^{-1}c(\gamma_1)\rho_0(\gamma_2)\otimes  c(\gamma_2 ).
\]
Applying   the map  $\rH^2(\Gamma, {\rm Tr}_{\rm alt})$ induced by ${\rm Tr}_{\rm alt}: {\rm M}^0(V)\times {\rm M}^0(V)\to \wedge^2 V$
gives  
\[
{\rm Tr}_{\rm alt}(c\cup c)\in \rH^2(\Gamma, \wedge^2 V).
\]
The following Proposition, in conjuction with \S \ref{poincaredual} below, shows that the form $\omega_\rho$
agrees, under some conditions, with a more standard construction
which uses cup product and duality.
 
\begin{prop}\label{omegaalt}
We have \[
\omega_\rho={\rm Tr}_{\rm alt}(c\cup c)
\]
in $\rH^2(\Gamma, \Omega^2_A)=\rH^2(\Gamma, \wedge^2 V)$.
\end{prop}

\begin{proof} 
Since $\tilde\iota: {\rm St}(A)\xrightarrow{\sim} S(A)$, we can calculate $\omega_\rho$ using the extension $S(A)$.
We can first calculate 
\[
\kappa'_\rho(\gamma_1,\gamma_2)=s(\rho(\gamma_1\gamma_2))s(\rho(\gamma_2))^{-1}s(\rho(\gamma_1))^{-1}
\]
 by using the lifts:
\[
s(\rho(\gamma))=(s(\rho_0(\gamma)), c(\gamma), 0)\in \St(R)\times {\rm M}^0(V)\times \wedge^2V=S(A).
\]
A straightfoward calculation using the group operation on $S(A)$ gives
\[
\kappa'_\rho(\gamma_1,\gamma_2)=(\kappa'_{\rho_0}(\gamma_1,\gamma_2), 0, {\rm Tr}_{\rm alt}(\rho_0(\gamma_2)^{-1}c(\gamma_1)\rho_0(\gamma_2), c(\gamma_2 )).
\]
The cohomology class of $\kappa_\rho$  maps to the one of $\kappa'_\rho$
in $\rH^2(\Gamma, \rK_2(R)\times \wedge^2V)$ under the map given by $\iota$. 
Hence, 
\[
\omega_\rho(\gamma_1, \gamma_2)={\rm Tr}_{\rm alt}(\rho_0(\gamma_2)^{-1}c(\gamma_1)\rho_0(\gamma_2), c(\gamma_2 ))={\rm Tr}_{\rm alt}((c\cup c)(\gamma_1, \gamma_2)).
\]
in cohomology.
\end{proof}

\subsection{The $2$-form and duality} \label{poincaredual}

Suppose  that $\Gamma$ satisfies Poincare duality in dimension $2$ ``over $R$" 
in the following sense: 

\begin{itemize}
\item[i)] There is an isomorphism ${\rm tr}: \rH^2(\Gamma, R)\simeq R$. 

\item[ii)] For any $\Gamma$-module $W$ which is a finite free  $R$-module,  
$\rH^{i}(\Gamma, W)$ is a finite free $R$-module which is trivial unless $i=0,1,2$.

\item[iii)] The cup product pairing
\[
\rH^{i}(\Gamma, W)\times \rH^{2-i}(\Gamma, W^\vee)\to \rH^2(\Gamma, W\otimes_R W^\vee)\to \rH^2(\Gamma, R)\xrightarrow{\rm tr} R
\]
is a perfect $R$-bilinear pairing. (Here, $W^\vee={\rm Hom}_R(W, R)$.)
\end{itemize}

Consider $\rho_0: \Gamma\to \SL_n(R)$ and apply the above to $W={\rm Ad}^0_{\rho}=\rM^0_{n\times n}(R)$. The trace form $(X, Y)\mapsto {\rm Tr}(XY)$ gives an $R$-linear map
\[
W\to W^\vee.
\]
(This is an isomorphism when $n$ is invertible in $R$. We then use this to identify $W$ with  $W^\vee$.)
Combining with the above we obtain
\[
\langle\  ,\  \rangle: \rH^{1}(\Gamma, W)\times \rH^{1}(\Gamma, W )\to   R.
\]
(If $n$ is invertible in $R$ this is a perfect pairing.)

Suppose that $c_1$ and $c_2$ are two $1$-cocycles of $\Gamma$ in $W$ that correspond to lifts of $\rho_0$ to representations $\rho_1$ and $\rho_2$ with values in $R[\ep]$. Recall that there is a natural isomorphism between 
the tangent space of the functor of deformations of $\rho_0$ to Artin local $R$-algebras and the cohomology group $\rH^{1}(\Gamma, W)$ (cf. \cite{Mazur} \S 21). Set $V^\vee= \rH^{1}(\Gamma, W)$ which is a finite free $R$-module and denote by 
$\rho: \Gamma\to \SL_n(A)$, with $A=R\times V$ the universal first-order deformation of $\rho_0$. Then, $\rho_i$, $i=1, 2$, correspond to $v_i\in V^\vee$
and $\rho_i$ is given by specializing $A=R\times V\to R[\ep]=R\times R\ep$ with $V\to R\ep$ given by $v_i\in V^\vee$. 

We can consider $v_1\wedge v_2\in \wedge^2V^\vee=(\wedge^2 V)^\vee$.
From the above description, it follows that 
\[
\langle c_1, c_2\rangle=(v_1\wedge v_2) ({\rm Tr}_{\rm alt}(c\cup c)).
\]
Therefore, by Proposition \ref{omegaalt},
\[
\langle c_1, c_2\rangle=(v_1\wedge v_2)(\omega_\rho).
\]
This translates to
\begin{equation}\label{Goldman1}
\langle c_1, c_2\rangle=\omega_\rho(c_1, c_2),
\end{equation}
in which we think of $(c_1, c_2)$ as a pair of tangent vectors.

\begin{Remark} {\rm
The equality (\ref{Goldman1}) implies that the form $\omega_\rho$ can be used to recover the 
symplectic form on the $\SL_n$-character varieties of   fundamental groups of closed surfaces constructed by Goldman \cite{Goldman}. Since $\omega_\rho$ is visibly closed, this gives a direct and completely algebraic argument for the closedness of Goldman's form. This approach is also suggested in \cite{FoGo}. There an identity like (\ref{Goldman1}) for $\Gamma$ the fundamental group of a surface is explained by topological means.}
\end{Remark}
  
\subsection{$\rK_2$ invariants and $2$-forms; profinite groups}
Let $\Gamma$ be a profinite group and $A$ a complete local Noetherian ring with finite residue field $\BF$ of characteristic $\ell$ and maximal ideal $\frakm$. 
We will view $A$ as a $W=W(\BF)$-algebra where $W(\BF)$ is the ring of Witt vectors. The ring $A$ carries the natural profinite $\frakm$-adic topology which induces a  profinite topology on $\GL_d(A)$, $\SL_d(A)$. There are ``continuous variants" of the constructions of the previous paragraphs for continuous representations
\[
\rho: \Gamma\to \SL_d(A)\subset \SL(A).
\]
For $n\geq 1$, set $A_n=A/\frak m^n$.

\subsubsection{} 
Our constructions give classes
\[
\hat\kappa_\rho=(\kappa_{\rho, n})_n\in \varprojlim\nolimits_n \rH^2(\Gamma, \rK_2(A_n)),\quad 
\hat\omega_\rho=(\omega_{\rho, n})_n\in \varprojlim\nolimits_n \rH^2(\Gamma, \Omega^2_{A_n/W}).
\]

Now let
\[
\rK_2^{\rm ct}(A):=\varprojlim\nolimits_n \rK_2(A_n), \quad 
\hat\Omega^1_{A/W}=\varprojlim\nolimits_n \Omega^1_{A_n/W}.
\]
There is a continuous  map
\[
d\log: \rK_2^{\rm ct}(A)\to \hat\Omega^2_{A/W}
\]
 obtained as the inverse limit of $d\log: \rK_2(A_n)\to \Omega^2_{A_n/W}$.

 \medskip

\section{Chern-Simons and volume}

In this section we give the main algebraic construction.
We first assume that $\Gamma$ is a discrete group. This case is less technical but still contains the main idea. The construction depends on the suitable choice of a $3$-cocycle.
The profinite case (for $\ell$-adic coefficients) is explained later; in this case, we show that such a $3$-cocycle can be given using the $\ell$-adic Borel regulator.

\subsection{The Chern-Simons torsor}

Until further notice, $\Gamma$ is a discrete group and $\rho:\Gamma\to \GL_d(A)$ is a homomorphism, $d\geq 2$.
Also, in what follows, we always assume
\smallskip

\begin{itemize}
\item[ (H)] $\rH_2(\Gamma, \Z)\simeq \Z$ and $\rH_3(\Gamma, \Z)=0$.
\end{itemize}
\smallskip

Let $C_\bullet(\Gamma)$, $C_\bullet(\GL_d(A))$, be the bar resolutions. We may regard $C_j(\GL_d(A))$ as $\Z[\Gamma]$-modules
using $\rho$ and obtain a morphism of complexes
\[
\rho: C_\bullet(\Gamma)\to C_\bullet(\GL_d(A)).
\]
This gives
\[
\rho: \bar C_{3,2}(\Gamma)\to \bar C_{3,2}(\GL_d(A))
\]
where $\bar C_{3,2}$ is as defined in \S \ref{ss:homology}.

For simplicity, set 
\[
D(\Gamma):=\bar C_{3}(\Gamma)/{\rm Im}(\partial_4),\quad \hbox{\rm and}\quad 
D(A):=\bar C_{3}(\GL_d(A))/{\rm Im}(\partial_4).
\] 
Note that $D(\Gamma)$ acts on $Z_2(\Gamma)\times D(A)$
 by 
 \[
 d+ (c, v)= (c+\partial_3(d), v+\rho(d)).
 \]
 Now define the $D(A)$-torsor $\calT_\rho$ of ``global sections" (cf. \cite{Freed}, \cite{FreedQuinn}):
 
 \begin{Definition} We set $\calT_\rho$ to be the set of group homomorphisms $T: Z_2(\Gamma)\to D(A)$
 which are 
 $D(\Gamma)$-equivariant, i.e. satisfy
 \[
 T(c+\partial_3(d))=T(c)+\rho(d).
 \]
\end{Definition}
Alternatively, since 
\[
\partial_3: D(\Gamma)=\bar C_3(\Gamma)/{\rm Im}(\partial_4)\hookrightarrow Z_2(\Gamma),
\]
is injective, the set $\calT_\rho$ can be described as
the set of homomorphic extensions
of $ \rho: D(\Gamma)\to D(A)$ to $Z_2(\Gamma)\to D(A)$, or as the set of splittings of the extension
\[
0\to D(A)\to E\to \rH_2(\Gamma, \Z)\to 0
\]
obtained by pushing out $0\to D (\Gamma)\to Z_2(\Gamma)\to \rH_2(\Gamma, \Z)\to 0$ by $\rho: D(\Gamma)\to D(A)$.
This set is non-empty and hence a $D(A)$-torsor, since $\rH_2(\Gamma, \Z)\simeq \Z$. 

\subsubsection{} Suppose in the above construction, we replace $\rho:\Gamma\to \GL(A)$ by 
${\rm inn}_h\cdot\rho$, for some $h\in \GL(A)$. We obtain a new torsor $\calT_{{\rm inn}_h\cdot\rho}$ defined using the $D(\Gamma)$-action
on $D(A)$ by
\[
v+'d=v+h\rho(d)h^{-1}.
\]
Observe that, for $c'-c=\partial_3(d)$, we have
\[
h\rho (d)h^{-1}=\rho(d)+F_h(\rho(c'-c))=\rho(d)+F_h(\rho(c'))-F_h(\rho(c)),
\]
so
\[
h\rho (d)h^{-1}+v-F_h(\rho (c'))=\rho(d)+v-F_h(\rho(c)).
\]
The last identity shows that 
\[
T\mapsto T'=T+ F_h\cdot\rho.
\]
gives a $D(A)$-equivariant bijection $\calT_\rho\xrightarrow{\sim} \calT_{{\rm inn}_h\cdot \rho}$.
\medskip

Even though it would be possible to formulate the constructions that follow in terms of the torsor $\calT_\rho$,
we choose a more concrete treatment that uses group co/cycles.

\subsection{Cocycles and cohomology classes} Suppose we are given a representation $\rho:\Gamma\to \GL_d(A)$, $d\geq 2$, and a group $G$ together with two homomorphisms $\psi: G\to {\rm Out}(\Gamma)$, $\phi: G\to {\rm Aut}_{\O-{\rm alg}}(A)$. We assume the following condition:

\begin{itemize}
\item[(E)] For each $\sigma\in G$, the  representation $\rho^\sigma$ given by 
\[
\rho^{\sigma}(\gamma)=\phi(\sigma)^{-1}(\rho(\widetilde{\psi(\sigma)} ( \gamma)))
\]
 is equivalent to $\rho$. Here, we denote by $\widetilde{\psi(\sigma)}$ some automorphism of $\Gamma$ that lifts   $\psi(\sigma)$.  
\end{itemize}

In what follows, we omit the notation of $\phi$ and $\psi$ for simplicity. 
We write $\sigma(a)$ instead of $\phi(\sigma)(a)$ and 
also write $\sigm$ for an automorphism of $\Gamma$ that lifts $\psi(\sigma)$. Then,
equivalently, the condition (E) amounts to:

\begin{itemize}
\item[(E')] For each $\sigma\in G$, there is $h_\sigm\in \GL_d(A)$ such that  
 \begin{equation}\label{sh}
 \rho(\ti\sigma\cdot \gamma)=h_{\ti \sigma}\cdot \sigma(\rho(\gamma))\cdot h_{\ti \sigma}^{-1}
 \end{equation}
 for all $\gamma\in \Gamma$.
 \end{itemize}
 
 Let  $\mathfrak Z_A(\rho)$ be the centralizer of the image ${\rm Im}(\rho)\subset \GL_d(A)$.
 The image of $h_{\ti \sigma}$ in ${\rm  GL}_d(A)/\mathfrak Z_A(\rho)$
 is uniquely determined by the automorphisms $\ti\sigma$, $\sigma$ and by $\rho$.
 
 \def\be{{\mathfrak r}}
 
\subsubsection{} Suppose that $B(A)$ is a $G$-module which supports a $G$-equivariant homomorphism
\[
\be=\be_A: D(A)=\bar C_3(\GL_d(A))/{\rm Im}(\partial_4)\to B(A)
\]
such that 
\[
\be (F_h(u))=0,
\]
for all $h\in \mathfrak Z_A(\rho)$, $u\in Z_2({\rm Im}(\rho))$. 

As explained in \S \ref{shufflepar},  for $u\in Z_2({\rm Im}(\rho))$, $h\in \mathfrak Z_A(\rho)$, the homotopy property 
gives
$
\partial_3 F_h(u)=0
$
and so $F_h$ gives 
\[
\nabla_{h, {\rm Im}(\rho)}=[F_h]: \rH_2({\rm Im}(\rho))\to \rH_3(\GL_d(A)).
\]
The condition 
$\be (F_h(u))=0$ is equivalent to:
\smallskip

\begin{itemize}
\item[(V)]  For $\be: \rH_3(\GL_d(A))\to B(A)$, we have $\be\cdot \nabla_{h, {\rm Im}(\rho)}=0$, for all $h\in \mathfrak Z_A(\rho)$.
\end{itemize}
\smallskip

Let $\rH^{\rm dec}_3(\GL_d(A))$ be the subgroup of $\rH_3(\GL_d(A))$
generated by the images of 
\[
\nabla_{h, C}: \rH_1(\langle h\rangle)\otimes_\Z \rH_2(C)\to \rH_3(\GL_d(A)),
\]
where $C$ runs over all subgroups of $\GL_d(A)$ and $h$ all elements centralizing $C$. 
For (V) to be satisfied for all $\rho$, it is enough to have
\smallskip

 \begin{itemize}
\item[(V')]   $\be: \rH_3(\GL_d(A))\to B(A)$ vanishes on $\rH^{\rm dec}_3(\GL_d(A))$.
\end{itemize}
 
\subsubsection{}  Choose, once and for all, a generator  of $\rH_2(\Gamma,\Z)$. Pick $c\in Z_2(\Gamma)$ with $[c]=1$ in $\rH_2(\Gamma, \Z)\simeq \Z$. Suppose that $\ti\sigma$ acts on $\rH_2(\Gamma, \Z)\simeq \Z$ via multiplication by $a_\sigma\in \Z^\times$. (This number only depends on $\sigma$ through $\psi(\sigma)\in {\rm Out}(\Gamma)$.)
Set  
\[
A_{\ti \sigma}(c)=  {\rho}(d(\sigm, c))-a_\sigma^{-1}  {F_{h_{\ti \sigma}}}(\sigma(\rho(c)))\in D(A).
\]
Here, $d(\sigm, c)$ is the (unique) element in $D(\Gamma)=\bar C_3(\Gamma)/{\rm Im}(\partial_4)$ with 
\[
\partial_3(d(\sigm, c))=a^{-1}_\sigma\ti\sigma(c)-c.
\]
In what follows, we will often omit the inclusion $\partial_3: D(\Gamma)\hookrightarrow Z_2(\Gamma)$ and write $a^{-1}_\sigma\ti\sigma(c)-c$ instead of $d(\ti\sigma, c)$ to ease the notation.

Assume that $\be: D(A)=\bar C_3(\GL_d(A))/{\rm Im}(\partial_4)\to B(A)$   satisfies condition (V).

\begin{lemma}\label{ind334}
The element $\be A_{\ti\sigma}(c)$ of $B(A)$ does not depend on the choices of $\sigm$ and $h_{\sigm}$.
\end{lemma}

\begin{proof}
First we check that $\be A_{\ti\sigma}(c)$ is independent of the choice of $h_\sigm$. If $h'_\sigm$ is another choice, then $h'_\sigm=h_\sigm\cdot z$ with $z\in \mathfrak Z_A(\rho)$. Notice that for $u=\sigma(\rho(c))\in Z_2({\rm Im}(\rho))$ we have
\[
F_{h'_\sigm}(u)=F_{h_\sigm\cdot z}(u)=F_{ z}(u)+F_{h_\sigm}({\rm Inn}_z (u))=F_{ z}(u)+F_{h_\sigm}( u).
\]
Hence, by condition (V), $\be F_{h'_\sigm}(u)=\be F_{h_\sigm}( u)$ and the result follows. Next, we show that
$A_\sigm(c)$ is actually independent of the choice of the automorphism $\sigm$ lifting $\psi(\sigma)$.
Suppose we replace $\ti\sigma$ by another choice $\ti\sigma'={\rm Inn}_{\delta}\cdot \ti\sigma $ lifting $\psi(\sigma)$
 and we take
 \[
 h_{{\rm Inn}_\delta\cdot\ti\sigma}=\rho(\delta)h_{\ti \sigma}.
 \]
 Then we have
 \begin{align*}
 a_\sigma A_{\ti\sigma'}(c)=&\rho( \delta\ti\sigma(c)\delta^{-1}-a_\sigma c)- F_{\rho(\delta)h_{\ti\sigma}}(\sigma\rho(c))\\
  =&\rho([ \delta \ti \sigma(c) \delta ^{-1}   - \ti\sigma(c)]+[ \ti\sigma(c)-a_\sigma c])-F_{h_{\ti \sigma}}(\sigma(\rho(c))-F_{\rho(\delta )}(h_{\ti \sigma} \sigma\rho(c) h_{\ti\sigma}^{-1})\\
=&a_\sigma A_{\ti\sigma}(c)+\rho([ \delta \ti \sigma(c) \delta ^{-1}   - \ti\sigma(c)])-F_{\rho(\delta)}(h_{\ti \sigma} \sigma\rho(c) h_{\ti \sigma}^{-1}).
 \end{align*}
But 
 \begin{align*}
\rho([ \delta \ti\sigma(c) \delta ^{-1}   - \ti\sigma(c)])=&\rho(F_{ \delta }(\ti\sigma(c))\\
=& F_{ \rho(\delta) }(\rho(\ti\sigma(c)))\\
=&F_{ \rho(\delta) }(h_{\ti\sigma}  \sigma \rho(c) h_{\ti\sigma}^{-1}).
 \end{align*}
So 
\[
a_\sigma A_{\ti\sigma'}(c)=a_\sigma A_{\ti\sigma}(c).
 \]
 We used  $h_{\sigm'}=\rho(\delta)h_{\ti \sigma}$ for this but now by applying the independence
 of that choice that we shown before, we see that
 $\be A_{\ti\sigma}$ does not depend on the choice of the lift $\ti\sigma\in {\rm Aut}(\Gamma)$
 or of $h_\sigm$.
\end{proof}

In view of Lemma \ref{ind334} it makes sense
to set
\[
\CS_\sigma(c)=\be A_\sigm(c)\in B(A).
\]
 We denote $B(A)(-1)$  the $G$-module $B(A)$ with twisted action:
$\sigma\in G$ acts by $a_\sigma^{-1}\cdot \sigma$.

\begin{prop} Assume   $\be: D(A)\to B(A)$  satisfies condition (V). The map $G\to B(A)(-1)$ given by $\sigma\mapsto \CS_\sigma(c)=\be A_\sigm(c)$ is a $1$-cocycle. Its cohomology class
 \[
 \CS_{\rho,\psi,\phi}\in \rH^1(G, B(A)(-1))
 \]
  is independent of the choice of $c\in Z_2(\Gamma)$ with $[c]=1$.
 The class $\CS_{\rho,\psi,\phi}$ depends only on $\psi$, $\phi$,  and the equivalence class of the representation $\rho$.
\end{prop}

\begin{proof} In the proof below some of the identities  are true in $D(A)$ before applying $\be$.  
However, eventually, the argument uses the independence given by Lemma \ref{ind334} which needs $\be$ to be applied. 
\smallskip

1) Suppose   $c'=c+\partial(d)$. Then  
 \begin{align*}
A_{\ti\sigma}(c')=\, & a_\sigma^{-1}(\rho(\ti\sigma c-a_\sigma c+\ti\sigma d-a_\sigma d))-a_\sigma^{-1}F_{h_{\ti\sigma}} ( \sigma(\rho(c)+\partial\rho(d))\\
=\, & A_{\ti\sigma}(c)+a_\sigma^{-1}\rho(\ti\sigma d)-(\rho d)-a_\sigma^{-1}F_{h_{\ti\sigma}}(\partial\sigma\rho d)\\
=\, & A_{\ti\sigma}(c)+a_\sigma^{-1}[h_{\ti\sigma}\sigma\rho(d)h_{\ti \sigma}^{-1}-F_{h_{\ti \sigma}}(\partial\sigma\rho d)]-(\rho d)\\
=\, & A_{\ti\sigma}(c)+(a_\sigma^{-1}\sigma-1)(\rho d).
\end{align*}
The last equality follows from
\[
h_{\ti\sigma} \sigma  \rho(d)  h^{-1}_{\ti \sigma}
-\sigma\rho(d)=F_{h_{\ti \sigma}}(\partial\sigma\rho(d)).
\]
This implies the independence after we show that $\sigma\mapsto \be A_\sigma(c)$ is a $1$-cocycle. 
\smallskip

2) We will now check the (twisted) cocycle condition 
 \[
 \be A_{\sigma\tau}=\be A_{\sigma}+a_\sigma^{-1}\sigma(\be A_\tau).
 \]
 (We omit $c$ from the notation). In view of Lemma \ref{ind334} we are free to calculate 
 using the lift $\ti\sigma \ti\tau$
 of the outer automorphism $\psi(\sigma\tau)$ and taking $h_{\ti\sigma\ti\tau}$ to be equal to $h_\sigm \sigma(h_{\ti\tau})$. Indeed, we have
  \begin{align*}
 \rho((\sigm \ti\tau) (\gamma))=& \rho(\sigm(\ti\tau (\gamma))\\
 = & h_{\sigm} \sigma(\rho(\ti\tau(\gamma))) h_{\sigma}^{-1}\\
 = & h_\sigm \sigma(h_{\ti\tau}) ( \sigma \tau) (\rho(\gamma)) (h_\sigm \sigma(h_{\ti\tau}))^{-1}.
 \end{align*}
 It is notationally simpler to work with 
 $B_\sigm=a_\sigma A_\sigm$.
 Write
 \[
 B_{\sigm\ti\tau}= \rho(\sigm\ti\tau(c)-a_{\sigma\tau}c)-F_{h_{\sigm\ti\tau}}(\sigma\tau\rho(c)).
 \]
 Now 
 \[
 \sigm\ti\tau(c)-a_{\sigma\tau}c=[\sigm\ti\tau(c)-a_{\tau}\sigm(c)]+[a_\tau\sigm(c)-a_{\sigma}a_\tau c]
 \]
 in $D(\Gamma)$. Hence, 
 \begin{align*}
 \rho(\sigm\ti\tau(c)-a_{\sigma\tau}c) = & \rho(\sigm(\ti\tau(c)-a_\tau c))+
 a_\tau \rho(\sigm(c)-a_\sigma c)\\
 =& \rho(\sigm(\ti\tau(c)-a_\tau c))+a_\tau B_{\sigm}(c) +a_\tau F_{h_\sigm}(\sigma\rho(c))
  \end{align*}
 Now
  \begin{align*}
 \rho(\sigm(\ti\tau(c)-a_{\tau}c))=\, & h_\sigm\cdot \sigma\rho(\ti\tau(c)-a_\tau c)\cdot h_\sigm^{-1}\\
=\, & \sigma\rho(\ti\tau c-a_{\tau} c)+F_{h_\sigm}(\sigma\rho(\ti\tau c-a_\tau c)).
 \end{align*}
 So,
 \[
   \rho(\sigm(\ti\tau(c)-a_{\tau}c)) =\sigma B_{\ti\tau}(c)+\sigma F_{h_{\ti\tau}}(\tau \rho(c))
  +F_{h_\sigm}(\sigma\rho(\ti\tau c-a_\tau c)).
  \]
  All together, we get
  \begin{align*}
  B_{\sigm\ti\tau}=\sigma B_{\ti\tau}(c)& +\sigma F_{h_{\ti\tau}}(\tau \rho(c))
  +F_{h_\sigm}(\sigma\rho(\ti\tau c-a_{\tau}c))+\\
  &+a_\tau B_{\sigm}(c) +a_{\tau}F_{h_\sigm}(\sigma\rho(c))-F_{h_{\sigm\ti \tau}}(\sigma\tau\rho(c)).
   \end{align*}
  Now
  \[
  F_{h_\sigm}(\sigma\rho(\ti\tau c-a_\tau c))=F_{h_\sigm}(\sigma\rho(\ti\tau c))-a_\tau F_{h_\sigm}(\sigma \rho c).
  \]
  which gives
  \[
  B_{\sigm\ti\tau}=\sigma B_{\ti\tau}(c)+a_\tau B_\sigm(c)+\sigma F_{h_{\ti\tau}}(\tau \rho(c))
  +F_{h_\sigm}(\sigma\rho(\ti\tau c))-F_{h_{\sigm\ti\tau}}(\sigma\tau\rho(c))
   \]
So it  is enough to show the identity
\[
F_{h_{\sigm\ti\tau}}(\sigma\tau\rho(c))=\sigma F_{h_{\ti\tau}}(\tau \rho(c))
 +F_{h_\sigm}(\sigma\rho(\ti\tau c)).
\]
 We have 
 \[
 \sigma F_{h_{\ti\tau}}(\tau \rho(c))=F_{\sigma(h_{\ti\tau})}(\sigma\tau\rho(c)).
 \]
 \begin{align*}
 F_{h_\sigm}(\sigma\rho(\ti\tau c))=\, & F_{h_\sigm}(\sigma(h_{\ti\tau})(\sigma\tau)(\rho(c))\sigma(h_{\ti\tau})^{-1})\\
 = & F_{h_\sigm}({\rm inn}_{\sigma(h_{\ti\tau})}(\sigma\tau\rho(c))).
 \end{align*}
 Now apply
 \[
 F_{h_{\sigm\ti\tau}}= F_{h_\sigm\sigma(h_{\ti\tau})} 
 =F_{\sigm(h_{\ti\tau})}+F_{h_\sigm}\cdot {\rm inn}_{\sigma(h_{\ti\tau})}
 \]
 to conclude $B_{\sigm\ti\tau}(c)=\sigma B_{\ti\tau}(c)+a_\tau  B_\sigm(c)$.
 Since $A_\sigm:=a_\sigma^{-1}B_\sigm$, this gives
 \[
 A_{\sigm\ti\tau}=a_{\sigma\tau}^{-1}B_{\sigm\ti\tau}=a_\sigma^{-1} \sigma (a_\tau^{-1} B_{\ti\tau})+a_\sigma^{-1}B_\sigm=a_\sigma^{-1} \sigma(A_{\ti\tau})+A_\sigm
 \]
as desired.
\smallskip

3) It remains to show the independence up to equivalence of representations.
Suppose we change $\rho$ to $\rho'={\rm inn}_g\cdot\rho$, with $g\in \GL_d(A)$, but leave $\phi$ and $\psi$ the same. Then, 
\[
\rho'(\sigm\gamma)=g\rho(\sigm\gamma)g^{-1} =h'_\sigm \sigma(g\rho(\sigm\gamma)g^{-1})) {h'}_\sigm^{-1},
\]
so we can take
\[
h'_\sigm=g h_\sigm \sigma (g)^{-1}.
\]
Then
\begin{align*}
 B'_\sigm=\, &g\rho(\sigm(c)-a_\sigma c)g^{-1}- F_{gh_\sigm \sigma(g)^{-1}}(\sigma(g\rho g^{-1})(c))\\
=\, &\rho(\sigm(c)-a_{\sigma}c)+F_g(\rho(\sigm (c)-a_{\sigma}c))- F_{gh_\sigm\sigma(g)^{-1}}(\sigma(g\rho g^{-1})(c))\\
=\, & B_\sigm+F_{h_\sigm}(\sigma(\rho c))+F_g(\rho(\sigm(c)-a_\sigma c))-F_{gh_\sigm \sigma(g)^{-1}}(\sigma(g\rho g^{-1})(c)).
\end{align*}
Now
\begin{align*}
F_g(\rho(\sigm (c)-a_{\sigma}c))=\, &F_g(h_\sigm  \sigma\rho(c)h_\sigm ^{-1})-a_\sigma F_g(\rho(c)),\\
F_{gh_\sigm \sigma(g)^{-1}}(\sigma(g\rho g^{-1})(c))
=\, &F_{gh_\sigm}(\sigma\rho(c))+ F_{\sigma(g)^{-1}}(\sigma(g\rho g^{-1})(c))=\\
=\, &F_g(h_\sigm\sigma\rho(c)h_\sigm^{-1})+F_{h_\sigm}(\sigma\rho(c))-F_{\sigma(g)}(\sigma\rho(c)),
\end{align*}
(The last equality is true since $F_h\cdot {\rm inn}_{h^{-1}}+ F_{h^{-1}}=F_{h\cdot h^{-1}}=F_1=0$
gives
\[
F_{\sigma(g)^{-1}}(\sigma(g\rho g^{-1})(c))=-F_{\sigma(g)}(\sigma\rho(c)).)
\]
Combining, we obtain
\begin{align*}
B'_\sigm=\, & B_\sigm-a_\sigma F_g(\rho(c))+\sigma F_g(\rho(c)),\quad  \hbox{\rm or}, \\
A'_\sigm=\, & A_\sigm+(a_\sigma^{-1}\sigma-1) F_g(\rho(c)) 
\end{align*}
which shows that the cohomology class depends on $\rho$ only up to equivalence.
\end{proof}
 
\subsubsection{} Suppose   there is an exact sequence of groups
\begin{equation}\label{exact}
1\to \Gamma\to \Gamma_0\to G\to 1
\end{equation}
and we are given a representation $\rho: \Gamma\to \GL_d(A)$. We take $\psi: G\to {\rm Out}(\Gamma)$ to be the natural homomorphism
given by the sequence and $\phi={\rm id}$, i.e. $G$ to act trivially on $A$. 

Assume that $\rho$ extends to a representation
$\rho_0: \Gamma_0\to \GL_d(A)$. The condition (E) is then satisfied:
For $\sigma\in G$ let $\sigm$  be the automorphism of $\Gamma$ given by 
\[
\gamma\mapsto s(\sigma)\gamma s(\sigma)^{-1}
\]
where $s(\sigma)\in \Gamma_0$ is any lift of $\sigma$.
We have
\[
\rho(\sigm\cdot\gamma)=h_\sigm \rho(\gamma) h_\sigm^{-1}
\]
for $h_\sigm=\rho_0(s(\sigma))$. Then
\[
\be A_\sigm(c)=\be \rho_0(a_\sigma^{-1}\cdot s(\sigma)\cdot c\cdot s(\sigma)^{-1}-c)-a_\sigma^{-1}\be F_{\rho_0(s(\sigma))}(\rho(c))\in B(A)
\]
and the class $\CS_\rho=[\be A_\sigma(c)]\in \rH^1(G, B(A)(-1))$ is independent of choices.

\subsection{Volume and Chern-Simons; profinite case}\label{profiniteCS} 

We now assume that  $\Gamma$ is a profinite group. 
Suppose that $\Gamma$ is topologically finitely 
generated. Then there is system 
\[
\cdots \subset \Gamma_{n'}\subset \Gamma_n\subset \cdots \subset \Gamma_1=\Gamma,
\]
for $n|n'$, of characteristic subgroups of finite index which give a basis of open neighborhoods of the identity.
Indeed, (cf. \cite{Andre}), we can take 
\[
\Gamma_n=\bigcap_{\Delta\subset \Gamma| [\Gamma:\Delta]|n}\Delta
\]
 to be the intersection of the finite set of  open normal subgroups of $\Gamma$ of index dividing $n$. We then have
\[
{\rm Aut}(\Gamma)=\varprojlim\nolimits_n {\rm Aut}(\Gamma/\Gamma_n),\quad {\rm Out}(\Gamma)=\varprojlim\nolimits_n {\rm Out}(\Gamma/\Gamma_n).
\]

As before, let $A$ be a complete local Noetherian algebra
with maximal ideal $\frakm_A$ and finite residue field. 
 Suppose  $\rho: \Gamma\to  \GL_d(A)$ is a continuous representation, $d\geq 2$.
 There is an embedding
 \[
 \iota: \GL_d(A)/\mathfrak Z_A(\rho)\hookrightarrow \prod_{i=1}^r \GL_d(A),\quad g\mapsto (g\rho(\gamma_i)g^{-1})
 \]
where $(\gamma_i)_i$ are topological generators of $\Gamma$. The induced topology on $\GL_d(A)/\mathfrak Z_A(\rho)$ is independent of the choice of $\gamma_i$
and is equivalent to the quotient topology.

Suppose we are given another profinite group $G$ and continuous homomorphisms
$\psi: G\to {\rm Out}(\Gamma)$ and $\phi: G\to {\rm Aut}_{\O-{\rm alg}}(A)$. 
Here, ${\rm Aut}_{\O-{\rm alg}}(A)$ is equipped with the profinite topology for which the subgroups $K_n=\{f\ |\ f\equiv {\rm id}\,{\rm mod}\, \frakm_A^n\}$, $n\geq 1$,
is a system of open neighborhoods of the identity. Similarly, we equip ${\rm Out}(\Gamma)={\rm Aut}(\Gamma)/{\rm Inn}(\Gamma)$ with the quotient topology, obtained from the topology of  ${\rm Aut}(\Gamma)$ 
for which the subgroups of automorphisms trivial on $\Gamma_n$ gives a system of open 
neighborhoods of the identity. We assume that $\psi$ is represented by a continuous set-theoretic map
\[
\ti\psi: G\to {\rm Aut}(\Gamma), \ \hbox{\rm i.e.}\ \psi(\sigma)=\ti\psi(\sigma) {\rm Inn}(\Gamma),\ \forall \sigma.
\]
Now suppose that $\gamma\mapsto \sigma^{-1}(\rho(\sigm\gamma))$ is equivalent to $\rho$, for all $\sigma\in G$. Here, $\sigm=\ti\psi(\sigma)$ and so $\sigma\mapsto \sigm$ is continuous. We have
\[
\rho(\sigm\cdot\gamma)=h_\sigm \sigma(\rho(\gamma)) h_\sigm^{-1}
\]
for all $\gamma\in \Gamma$, where $[h_\sigm]\in \GL_d(A)/\mathfrak Z_A(\rho)$
is determined by $\rho$ and by $\sigma$ through $\ti\psi(\sigma)$ and $\phi(\sigma)$.

\begin{lemma}\label{lemmaCont} The map 
$
G\to \GL_d(A)/\mathfrak Z_A(\rho)$, given by $\sigma\mapsto [h_\sigm],
$
is continuous. 
\end{lemma}

\begin{proof} By the above, it is enough to show that 
the inverse image under $\iota$ of each open neighborhood $\prod_i V_i\subset \prod_i \GL_d(A)$
of $(\rho(\gamma_i))_i$ contains an open neighborhood of $1\cdot \mathfrak Z_A(\rho)$.  
It is enough to consider $V_i=\rho(\gamma_i) (1+{\rm M}_{d}(\frakm_A^n))$.
Pick $m$ such that $\rho(\Gamma_m)\equiv {\rm I}\, {\rm mod}\, \frakm_A^n$ and then choose an open normal subgroup of finite index $U\subset G$ such that $\sigma\in K_n$
and $\sigm$ is trivial on $\Gamma/\Gamma_m$. Then we have $\ti\sigma \gamma=\gamma\cdot \gamma'$,
with $\gamma'\in \Gamma_m$ and so
\[
\rho(\ti\sigma \gamma)=\rho(\gamma\gamma')=\rho(\gamma)\, {\rm mod}\, \frakm_A^n
\]
while 
$
\sigma(\rho(\gamma))\equiv \rho(\gamma)\, {\rm mod}\, \frakm_A^n.
$
We deduce that, for all $\sigma\in U$, we have
 \[
\sigma^{-1} \rho(\ti\sigma \gamma) \equiv \rho(\gamma)\, {\rm mod}\, \frakm_A^n
\]
for all $\gamma\in \Gamma$. Hence, for $\sigma\in U$, $[h_\sigm]\in \GL_d(A)/\mathfrak Z_A(\rho)$ belongs to $\prod_i V_i$. \end{proof}

\subsubsection{}
 Suppose we have a short exact sequence of continuous homomorphisms of profinite groups
\[
1\to \Gamma\to \Gamma'\to G\to 1
\]
 in which $\Gamma$ is topologically finitely generated. By \cite[I, \S 1, Prop. 1]{SerreGaloisCoh}  $\Gamma'\to G$ affords a continuous set-theoretic section $s: G\to \Gamma'$. Then the natural $\psi: G\to {\rm Out}(\Gamma)$ can be represented by the continuous map (not always a homomorphism) $G\to {\rm Aut}(\Gamma)$ given by $\sigma\mapsto (\gamma\mapsto s(\sigma)\gamma s(\sigma )^{-1})$. 

\subsubsection{} The constructions in the previous paragraph
can now be reproduced for continuous $\Z_\ell$-homology. Let us collect the various parts 
of the set-up.
Suppose  that we are given:

 \begin{itemize}
 \item A topologically finitely generated profinite group $\Gamma$ which is such that $\rH_2(\Gamma, \Z_\ell)\simeq \Z_\ell$, $\rH_3(\Gamma, \Z_\ell)=0$. 
 
 \item A continuous representation $\rho: \Gamma\to  \GL_d(A)$ with $d\geq 2$
 and $A$ a complete local Noetherian $\O$-algebra with finite residue field of characteristic $\ell$. 
 
 \item A profinite group $G$, a continuous homomorphism $\phi: G\to {\rm Aut}_{\O-{\rm alg}}(A)$ and a continuous map $\ti\psi: G\to {\rm Aut}(\Gamma)$ which induces a (continuous) homomorphism $\psi: G\to {\rm Out}(\Gamma)$.  Denote by $a: G\to \Z^\times_\ell$ the character which gives the action of $G$ on $\rH_2(\Gamma, \Z_\ell)\simeq \Z_\ell$, which is induced by $\psi(\sigma)$.
 
\item A  topological $\Z_\ell$-module $B(A)_\ell$ with a continuous $\Z_\ell$-homomorphism 
\[
\be: \bar C_3(\GL_d(A))/{\rm Im}(\partial_4)_\ell\to B(A)_\ell.
\]
such that:

\begin{itemize}
  \item $G$ acts continuously on $B(A)_\ell$ and $\be$ is $G$-equivariant,

 \item  $\be: \rH_3(\GL_d(A), \Z_\ell)\to B(A)_\ell$ vanishes on ${\rm H}^{\rm dec}_3(\GL_d(A), \Z_\ell)$, which is defined similarly to the discrete case before, but with $C$ running over closed subgroups of $\GL_d(A)$.
 \end{itemize}
  \end{itemize}
  
 Denote by $B(A)_\ell(-1)$ the $\Z_\ell$-module $B(A)_\ell$ with the twisted $G$-action  
 \[
 \sigma\cdot b=a_\sigma^{-1}\phi(\sigma)(b).
 \]

  \begin{prop}\label{masterpro} Under the assumptions above, 
suppose in addition that for each $\sigma\in G$, the representation 
 \[
 \gamma\mapsto \rho^\sigma(\gamma):=\phi(\sigma)^{-1}(\rho(\ti\psi(\sigma)(\gamma))
 \]
 is equivalent to $\rho$. Choose a $\Z_\ell$-generator $[c]$ of $\rH_2(\Gamma,\Z_\ell)$ and $c\in Z_2(\Gamma)_\ell$ that represents it. Then the map
 \[
 \sigma\mapsto \be\rho(a_\sigma^{-1}\ti\psi(\sigma)(c)-c)-a_\sigma^{-1}\be F_{h_{\ti\psi(\sigma)}}(\phi(\sigma)(\rho(c)))
 \] 
 gives a continuous $1$-cocycle $G\to B(A)_\ell(-1)$ whose
 class 
 \[
 \CS_{\rho,\psi, \phi}\in \rH^1_{\rm cts}(G, B(A)_\ell(-1))
 \]
depends only on $[c]$, $\psi$, $\phi$, and the equivalence class of $\rho$.
 \end{prop}
 
 \begin{proof}
It is similar to the proof in the discrete case. The additional claim of continuity 
of the cocycle map follows from Lemma \ref{lemmaCont}. 
\end{proof}

\subsubsection{}\label{quo} Let $\Gamma\to \Gamma'$ be a quotient profinite group with characteristic kernel
and such that $\rH_2(\Gamma',\Z_\ell)\simeq\Z_\ell$, $\rH_3(\Gamma',\Z_\ell)=0$. Assume that
$\rH_2(\Gamma, \Z_\ell)\simeq\Z_\ell\to \rH_2(\Gamma',\Z_\ell)\simeq\Z_\ell$ is the identity and take $[c']=[c]$.
Since the kernel of $\Gamma\to \Gamma'$ is a characteristic subgroup,
 $\psi: G\to {\rm Out}(\Gamma)$ factors to give  $\psi': G\to {\rm Out}(\Gamma')$.
Finally, assume that $\rho:\Gamma\to \GL_d(A)$ factors 
\[
\Gamma\to \Gamma'\xrightarrow{\rho'} \GL_d(A).
\]
Then we have
\[
\CS_{\rho,\psi,\phi}=\CS_{\rho',\psi',\phi}.
\]

\subsubsection{}\label{res} Let $\Gamma'\subset \Gamma$ be an open  subgroup. Suppose we also have 
$\rH_2(\Gamma',\Z_\ell)\simeq\Z_\ell$, $\rH_3(\Gamma',\Z_\ell)=0$, and that the natural map 
$\rH_2(\Gamma',\Z_\ell)\to \rH_2(\Gamma,\Z_\ell)$ is multiplication by the index $[\Gamma:\Gamma']$.
Choose generators $[c]$, $[c']$, such that $[c]=[\Gamma:\Gamma']^{-1}\cdot [c']$.

 Suppose that $(\rho, \psi, \phi)$ is as above. Suppose that there is a continuous homomorphism 
$\psi': G\to {\rm Out}(\Gamma')$ which is compatible with $\psi$ in the following sense: For each $\sigma\in G$, there is $\widetilde{\psi(\sigma)}\in {\rm Aut}(\Gamma)$ representing $\psi(\sigma)\in {\rm Out}(\Gamma)$ such that the restriction $\widetilde{\psi(\sigma)}_{|\Gamma'}\in {\rm Aut}(\Gamma')$ represents $\psi'(\sigma)\in {\rm Out}(\Gamma')$.

\begin{prop}\label{resprop}
Under the above assumptions, we have
\[
\CS_{\rho_{|\Gamma'}, \psi', \phi}=[\Gamma:\Gamma']\cdot \CS_{\rho,\psi, \phi}
\]
in $\rH^1_{\rm cts}(G, B(A)_\ell(-1))$.
\end{prop}

\begin{proof}
The map $C_{3,2}(\Gamma')_\ell \to C_{3,2}(\GL_d(A))_\ell$ given by the restriction $\rho_{|\Gamma'}: \Gamma'\to \GL_d(A)$ is the composition
\[
C_{3,2}(\Gamma')_\ell\to C_{3,2}(\Gamma)_\ell \xrightarrow{\rho} C_{3,2}(\GL_d(A))_\ell.
\]
The class $\CS_{\rho_{|\Gamma'}, \psi', \phi}$ is given by the $1$-cocycle
\[
\sigma\mapsto \be\rho(a_\sigma^{-1}\ti\psi'(\sigma)(c')-c')-a_\sigma^{-1}\be F_{h_{\ti\psi'(\sigma)}}(\phi(\sigma)(\rho(c')))
\]
where $c'\in Z_2(\Gamma')$ is a fundamental cycle, i.e. a $2$-cycle with $[c']=1$ in $\rH_2(\Gamma',\Z_\ell)$. Here, we can take $h_{\ti\psi'(\sigma)}$ to be given as $h_{\ti\psi(\sigma)}$; note that 
$h_{\ti\psi(\sigma)}$ is well-defined in $\GL_d(A)/\mathfrak Z_A(\rho)$ which maps to
$\GL_d(A)/\mathfrak Z_A(\rho_{|\Gamma'})$. Since $\rH_2(\Gamma',\Z_\ell)\to \rH_2(\Gamma,\Z_\ell)$ is multiplication by
$[\Gamma:\Gamma']$, if $c\in Z_2(\Gamma')$ is a fundamental $2$-cycle for $\Gamma'$, then the image of $c'$ in $Z_2(\Gamma)$ is $c'=[\Gamma:\Gamma']\cdot c+\partial_3(d)$ and the result follows from the definitions.
\end{proof}

\subsection{An $\ell$-adic regulator}\label{lregulator} Let $A=\O\lps x_1,\ldots , x_m\rps$ with $\O=\O_E$, the ring of integers of a totally ramified extension of $W(\BF)[1/\ell]$ of ramification index $e$. We will allow $m=0$ which corresponds to $A=\O$. The main example of $B(A)_\ell$ and 
\[
\be_A: \bar C_3(\GL_d(A))/{\rm Im}(\partial_4)_\ell\to B(A)_\ell
\]
which we use in the above is obtained by taking $B(A)_\ell=\sO(D_1(m))$ and $\be_A$  
given by an $\ell$-adic regulator. We will now explain this construction. Recall we set $\sO=\sO(D)=\sO(D_1(m))$.

\subsubsection{} Fix an odd positive integer $s\geq 3$. For our example, we actually take $s=3$. Consider the $A$-algebra
\[
\hat\calA={\rm Mat}_{d\times d}(A)\lps z_0,\ldots ,z_{n}\rps\otimes_A \wedge^\bullet_A (Adz_0+\cdots +Adz_n).
\]
Set $\calA$ for the quotient of $\hat\calA$ by the ideal generated by $(z_0+\cdots +z_s)-1$,
$dz_0+\cdots +dz_s$. We can write elements $T\in \calA$ in the form
\[
T=\sum_{\bf a} \sum_{u=0}^s T_{{\bf a}, u} z_0^{a_0}\cdots z_s^{a_s} dz_0\wedge\cdots \wedge\hat dz_u\wedge\cdots \wedge dz_s
\]
with ${\bf a}=(a_0,\ldots ,a_s)\in {\mathbb N}^{s+1}$, $T_{{\bf a}, u}\in {\rm Mat}_{d\times d}(A)$.

Take ${\bf X}=(X_0,\ldots,  X_s)$, $X_i\in {\rm Mat}_{d\times d}(\frakm^b)$, $i=0,\ldots, s$, $b\geq 1$.
Let
\[
\nu({\bf X})=1+(X_0z_0+\cdots + X_sz_s)\in \calA
\]
which is invertible with
\[
\nu({\bf X})^{-1}=1+\sum\nolimits_{i\geq 1}(-1)^i (X_0z_0+\cdots + X_sz_s)^i.
\]
Set $d\nu({\bf X})=X_0dz_0+\cdots + X_sdz_s$ so then
$\nu({\bf X})^{-1} d\nu({\bf X})$ is in $\calA$.
Finally set
\[
T({\bf X})=(\nu({\bf X})^{-1} d\nu({\bf X}))^s=\sum_{\bf a} \sum_{u=0}^s T_{{\bf a}, u} z_0^{a_0}\cdots z_s^{a_s} dz_0\wedge\cdots \wedge\hat dz_u\wedge\cdots \wedge dz_s
\]
where, as we can see, $T_{{\bf a}, u}\in {\rm Mat}_{d\times d}({\frak m}^{b(\onorm{\bf a}+s)})$. 

Following Choo and Snaith \cite{ChooSnaith} we set:
\[
\Phi_s(T({\bf X}))=\sum_{\bf a} \frac{a_0! a_1!\cdots a_s!}{(\onorm{\bf a}+s)!} (\sum_{u=0}^s (-1)^u  {\rm Trace}(T_{{\bf a}, u})).
\]
By the above, ${\rm Trace}(T_{{\bf a}, u})\in \frakm^{b(\onorm{\bf a}+s)}$.
Using Lemma \ref{multiestimate} we see 
\[
|\frac{a_0! a_1!\cdots a_s!}{(\onorm{\bf a}+s)!}|_\ell\leq |s!|_\ell^{-1} \ell^{(s+1)d_\ell(\onorm{\bf a}+s)}\leq C_1 \ell^{C_2d_\ell(\onorm{\bf a})}.
\]
By Proposition \ref{sequence} (a), $\Phi_s(T({\bf X}))\in \sO(D_1(m))$ and we have
\begin{equation}\label{b1}
||\Phi_s(T({\bf X}))||_{(1/\ell)^a}\leq C_1\ell^{N(C_2, ab)}
\end{equation}
for $a\in \Q\cap (0, 1/e]$.
\begin{lemma}\label{lemmaEst1}
Fix $r=(1/\ell)^{a/e}$. For each $\epsilon>0$, 
there is $b_0$ such that for all $b\geq b_0$, ${\bf X}\in {\rm Mat}_{d\times d}(\frakm^b)^{s+1}$
implies $
||\Phi_s(T({\bf X}))||_{r}<\epsilon$.
\end{lemma}

\begin{proof} It follows from (\ref{b1}), Lemma \ref{trivial}, and the above.
\end{proof}

Hence, the map ${\rm Mat}_{d\times d}(\frakm)^{s+1}\to \sO(D)$ given by ${\bf X}\mapsto \Phi_s(T({\bf X}))$
is continuous for the $\frakm$-adic and Fr\'echet topologies of the source and target.

\subsubsection{}
Now set $K_b=\ker(\GL_d(A)\to \GL(A/\frakm^b))$. For $(g_0,\ldots , g_s)\in K_1$, we set
\[
\ti\Phi_s (g_0,\ldots , g_s)=\Phi_s(T(g_0-1, \ldots, g_s-1)).
\]

\begin{thm}\label{ChooSnaithThm}
(1) For $h\in K_1$, $(g_0,\ldots , g_s)\in K_1^{s+1}$, we have
\[
\ti\Phi_s (hg_0,\ldots , hg_s)=\ti\Phi_s (g_0,\ldots , g_s)=\ti\Phi_s (g_0h,\ldots , g_sh).
\]

(2) For $g\in \GL_d(A)$, $(g_0,\ldots , g_s)\in K_1^{s+1}$, we have
\[
\ti\Phi_s (gg_0g^{-1},\ldots , gg_sg^{-1})=\ti\Phi_s (g_0,\ldots , g_s).
\]

(3) $\ti\Phi_s$ is alternating, i.e. for each permutation $p$, 
\[
\ti\Phi_s(g_{p(0)}, \ldots , g_{p(s)})=(-1)^{{\rm sign}(p)}\ti\Phi_s(g_{0}, \ldots , g_{s}).
\]

(4) The map $\ti\Phi_s:  K_1^{s+1}\to \sO(D)$ extends linearly to $\ti\Phi_s: \Z_\ell\lps K_1^{s+1}\rps\to \sO(D)$
which gives a continuous $s$-cocycle.

\end{thm}

\begin{proof}
The identities in (1) and (2) and the cocycle identity in (4) are stated in Theorem 3.2 \cite{ChooSnaith} for the evaluations
at all classical points $A\to \O_F$. For these evaluations, they follow from the expression for $\Phi_s(T({\bf X}))$ as a constant multiple of 
\[
\int_{\Delta^{s}} {\rm Trace}((\nu({\bf X})^{-1} d\nu({\bf X}))^s)
\] 
(see \cite{Hamida}); here the integration is over $\Delta^s$ given by $z_0+\cdots +z_s=1$. The
identities in $\sO(D)$ follow. 

The continuous extension of $\ti\Phi_s$ to $\Z_\ell\lps K_1^{s+1}\rps\to \sO(D)$
follows from Lemma \ref{lemmaEst1} since
\[
\Z_\ell\lps K_1^{s+1}\rps=\varprojlim\nolimits_b \Z_\ell[(K_1/K_b)^{s+1}];
\]
 see also Proposition \ref{sequence} (b). The alternating property (3) follows
 quickly from the definition of $T(X_0, \ldots, X_s)$ and $\Phi_s(T(X_0, \ldots, X_s))$ by noting that it involves 
 the exterior product. 
  \end{proof}

\subsubsection{} We now define a transfer of the cocycle $\ti\Phi_s$ from $K_1$ to $\GL_d(A)$
as follows: Denote reduction modulo $\frakm$ by $a\mapsto \bar a$ and apply the Teichm\"uller representative on the entries to give a set-theoretic lift $\GL_d(A)/K_1=\GL_d(\BF)\to \GL_d(A)$ which we 
denote by $h\mapsto [h]$. Note that for every $g\in \GL_d(A)$, $h\in \GL_d(\BF)$,
$[h]g [h\bar g ]^{-1}\in K_1$. We let the
transfer of $\ti\Phi_s$ be
\[
\Psi_s (g_0,\ldots , g_s):=\frac{1}{\# \GL_d(\BF)}  \sum_{h\in \GL_d(\BF)} \ti\Phi_s([h]g_0 [h\bar g_0 ]^{-1}, \ldots , [h]g_s [h\bar g_s ]^{-1} ).
\]
 (cf. [NSW], p. 48, Ch. I, \S 5). This gives a continuous homogeneous $s$-cocycle 
 \[
 \Psi_s: \Z_\ell\lps \GL_d(A)^{s+1}\rps\to \sO(D),
 \]
 where $\GL_d(A)$ acts trivially on $\sO(D)$. If $(g_i)\in K_1^{s+1}$, then $[h]g_i [h\bar g_i ]^{-1}=[h]g_i[h]^{-1}$, and so $\Psi_s(g_i)=\ti\Phi_s(g_i)$, by
Theorem \ref{ChooSnaithThm} (2).

For $s=3$ we get a continuous 
\[
 \Psi_{3, A}: \bar C_3(\GL_d(A))/{\rm Im}(\partial_4)_\ell\to  \sO(D).
\]

\subsubsection{} As we shall see below, the restriction of $\Psi_{3, A}$ to homology agrees, up to non-zero constant, with the  $\ell$-adic (Borel) regulator.
This follows from work of Huber-Kings and Tamme. Also, as we will explain, this comparison allows us to  also deduce that 
$\Psi_{3,A}$ vanishes on the subgroup $\rH^{\rm dec}_3(\GL_d(A))$, when $d\geq 3$. Hence, we can set
\[
 \be_A=\Psi_{3, A}: \bar C_3(\GL_d(A))/{\rm Im}(\partial_4)_\ell\to  \sO(D)
\]
and use this in the constructions of the previous section.

We now explain this in more detail:

Consider $A=\O=\O_E$, i.e. $m=0$ and $d\geq s$. By \cite{HubKings} 
(see also \cite{Tamme} Theorem 2.1) the
Lazard isomorphism
 \[
\rH^s_{\rm la}(\GL_d(\O), E)\simeq \rH^s({\mathfrak {gl}}_d, E)
\]
(the subscript here stands for ``locally analytic") is induced on the level of cochains
by the map
\[
\Delta: \O^{\rm la}(\GL_d(\O)^{\times k})\to \wedge^k{\mathfrak {gl}}^\vee_d,
\]
which is given on topological generators by $f_1\otimes\cdots \otimes f_k\mapsto df_1(1)\wedge\cdots \wedge df_k(1)$. Here, $df(1)$ is the differential of the function $f$ evaluated at the identity. Now by \cite[Theorem 2.5]{Tamme}, the restriction of $\Psi_{s,\O}$ to the homology 
\[
\Psi_{s,\O}: \rH_s(\GL_d(\O), \Z_\ell)\to E
\]
relates to the $\ell$-adic (Borel) regulator: By  \cite[Theorem 2.5]{Tamme}, (see also \cite{HubKings}),
$\Psi_{s,\O}$, up to non-zero constant, is obtained from the element $f_s\in \rH^s_{\rm la}(\GL_d(\O), E)$ which under
the Lazard isomorphism
 \[
\rH^s_{\rm cts}(\GL_d(\O), E)\simeq \rH^s_{\rm la}(\GL_d(\O), E)\simeq \rH^s({\mathfrak {gl}}_d, E)
\]
is the class of  the cocycle $\wedge^s_E {\mathfrak {gl}}_d\to E$ given by
\[
X_1\wedge \cdots \wedge X_s\mapsto p_s(X_1, \cdots , X_s)=\sum_{\sigma\in S_s} (-1)^{{\rm sign}(\sigma)} {\rm Trace}(X_{\sigma(1)}\cdots  X_{\sigma(s)}).
\]
We can easily see, using the cyclic invariance of the trace, that if in $(X_1,\ldots , X_s)$ there is a matrix which commutes with all the others, then $p_s(X_1,\ldots , X_s)=0$. 

Now suppose $s=3$. 
Let $C\subset \GL_d(\O)$ be a closed (therefore $\ell$-analytic, see \cite{SerreLie}) subgroup of $\GL_d(\O)$
with $E$-Lie algebra $\mathfrak c$. For $h$ in the centralizer of $C$, we denote by $\mathfrak h$ the $1$-dimensional $E$-Lie algebra of the $\ell$-analytic subgroup $h^{\Z_\ell}$ of $\GL_d(\O)$ 
(i.e. the closure of the powers of $h$.) Lazard's isomorphism applies to $C$ and gives
\[
\rH^2_{\rm cts}(C, E) \simeq \rH^2({\mathfrak {c}}, E)^C.
\]
These isomorphisms fit in a commutative diagram
\[
\begin{matrix} 
\rH^3_{\rm la}(\GL_d(\O), E)&\to &{\rm Hom} (\rH_3(\GL_d(\O)), E)&\xrightarrow{\nabla_{h, C}^\vee} &{\rm Hom} (\rH_2(C), E)\simeq \rH^2_{\rm ct}(C, E)\\
\downarrow &&\downarrow && \downarrow \\
\rH^3({\mathfrak {gl}}_d, E)&\to &{\rm Hom}_E(\rH_3({\mathfrak {gl}}_d), E)&\xrightarrow{\nabla_{{\mathfrak h}, {\mathfrak c}}^\vee}  & {\rm Hom}_E(\rH_2({\mathfrak c}), E)\simeq \rH^2({\mathfrak {c}}, E)
\end{matrix}
\]
with the last vertical map an injection.
Here, 
\[
\nabla_{\mathfrak h,\mathfrak c}: E\otimes_ E \rH_2(\mathfrak c)\to \rH_3(\mathfrak {gl}_d)
\]
is given by sending $x\otimes (\sum_{j} a_{j} (y_{j1}\wedge y_{j2}))$ to
\[
  \sum_j a_j(x\wedge y_{j1}\wedge y_{j2}-  y_{j1}\wedge x\wedge y_{j2}+y_{j1}\wedge y_{j2}\wedge x)=3  \sum_j a_j(x\wedge y_{j1}\wedge y_{j2}).
\] 
In this, $[x, y_1]=[x, y_2]=0$,
and  $ \sum_{j} a_{j} [y_{j1},y_{j2}]=0$. It then follows that  $f_3\in \rH^3_{\rm la}(\GL_d(\O), E)$
maps to $0$ in ${\rm Hom} (\rH_2(C), E)$. This implies the desired result for the evaluation of $\be_A$  at each point $A\to \O$ and therefore also for $\be_A$.

\medskip

 \section{Representations of \'etale fundamental groups}
 
 We will now apply the constructions of the previous section to the case in which
 the profinite group $\Gamma$ is the geometric \'etale fundamental group of a
 smooth projective curve defined over a field $k$  of characteristic $\neq \ell$.

\subsection{\'Etale fundamental groups of curves} 

Let $X$ be a smooth curve over $k$. Set  $\bar X=X\otimes_k\bar k$
and choose a $\bar k$-valued point $\bar x$ of $X$.
We have the standard exact sequence of \'etale fundamental (profinite) groups
\[
1\to \pi_1(\bar X, \bar x)\to \pi_1(X, \bar x)\to  \Gk\to 1.
\]
(cf. \cite{SGA1}, Exp. IX, \S 6.) We will assume that $X$ is projective and, for
simplicity, that $\bar X$ is irreducible.

We set $\Gamma=\pi_1(\bar X,\bar x)$, $\Gamma_0=\pi_1(X, \bar x)$, considered as  profinite groups. Note that $\Gamma=\pi_1(\bar X, \bar x)$ is topologically finitely generated (\cite{SGA1}, Exp. X, Theorem 2.6).  By \cite[I, \S 1, Prop. 1]{SerreGaloisCoh}, there is a continuous set theoretic section 
$s: G=\Gk\to \Gamma_0$. In this case, such a section can be constructed as follows: Choose a point of $X$ defined over a finite separable extension $k\subset k'\subset k^{\rm sep}$. We can assume that $k'/k$ is Galois and so it corresponds to a finite index normal open subgroup $U\subset  G$. As usual, pull-back by the morphism $\Spec(k')\to X$ gives a continuous homomorphic section $s_U: U\to \Gamma_0$. We can now extend $s_U$ to the desired $s$ by choosing a representative $g_i$ of each coset $G/U$ and arbitrarily assigning $s(g_i)=s_i\in \Gamma_0$;
then $s(g_iu)=s_i s_U(u)$ works.

We have (\cite{NSW}, Ch. II, Thm (2.2.9))
\[
\rH_i(\Gamma, \Z_\ell)\simeq \rH^{i}_{\rm cts}(\Gamma, \Q_\ell/\Z_l)^*
\]
where $(\ )^*={\rm Hom}(\ ,\Q_\ell/\Z_\ell)$ is the Pontryagin dual.
Now, since $\bar X$ is a ${\rK}(\pi_1(\bar X), 1)$-space for $\ell$-torsion \'etale sheaves (cf. \cite[Theorem 11]{Friedlander}),
\[
\rH^{i}_{\rm cts}(\Gamma, \Q_\ell/\Z_\ell)=\varinjlim_n \rH^i(\Gamma, \ell^{-n}\Z/\Z)=
\varinjlim_n \rH^i_{\et}(\bar X, \ell^{-n}\Z/\Z). 
\]
Since $\rH^3_{\et}(\bar X, \ell^{-n}\Z/\Z)=0$, 
$\rH^2_{\et}(\bar X, \ell^{-n}\Z/\Z)=(\Z/\ell^n\Z)(-1)$,
we get
\[
\rH_3(\Gamma, \Z_\ell)=0,\quad \rH_2(\Gamma, \Z_\ell)\simeq \Z_\ell(1).
\]
In fact, the isomorphism $\rH_2(\Gamma, \Z_\ell)\simeq \Z_\ell(1)$ is canonical, given 
by Poincare duality.

\subsection{The $\ell$-adic volume}

Suppose  $A=\O\lps x_1,\ldots, x_m\rps$, where $\O=\O_E$ is the ring of integers
in a finite extension $E$ of $\Q_\ell$ with residue field $\BF$; this includes the case $A=\O$ (for $m=0$).
Recall $\sO$ is the ring of analytic functions on the polydisk $D=D_1(m)$
(when $m=0$, $\sO=E$.)

Let $\rho_0: \pi_1(X, \bar x)\to \GL_d(A)$ be a continuous representation.
Apply the construction of Proposition \ref{masterpro} to $\Gamma=\pi_1(\bar X, \bar x)$, $\Gamma_0=\pi_1(X, \bar x)$, $G=\Gk$,
with $G$ acting trivially on $A$, $\ti\psi: G\to {\rm Aut}(\Gamma)$ given via $s$,  $\rho={\rho_0}_{|\pi_1(\bar X, \bar x)}$ and
\[
\be_A: \bar C_3(\GL_d(A))/{\rm Im}(\partial_4)_\ell\to \sO
\]
 given by the $\ell$-adic regulator. This gives a continuous $1$-cocycle $\Gk\to \sO(-1)$:

 \begin{Definition}
The cohomology class 
\[
{\CS}(\rho)=\CS_\rho\in \rH^1_{\rm cts}(\Gk, \sO(-1))=\rH^1_{\rm cts}(k, \sO(-1)),
\]
given by the construction of Proposition \ref{masterpro}, is the $\ell$-adic volume of $\rho_0$.
(It depends only on the restriction $\rho=\rho_0{|_{\pi_1(\bar X, \bar x)}}$.)
\end{Definition}

The restriction of the cocycle to $G_{k\cy}={\rm Gal}(k^{\rm sep}/k\cy)$ gives a well-defined continuous homomorphism
\[
\CS_{\rho,k\cy}: G_{k\cy}\to \sO.
\]
We set
\[
G_\infty:=\Gal(k\cy/k)
\]
which, via $\chi_{\rm cycl}$, identifies with a subgroup of $\Z_\ell^\times=\Z/(\ell-1)\times \Z_\ell$. 

Consider the restriction-inflation exact sequence
\begin{multline*}
1\to \rH^1_{\rm cts}(G_\infty, \sO(-1))\to \rH^1_{\rm cts}(k, \sO(-1))\to \\
\to \rH^1_{\rm cts}(k\cy, \sO(-1))^{G_\infty}\to \rH^2_{\rm cts}(G_\infty, \mathscr O(-1)).\ \ 
\end{multline*}
Using  that $\sO$ is a $\Q_\ell$-vector space   we can see that $\rH^1_{\rm cts}(G_\infty, \sO(-1))=0$,
$\rH^2_{\rm cts}(G_\infty, \sO(-1))=0$. Hence, restriction gives
\[
\rH^1_{\rm cts}(k, \sO(-1))\xrightarrow{\simeq} \rH^1_{\rm cts}(k\cy, \sO(-1))^{G_\infty},
\]
and $\CS({\rho})$ is determined by the continuous homomorphism 
\[
\CS_{\rho,k\cy}: G_{k\cy}\to \sO(-1)
\]
which is $G_\infty$-equivariant. In what follows, we will also simply write $\CS(\rho)$ for this homomorphism
and omit the subscript $k\cy$. Actually, using the continuity, we see that $\CS(\rho)$ factors through the maximal
abelian pro-$\ell$-quotient 
\[
G^{\rm ab}_{k\cy,\ell}={\rm Gal}^{\rm ab}(k^{\rm sep}/k\cy)_\ell.
\]

\subsection{}
In the following paragraphs we elaborate on some properties of $\CS({\rho})$. We start with an alternative definition.

\subsubsection{} We can also use the Leray-Serre spectral sequence 
\[
{\rm E}^{p, q}_2: \rH^p(k, \rH^q_\et(\bar X, \Q_\ell/\Z_\ell))\Rightarrow \rH^{p+q}_\et(X,\Q_\ell/\Z_\ell)
\]
to give a construction 
of  a class $\CS^{s}(\rho)$ as follows: 

Set $\rH^q(\bar X):=\rH^q_\et(\bar X, \Q_\ell/\Z_\ell)$,  $\rH^q( X):=\rH^q_\et(X, \Q_\ell/\Z_\ell)$.
 We are interested in $\rH^3(X)$. The spectral sequence gives a filtration
\[
(0)=F^4\rH^3(X)\subset F^3\rH^3(X)\subset F^2\rH^3(X)\subset F^1\rH^3(X)\subset F^0\rH^3(X)=\rH^3(X)
\]
with graded pieces ${\rm gr}_p \rH^3(X)\simeq {\rm E}^{p,3-p}_\infty$. 
Using that $\rH^q(\bar X, \Q_\ell/\Z_\ell)=(0)$ unless $q=0$, $1$, $2$,
we see that ${\rm gr}_0\rH^3(X)=(0)$ and that 
\[
{\rm E}^{1, 2}_\infty={\rm E}^{1,2}_4\subset {\rm E}^{1,2}_3\subset \rH^1(k, \rH^2(\bar X))=\rH^1(k, \Q_\ell/\Z_\ell(-1))
\]
with
 \begin{align*}
{\rm E}^{1,2}_3=&\ {\rm ker}(d^{1,2}_2: \rH^1(k, \Q_\ell/\Z_\ell(-1))\to \rH^3(k, \rH^1(\bar X)))\\
{\rm E}^{1,2}_\infty=&\ {\rm E}^{1,2}_4={\rm ker}(d^{1,2}_3: {\rm E}^{1,2}_3\to \rH^4(k, \Q_\ell/\Z_\ell)).
\end{align*}
We obtain 
\begin{equation*}
\eta: \rH^3(X, \Q_\ell/\Z_\ell)=F^1\rH^3(X)\twoheadrightarrow {\rm gr}_1\rH^3(X)={\rm E}^{1,2}_\infty\hookrightarrow \rH^1(k, \Q_\ell/\Z_\ell(-1)).
\end{equation*}

In what follows, for simplicity, we omit denoting the base point and simply write $\pi_1(X)$ and $\pi_1(\bar X)$.
Let us now compose $\eta$ with the natural 
\[
\rH^3(\pi_1(X),\Q_\ell/\Z_\ell)\simeq \rH^3(X,\Q_\ell/\Z_\ell)
\]
 ($X$ is a $\rK(\pi_1, 1)$-space for $\ell$-torsion \'etale sheaves) and then take 
Pontryagin duals to obtain
\[
\eta': \rH_1(k, \Z_\ell(1))\cong \rH^1(k, \Q_\ell/\Z_\ell(-1))^* \to \rH^3(\pi_1(X),\Q_\ell/\Z_\ell)^*\cong \rH_3(\pi_1(X),\Z_\ell).
\]
By further composing $\eta'$ with $\rH_3(\rho_0): \rH_3(\pi_1(X), \Z_\ell)\to \rH_3(\GL_d(A), \Z_\ell)$ and the $\ell$-adic regulator $\be_A: \rH_3(\GL_d(A), \Z_\ell)\to \sO$
we obtain a continuous homomorphism
\[
 \rH_1(k, \Z_\ell(1))\to \sO.
\]
By the universal coefficient theorem, this uniquely corresponds to a class
\[
\CS^s(\rho)\in \rH^1_{\rm cts}(k, \sO(-1)).
\]
\begin{Remark}{\rm 
a) By tracing through all the maps in the construction, one can check  
\[
\CS^s(\rho)=\pm \CS(\rho),
\]
where the sign depends on the normalization of the differentials in the spectral sequence. Since we are not going to use this, we omit the tedious details.

b) We can also see that the homomorphism
\[
\eta: \rH^3(X,\Q_\ell/\Z_\ell)\to \rH^1(k, \Q_\ell/\Z_\ell(-1))
\]
above is, up to a sign, given by the push-down 
\[
{\rm R}^if_{\et, *}: \rH^i(X,\Q_\ell/\Z_\ell(m))\to \rH^{i-2}(k, \Q_\ell/\Z_\ell(m-1))
\]
for $i=3$, $m=0$, and the structure morphism $f: X\to \Spec(k)$.
}
\end{Remark}

\subsubsection{}
Suppose that the $\ell$-cohomological dimension ${\rm cd}_\ell(k)$ of $k$  is $\leq 2$.
Then  ${\rm E}^{1,2}_\infty=\rH^1(k, \Q_\ell/\Z_\ell(-1))$ and
${\rm E}^{3,1}_2=(0)$ in the above.
Then the spectral sequence gives a natural exact sequence
\begin{multline}\label{exact434}
 \ \ (\Q_\ell/\Z_\ell(-1))^{G_k}\to \rH^2(k, \rH^1_\et(\bar X, \Q_\ell/\Z_\ell))\to \\
\to \rH^3_\et(X, \Q_\ell/\Z_\ell)\xrightarrow{\eta} \rH^1(k, \Q_\ell/\Z_\ell(-1))\to 0.\ \ 
\end{multline}

Often, the situation simplifies even more:

\begin{thm}\label{Jannsen}(Jannsen) Assume $\ell\neq 2$. Suppose that $k$ is a number field, a global function field
of characteristic $\neq \ell$, or a finite extension of $\Q_p$  ($p=\ell$ is allowed).
Then 
\[
\eta: \rH^3_\et(X, \Q_\ell/\Z_\ell)\xrightarrow{\sim} \rH^1(k, \Q_\ell/\Z_\ell(-1))
\]
is an isomorphism.
\end{thm}

\begin{proof} Note that since we assume $\ell\neq 2$, 
the $\ell$-cohomological dimension ${\rm cd}_\ell(k)$ of $k$  is $\leq 2$,
for all the fields considered in the statement.
The exact sequence (\ref{exact434}) implies that it is enough to show 
\[
\rH^2(k, \rH^1_\et(\bar X, \Q_\ell/\Z_\ell))=(0).
\]
This vanishing follows from the results of \cite{Jannsen}.
In the number field case, this is \cite{Jannsen} \S 7, Cor. 7 (a).
In the global function field case, Jannsen shows  a more general result (loc. cit., Theorem 1). Finally, the local case is shown 
in the course of the proof of the number field case in loc. cit. \S 7. 
\end{proof}

\begin{Remark}{\rm
Jannsen conjectures a vanishing statement which is a lot more general.
See \cite{Jannsen} Conjecture 1   
and \S 3, Lemma 5.  }
\end{Remark}

\begin{cor}\label{CorJannsen}
Under the assumptions of Theorem \ref{Jannsen}, we have
 \[
\rH^1_{\rm cts}(k, \sO(-1))\simeq \Hom_{\Z_\ell}(\rH_3(\pi_1(X),\Z_\ell), \sO)
\]
and, under this isomorphism, the $\ell$-adic volume $\CS^s(\rho)$ is given by the $\Z_\ell$-homomorphism
 \[
 \rH_3(\pi_1(X),\Z_\ell)\to \sO
 \]
which is the composition of $\rH_3(\rho)$ with the
 $\ell$-adic regulator.\endproof
 \end{cor}
 
 We now continue our discussion of the group 
 \[
\rH^1_{\rm cts}(k, \sO(-1))\xrightarrow{\simeq} \rH^1_{\rm cts}(k\cy, \sO(-1))^{G_\infty}.
\]

 \subsubsection{}
Assume that $k$ is a finite field of order $q$, ${\rm gcd}(\ell, q)=1$. Then $G^{\rm ab}_{k\cy,\ell}=(1)$ and so $
\rH^1_{\rm cts}(k, \sO(-1))=(0)$. Hence,
$\CS(\rho)=0$
for all $X$ and $\rho$.
 
\subsubsection{}  Let $k$ be a local field which is a finite extension of $\Q_p$.
Write $G_\infty=\Delta \times \Gamma$, where $\Delta=\Gal(k(\zeta_\ell)/k)$ is a finite cyclic 
group of order that divides $\ell-1$ and $\Gamma\simeq \Z_\ell$. 
By a classical result of Iwasawa
\[
G^{\rm ab}_{k\cy,\ell}\cong \begin{cases}
 \Z_\ell(1), \ \ \ \ \ \ \ \ \ \ \ \ \ \ \ \ \ \ \ \ \ \ \ \hbox{\rm if\ } \ell\neq p, & \\
 \Z_\ell\lps G_\infty\rps^{[k:\Q_\ell]}\oplus \Z_\ell(1),\, \ \ \hbox{\rm if}\ \ell=p .
 \end{cases}
\]
as $\Z_\ell\lps G_\infty\rps$-modules.  
(See for example,  \cite[Theorem (11.2.4)]{NSW}). 
It follows that $\CS(\rho)$ takes values in   
\[
\rH^1_{\rm cts}(k, \sO(-1))\simeq  \rH^1_{\rm cts}(k\cy, \sO(-1))^{G_\infty}\cong \begin{cases}
 (0), \ \ \ \ \ \ \hbox{\rm if\ } \ell\neq p & \\

 \sO^{[k:\Q_\ell]},\ \ \hbox{\rm if}\ \ell=p .
 \end{cases}
\]

\subsubsection{} Let $k$ be a number field with $r_1$ real and $r_2$ complex places. 
For a place $v$ of $k$, fix $\bar k\hookrightarrow \bar k_v$ which gives $G_v=\Gal(\bar k_v/k_v)\hookrightarrow G_k$.  Using the local case above, we see that for all finite places $v$ away from $\ell$, the restriction of $\CS(\rho)$ to $G_v\cap {\rm Gal}(\bar k /k\cy)$ is trivial. It follows that $\CS(\rho)$ factors through 
the Galois group $\mathscr X_\infty$ of the maximal abelian pro-$\ell$ extension of $k\cy$ which is unramified outside $\ell$.  
We have
\[
\CS(\rho)\in \Hom_{\rm cts}(\mathscr X_\infty, \sO(-1))^{G_\infty}=\Hom_{\rm cts}(\mathscr X_\infty(1)_{G_\infty}, \sO).
\]
The  Galois group $\mathscr X_\infty$ is a classical object of Iwasawa theory:

Set $K=k(\zeta_\ell)$, denote by $k_\infty$ the cyclotomic $\Z_\ell$-extension of $k$, and denote by $K_\infty=Kk_\infty$  the cyclotomic $\Z_\ell$-extension of $K$.
 Then
\[
G_\infty=\Gal(k\cy/k)=\Gal(K_\infty/k).
\]
As above, $G_\infty=\Delta\times \Gamma$, $\Gamma\simeq \Z_\ell$. Denote as usual  
\[
\Lambda=\Z_\ell\lps T\rps\simeq \Z_\ell\lps \Gamma\rps
\]
with the topological generator 
$1$ of $\Z_\ell\simeq \Gamma$ mapping to $1+T$. Then $\Z_\ell\lps G_\infty\rps\simeq \Lambda[\Delta]$.
By results of Iwasawa  (\cite{Iwa}, see slso \cite{NSW} Theorems (11.3.11), (11.3.18)):

1) $\mathscr X_\infty$ is a finitely generated $\Lambda[\Delta]$-module,

2) $\mathscr X_\infty$ has no non-trivial finite $\Lambda$-submodule, 

3) There is an exact sequence of $\Lambda[\Delta]$-modules
\[
0\to t_\Lambda(\mathscr X_\infty)\to \mathscr X_\infty\to  \Lambda[\Delta]^{r_2}\oplus \bigoplus_{v\in S_{\rm real}(k)} {\rm Ind}^{\langle c_v\rangle}_{\Delta}\Lambda^-\to T_2(\mathscr X_\infty)\to 0.
\]
Here, $t_\Lambda (\mathscr X_\infty)$ is the $\Lambda$-torsion submodule of $\mathscr X_\infty$ and $T_2(\mathscr X_\infty)$ is a finite $\Lambda$-module. Also,
 $c_v\in \Delta$ is the complex conjugation at $v$ and $\Lambda_\ell^-$ is the $c_v$-module with $c_v$ acting as multiplication by $-1$.  
 
 We now see that
 \begin{align*}
 \Hom_{\rm cts}(\Lambda[\Delta], \sO(-1))^{\Delta\times\Gamma}&\simeq \sO,  \\
 \Hom_{\rm cts}({\rm Ind}^{\langle c_v\rangle}_{\Delta}\Lambda^-, \sO(-1))^{\Delta\times\Gamma}&\simeq \sO.
 \end{align*}
 Therefore, we obtain
 \begin{equation*}\label{sesO}
0\to \sO^{r_1+r_2}\to \Hom_{\rm cts}(\mathscr X_\infty, \sO(-1))^{G_\infty} \to \Hom_{\rm cts}(t_\Lambda(\mathscr X_\infty), \sO(-1))^{G_\infty}\to 0.
 \end{equation*}
 To continue, 
 consider the following generalization of Leopoldt's conjecture  due to Schneider \cite{Schn}, for an integer $m\neq 1$:
 \smallskip
 \smallskip
 
 \begin{itemize}
 \item[]Conjecture $(C_m)$: \  $\rH_{\et}^2(\O_k[1/\ell], \Q_\ell/\Z_\ell(m))=(0)$. 
 \end{itemize}
 \smallskip

\begin{Remark}
{\rm This is also a very special case, for $X=\Spec(k)$, of the conjectures of \cite{Jannsen} mentioned above. 
$(C_0)$ is equivalent to Leopoldt's conjecture for $k$ and $\ell$. For $m\geq 2$, conjecture $(C_m)$ was shown by Soul\'e \cite{Soule} by relating the Galois cohomology group to the 
  group $\rK_{2m-2}(\O_k)$ which is finite by work of Borel. If $k$ is totally real, then $(C_m)$ implies $(C_{1-m})$ for $m$ even. Hence, if $k$ is totally real,  $(C_{m})$, for $m$ odd and negative, is true. 
\rm}
\end{Remark}
 
 By \cite{KNF96} Lemma 2.2, Theorem 2.3, assuming $(C_{-1})$, we have
 \[
 (t(\mathscr X_\infty)(1))_{G_\infty}=0,
 \]
 and so the last term in the short exact sequence above is trivial
 \[
 \Hom_{\rm cts}(t_\Lambda(\mathscr X_\infty), \sO(-1))^{G_\infty}=(0).
 \]
This gives that, assuming $(C_{-1})$, we have
\[
 \Hom_{\rm cts}(\mathscr X_\infty, \sO(-1))^{G_\infty}=\Hom_{\Z_\ell}(\mathscr X_\infty(1)_{G_\infty}, \sO)\simeq \sO^{r_1+r_2}
 \]
 and so $\CS(\rho)$ can be thought of as taking values in $\sO^{r_1+r_2}$.
 
 In fact, assuming $(C_{-1})$, \cite{KNF96} Theorem 2.3 gives  a canonical isomorphism
 \[
 \rH_{\et}^1(\O_k[1/\ell], \Q_\ell(-1))\cong \Hom_{\Z_\ell}(\mathscr X_\infty, \Q_\ell(-1))^{G_\infty}.
 \]
Consider now the
 semilocal pairing
 \[
 (\bigoplus_{v|\ell} \rH^1(k_v, \Q_\ell(-1)))\times (\bigoplus_{v|\ell} \rH^1(k_v, \Q_\ell(2)))\to \Q_\ell
 \]
 obtained by adding the local duality pairings (see \cite{KNF96}).  
 Assuming  $(C_{-1})$, Theorem 1.3 of loc. cit., states that the image of 
 \[
 r^\ell_{-1}: \rH_{\et}^1(\O_k[1/\ell], \Q_\ell(-1))\to \bigoplus_{v|\ell} \rH^1(k_v, \Q_\ell(-1))
 \]
 is the exact orthogonal of the image of 
 \[
 \rK_{3}(\O_k)\otimes_{\Z_\ell}\Q_\ell\xrightarrow{c_{2,1}}   \rH_{\et}^1(\O_k[1/\ell], \Q_\ell(2))\xrightarrow{r^\ell_2} \bigoplus_{v|\ell} \rH^1(k_v, \Q_\ell(2)),
 \]
 under this pairing. Here, $c_{2,1}$ is Soule's Chern class map \cite{Soule} which is an isomorphism by the Quillen-Lichtenbaum conjecture. Both $r^\ell_{-1}$ and $r^\ell_2$ are injective. 
 Note that, as $\Q_\ell$-vectors spaces, $\rK_{3}(\O_k)\otimes_{\Z_\ell}\Q_\ell\simeq \Q_\ell^{r_2}$,
 while $ \bigoplus_{v|\ell} \rH^1(k_v, \Q_\ell(2))\simeq \bigoplus_{v|\ell} \rH^1(k_v, \Q_\ell(-1)) \simeq \Q_\ell^{r_1+2r_2}.$
 
 We have shown:
  
 \begin{prop}
 Suppose that $k$ is a number field and assume conjecture $(C_{-1})$ for $k$ and $\ell$. Then,
 $\CS(\rho)\in \rH^1(k, \sO(-1))$ is  determined by its restrictions $\CS(\rho)_{k_v}\in \rH^1(k_v, \sO(-1))$,
 for $v|\ell$,
 and 
 \[
 (\CS(\rho)_{k_v})_v\in  \bigoplus_{v|\ell} \rH^1(k_v, \Q_\ell(-1))\otimes_{\Q_\ell}\sO
 \]
 lies in orthogonal complement of $\rK_{3}(\O_k)\otimes_{\Z_\ell}\sO$ under the semi-local duality pairing above.
 Hence, in this case we can view $\CS(\rho)$ as a linear functional
 \[
\CS(\rho): \frac{\bigoplus_{v|\ell} \rH^1(k_v, \Q_\ell(2))}
 {\rK_{3}(\O_k)\otimes_{\Z_\ell}\Q_\ell}\to \sO.  
 \] \endproof
 \end{prop}

 \begin{Remark}
{\rm   At this point, we have no explicit calculations and no proof that the volume is not identically zero. For $k$ a number field, we can obtain examples by taking $X$ to be a Shimura curve and $\rho$   the $\ell$-adic local system of the Tate module of 
a universal abelian scheme over $X$. It is an interesting problem to calculate $\CS({\rho})$ for these examples.
}
\end{Remark}

\subsection{Variant: Finite groups and higher dimension.} Here, we let $G$ be a finite group and give a construction of classes in $\rH^1(k, \Q_\ell/\Z_\ell(-1))$ which is more in the spirit of the construction in \cite{KimCS1}. If $\pi: Y\to X$ is an \'etale $G$-cover
(corresponding to $\rho: \pi_1(X)\to G$), we obtain  a homomorphism
\[
\mathfrak K(\pi): \rH^3(G, \Q_\ell/\Z_\ell)\to \rH^3(X, \Q_\ell/\Z_\ell)
\] 
by pulling back from the classifying space. For $\alpha\in \rH^3(G, \Q_\ell/\Z_\ell)$ we can now set
\[
{\rm CS}(Y/X, \alpha):=\eta(\mathfrak K(\pi)(\alpha))\in \rH^1(k, \Q_\ell/\Z_\ell(-1))
\]
where
$
\eta: \rH^3(X, \Q_\ell/\Z_\ell)\to \rH^1(k, \Q_\ell/\Z_\ell(-1))
$
is obtained from the Leray-Serre spectral sequence. 
This can also be given an explicit cocycle description: Let us choose 
\[
\ti\alpha: \bar C_3(G)/{\rm Im}(\partial_4)\to \Q_\ell/\Z_\ell
\]
giving $\alpha\in \rH^3(G, \Q_\ell/\Z_\ell)={\rm Hom}(\rH_3(G,\Z), \Q_\ell/\Z_\ell)$.
Then, for $c$, $\sigm$  as before, and $\delta(\sigm, c)\in \bar C_3(\pi_1(\bar X))/{\rm Im}(\partial_4)_\ell$, with
\[
\partial_3(\delta(\sigm, c))=\sigm \cdot c\cdot \sigm^{-1}-  \chi_{\rm cycl}(\sigma)\cdot c,
\]
we can take
\[
\sigma\mapsto \chi_{\rm cycl}(\sigma)^{-1}\ti\alpha [\rho(\delta(\sigm, c))-F_{\rho(\ti\sigma)}(\rho(c))]\in \Q_\ell/\Z_\ell.
\]

\subsubsection{}\label{441}  Consider now a continuous $\rho: \pi_1(X)\to \GL_d(A)$ with $A_n=A/\frakm^n$ finite, for each $n\geq 1$. We can apply the construction above to $\rho_n: \pi_1(X)\to \GL_d(A_n)$. We obtain
\[
\eta\cdot \frak K(\rho_n): \rH^3(\GL_d(A_n), \Q_\ell/\Z_\ell)\to \rH^1(k, \Q_\ell/\Z_\ell(-1)).
\]
For each $n\geq 1$, the diagram
\[
\begin{matrix}
\rH^3(\GL_d(A_n), \Q_\ell/\Z_\ell) &\xrightarrow{\eta\cdot \mathfrak K(\rho_n)} &\rH^1(k, \Q_\ell/\Z_\ell(-1))\\
{\rm Infl}\downarrow && \downarrow{\rm id} \\
\rH^3(\GL_d(A_{n+1}), \Q_\ell/\Z_\ell)&\xrightarrow{\eta\cdot \mathfrak K(\rho_{n+1})} &\rH^1(k, \Q_\ell/\Z_\ell(-1))
\end{matrix}
\]
is commutative and we obtain
\[
\eta\cdot \mathfrak K(\rho): \rH^3_{\rm cts}(\GL_d(A), \Q_\ell/\Z_\ell)\cong \varinjlim\nolimits_n \rH^3(\GL_d(A_n), \Q_\ell/\Z_\ell) \xrightarrow{\ \ } \rH^1(k, \Q_\ell/\Z_\ell(-1)).
\]
When $d\geq 2$, we can recover the previous construction   after taking Pontryagin duals and composing with the $\ell$-adic regulator. 

\subsubsection{} More generally, suppose that $f: X\to \Spec(k)$ is a smooth proper variety of dimension $n$ over the field $k$ and $\ell$ a prime different from the characteristic of $k$. We can then consider the push-down homomorphism
\[
\eta ={\rm R}^{2n}f_{\et, *}: \rH^{2n+1}(X, \Q_\ell/\Z_\ell)\to \rH^1(k,\Q_\ell/\Z_\ell(-n)).
\]
Similarly, we have 
\[
\eta_{ \Q_\ell}={\rm R}^{2n}f_{\et, *}: \rH^{2n+1}(X, \Q_\ell)\to \rH^1(k,\Q_\ell(-n)).
\]

Suppose $G$ is a finite group. If $\pi: Y\to X$ is an \'etale $G$-cover
we obtain  a homomorphism
\[
\mathfrak K(\pi): \rH^{2n+1}(G, \Q_\ell/\Z_\ell)\to \rH^{2n+1}(X, \Q_\ell/\Z_\ell)
\] 
by pulling back from the classifying space. For $\alpha\in \rH^{2n+1}(G, \Q_\ell/\Z_\ell)$ we set
\[
{\rm CS}(Y/X, \alpha):=\eta_n(\mathfrak K(\pi)(\alpha))\in \rH^1(k, \Q_\ell/\Z_\ell(-n)).
\]

Recall that we have (cf. \cite{Wagoner})
\[
\rH^{2n+1}_{\rm cts}(\GL_d(\Z_\ell), \Q_\ell)=(\varprojlim_n(
\varinjlim_s\rH^{2n+1}(\GL_d(\Z/\ell^s\Z), \Z/\ell^n\Z)))\otimes_{\Z_\ell}\Q_\ell.
\]
If $\sF$ is an \'etale $\Z_\ell$-local system on $X$ of rank $d\geq 2$ we obtain
\[
\mathfrak K(\sF)_{\Q_\ell}: \rH^{2n+1}_{\rm cts}(\GL_d(\Z_\ell), \Q_\ell) 
\to \rH^{2n+1}(X, \Q_\ell)
\]
from the corresponding system of $\GL_d(\Z/\ell^s\Z)$-covers as before.
For each $d'>d$, the local system $\sF$ gives the local system $\sF'=\sF\oplus\Z_\ell^{d'-d}$ of rank $d'$. For $d'>>0$, the $2n+1$-th $\ell$-adic regulator $\be_{n,\ell}$ is a non-trivial element of the $\Q_\ell$-vector space
$
  \rH^{2n+1}_{\rm cts}(\GL_{d'}(\Z_\ell ), \Q_\ell)
$
(by stability and \cite[Prop. 1]{Wagoner}).
We can now define 
\[
\CS({\sF})\in \rH^1_{\rm cts}(k, \Q_\ell(-n)) 
\]
to be given by value of the composition $\eta_{\Q_\ell}\cdot \mathfrak K(\sF')_{\Q_\ell}$ at $\be_{n, \ell}$.

\medskip

\section{Deformations and lifts}\label{deflifts}

Here, we apply our constructions to universal (formal) deformations of a modular representation of the \'etale  fundamental group
of a curve. In particular, we explain how the work in Section \ref{sect2} can be using to provide a symplectic structure
on the formal deformation space of a modular representation, provided the deformation is unobstructed.
  
Again, we omit denoting our choice of base point and simply write $\pi_1(X)$ and $\pi_1(\bar X)$.

\subsection{Lifts}\label{ss:lifts}
Fix a continuous representation $\rho_0:  \pi_1(X)\to \GL_d(\BF)$ with $\BF$ a finite field of characteristic $\ell\neq 2$. Suppose $\ep: \pi_1(X)\to   \O^\times$ is a character so that $\ep\,{\rm mod}\, \frakm=\det(\rho_0)$. 
We will denote by $\bar\rho_0$, resp. $\bar\ep$, the restrictions of $\rho_0$, resp. $\ep$, to the geometric fundamental group $\pi_1(\bar X)\subset \pi_1(X)$. 

Denote by ${\mathcal C}_\O$ the category of complete Noetherian local $\O $-algebras $A$ together with an isomorphism $\alpha: A/\frak m_A\xrightarrow{\sim} \BF$.

\begin{lemma} (Schur's Lemma, \cite{Mazur} Ch. II, \S 4, Cor.) Let $\bar\rho: \pi_1(\bar X)\to \GL_d(A)$ be a continuous representation with $A\in {\mathcal C}_\O$. If the associated residual representation $\bar\rho_0$ is absolutely irreducible, any matrix in ${\rm M}_d(A)$ which commutes with all the elements in the image of $\bar\rho$ is a scalar. \endproof
\end{lemma}

In what follows, we always assume that 
\[
\bar\rho_0:  \pi_1(\bar X)\to \GL_d(\BF)
\]
is absolutely irreducible, i.e. it is irreducible as an $\bar \BF$-representation.

Let $\bar\rho: \pi_1(\bar X)\to \GL_d(A)$ be a continuous representation with $A\in {\mathcal C}_\O$ which lifts $\bar\rho_0$ and with $\det(\bar\rho)=\bar \ep$. Suppose that for all $g\in \pi_1(X)$, there is $h_g\in \GL_d(A)$
with 
\[
\bar\rho(g\gamma g^{-1})=h_g \bar\rho(\gamma)h_g^{-1},\quad \forall \gamma\in \pi_1(\bar X).
\]
By Schur's lemma above, $h_g$ is uniquely determined up to a scalar in $A^\times$
and
\[
h_{gg'}=z(g,g') h_g h_{g'}, \quad z(g, g')\in A^\times.
\]
Mapping $g$ to $\pi(h_g)=h_g\,{\rm mod}\, A^\times$ gives a homomorphism
\[
\rho_{\rm PGL}: \pi_1(X)\to {\rm PGL}_d(A)
\]
which extends  $\pi_1(\bar X)\xrightarrow{\bar\rho} \GL_d(A)\xrightarrow{\pi} {\rm PGL}_d(A)$.
We can see that $\rho_{\rm PGL}$ is continuous for the profinite topologies on 
$\pi_1(X)$ and ${\rm PGL}_d(A)$. 

The following will be used in the last section. 

\begin{prop}\label{Mackeyprop} Suppose $\ell$ does not divide $d$.
Under the above assumptions, there is a  lift of $\rho_{\rm PGL}$ to a continuous representation
\[
\rho: \pi_1(X)\to {\rm GL}_d(A)
\]
such that $\det(\rho)=\ep$ and $\rho_{|\pi_1(\bar X)}=\bar\rho$.
\end{prop}

\begin{proof}
A version of this is well-known but we still provide the details for completeness.
To give such a lift we need to choose, for each $g\in \pi_1(X)$, $h_g\in \GL_d(A)$ such that: 
\begin{itemize}
\item[a)]  $\bar\rho(g\gamma g^{-1})=h_g \bar\rho(\gamma)h_g^{-1}$, $\forall \gamma\in \pi_1(\bar X)$, $\forall g\in \pi_1(X)$,

\item[b)]  $\det(h_g)=\ep(g)$, $\forall g\in \pi_1(X)$,

\item[c)] $\rho(\gamma)=\bar\rho(\gamma)$, $\forall \gamma\in \pi_1(\bar X)\subset \pi_1(X)$,

\item[d)] $h_{gg'}=h_gh_{g'}$, i.e. $z(g, g')=1$, $\forall g, g'\in \pi_1(X)$,

\item[e)] $g\mapsto h_g$ is continuous.

\end{itemize}

For each $g\in \pi_1(X)$, consider the set 
\[
Y_g(A)=\{ h\in \GL_d(A)\ |\ \pi(h)=\pi(h_g),\ \det(h)=\ep( g )\}.
\]
There is a simply transitive action of $\mu_d(A)=\{a\in A^\times\ |\ a^d=1\}$ on $Y_g(A)$.
The existence of $\rho_0: \pi_1(X)\to \GL_d(\BF)$ implies that $Y_g(\BF)$  is not empty
since it contains $\rho_0(g)$. Then $Y_g(\BF)\simeq \mu_d(\BF)$. 

Recall we assume ${\rm gcd}(d,\ell)=1$. Hensel's lemma implies that 
$\mu_d(A)\to \mu_d(\BF)$ given by reduction modulo $\frakm_A$ 
is an isomorphism. Now consider the map 
\[
Y_g(A)\to Y_g(\BF)
\]
given by reduction modulo $\frakm_A$.
Pick $h'\in \GL_d(A)$ with $\pi(h')=[h_g]\in {\rm PGL}_d(A)$ and $h'\to \rho_0(g)\in Y_g(\BF)$, and write $h=a h'$, $a\in 1+\frakm_A\subset A^\times$.
We want to choose $a$ so that $\det(h)=\ep(g)$, i.e. $a^d\det(h')=\ep(g)$. We have $\det(\bar h')=\bar \ep(g)\in \BF^\times$, so $\det(h')\ep(g)^{-1}\in 1+\frakm_A$ and  $a^d=\det(h')\ep(g)^{-1}$ has a solution since ${\rm gcd}(d,\ell)=1$.
This shows that $Y_g(A)$ is also non-empty. Hence reduction modulo $\frakm_A$ 
gives a bijection 
\[
Y_g(A)\simeq Y_g(\BF)\simeq \mu_d(\BF).
\]
We can now choose $h_g\in Y_g(A)$ to be the unique element whose reduction is $\rho_0(g)$. Then $g\mapsto h_g$ satisfies properties (a), (b) and by comparing with the reduction, properties (c), (d) and (e).
In fact, we see that the lift $\rho$ given by $\rho(g)=h_g$ reduces to $\rho_0$ modulo $\frakm_A$.
 \end{proof}

\begin{Remark} {\rm We keep the assumptions of the proposition above.

a) The map $z: \pi_1(X)\times \pi_1(X)\to A^\times$  given by $(g, g')\mapsto z(g, g')$ is a $2$-cocycle.
A classical argument (see for example \cite{Mackey} Thm 8.2) applies to show that $z$ is the inflation of a continuous $2$-cocycle
\[
\nu: \Gk\times \Gk\to \mu_d(A)=\mu_d(\BF).
\]
We can see that the existence of $\rho_0$ which lifts $\rho_{\rm PGL}\, {\rm mod}\, \frakm_A$
implies that the class $[\nu]\in \rH^2(k, \mu_d(\BF) )$ vanishes.
This provides an alternative point of view of the proof.

b) The lift $\rho$ given by the proof of the proposition reduces to $\rho_0$ modulo $\frakm_A$.

c) The lift $\rho$ is not unique. Consider a character 
$\chi: G_k\to \mu_d(A)=\mu_d(\BF)$.
The twist  $ \rho\otimes_A \chi$ satisfies all the requirements of 
the proposition and we can easily see that all representations that satisfy these requirements
are such twists of each other. }
\end{Remark}

\subsection{Universal deformation rings}\label{deform}

Following \cite{deJong} \S 3, we now consider the deformation functors  
\[
  {\rm Def}(\pi_1( X),  \rho_0,  \ep), \quad  {\rm Def}(\pi_1(\bar X), \bar\rho_0, \bar\ep).
\]
By definition, ${\rm Def}(\pi_1(  X),  \rho_0,  \ep)$ is the functor from
$\calC_\O$ to Sets which maps $(A, \alpha)$ to the set of equivalence classes of continuous representations 
\[
\rho_A: \pi_1(X)\to \GL_d(A)
\]
such that $\alpha(\rho_A\, {\rm mod}\, \frakm_A)= \rho_0$, $\det(\rho_A)=(\O^\times\to A^\times)\cdot \ep$. Here, $\rho_A$ is equivalent to $\rho'_{A}$ if and only if there exists an element 
$g\in \GL_n(A)$ such that  $\rho'_A(\gamma)=g^{-1}\rho_A(\gamma)g$,
for all $\gamma\in \pi_1( X)$. The functor $ {\rm Def}(\pi_1(\bar X), \bar\rho_0, \bar\ep)$ is defined similarly.

Under our condition that $\bar\rho_0$ is absolutely irreducible,  ${\rm Def}(\pi_1(\bar X), \bar\rho_0, \bar\ep)$
is representable in the category ${\mathcal C}_\O$ and there is a universal pair   $(\bar A_{\rm un}, \bar \rho_{\rm un})$. (This follows by applying Schlessinger's criteria, see loc. cit. 3.2 and \cite{Mazur2}, Sect. 1.2.
We use here that $\pi_1(\bar X)$ is topologically finitely generated). If $\pi_1( X)$ is also topologically finitely generated, as it happens when $k$ is a finite field, then ${\rm Def}(\pi_1(X), \rho_0, \ep)$
is also representable in the category ${\mathcal C}_\O$ and there is also a universal pair   $(A_{\rm un},  \rho_{\rm un})$.

As in \cite{deJong} 3.10,   we have $A_{\rm un}\simeq \O\lps t_1,\ldots , t_m\rps$ for some $m$. The formal smoothness statement holds because the obstruction group
\[
\rH^2(\pi_1(\bar X), {\rm Ad}_{\bar\rho_0}^0(\BF)) 
\]
vanishes, see loc. cit. for details. 

\subsection{Galois action on the deformation rings}\label{galoisactiondeform}
 For every $\sigma\in \Gk$, we give an automorphism $\phi(\sigma): \bar A_{\rm un}\to \bar A_{\rm un}$ as in \cite{deJong}  3.11:  
 
 For simplicity, we drop the subscript $\rm un$ and write $\bar A=\bar A_{\rm un}$ etc. Choose an element $\ti \sigma \in \pi_1(X)$ which maps to   $\sigma\in \Gk$, and $h\in \GL_d(\bar A _{\rm un})$ such that $h\, {\rm mod}\, \frakm_{\bar A}=\rho_0(\tilde \sigma)$. Consider the ``twisted" representation
\[
\bar\rho^{\ti\sigma}: \pi_1(\bar X)\to \GL_d(\bar A),\quad \gamma\mapsto h\bar\rho(\ti\sigma^{-1}\gamma \ti\sigma)h^{-1}.
\]
We have $\bar\rho^{\ti\sigma}\, {\rm mod}\, \frakm_{\bar A}=\bar\rho_0$, and $\det(\bar\rho^{\ti\sigma})=\bar\ep$. Hence, $(\bar A, \bar\rho^\sigm)$ is a deformation of $\bar\rho_0$ with determinant $\bar\ep$. By the universal property of $(\bar A, \bar\rho)$ we obtain
a $\O$-algebra homomorphism $\phi: \bar A\to \bar A$ and $h'\in \GL_d(\bar A)$ such that 
\begin{equation}\label{twist}
\phi(\bar\rho(\gamma))=h'\bar\rho^{\ti\sigma}(\gamma)h'^{-1},
\end{equation}
 for all $\gamma\in \pi_1(\bar X)$. The above combine to
\begin{equation}\label{twist2}
\bar\rho(\ti\sigma^{-1}\gamma \ti\sigma)=h_1\phi(\bar\rho(\gamma))h_1^{-1}.
\end{equation}
The automorphism $\phi(\sigma)$ is independent of the choice of $\ti\sigma$ lifting $\sigma$ and of the element $h$ as above. Indeed, if $\ti\sigma'$, $h'$ is another choice giving $\phi'$, then $\ti\sigma'= \ti\sigma\cdot \delta$, for $\delta\in \pi_1(\bar X)$ and we can easily see that $\bar\rho^{{\ti\sigma}'}$ is equivalent to $\bar\rho^{{\ti\sigma}}$ and $\phi(\bar\rho)$ is equivalent to $\phi'(\bar\rho)$. Hence, the two maps $\phi'$, $\phi: \bar A\to \bar A$ agree by the universal property of $(\bar A, \bar\rho)$.
It now easily follows also that
\[
\phi(\sigma\sigma')=\phi(\sigma)\phi(\sigma')
\]
for all $\sigma$, $\sigma'\in \Gk$. 

\begin{prop}\label{contAction}
The homomorphism $\phi: \Gk\to {\rm Aut}_{\O}(\bar A)$ is continuous where ${\rm Aut}_{\O}(\bar A)$ has the profinite topology given by the finite index normal subgroups $\calK_n=\ker({\rm Aut}_\O(\bar A)\to {\rm Aut}_\O(\bar A/\frakm^n))$. 
  \end{prop}
    
  \begin{proof} It is enough to show
  that given $n\geq 1$, there is a finite Galois extension $k'/k$ such that if $\sigma\in U=G_{k'}$,
 then $\phi(\sigma)\in \calK_n$. Since $\rho: \pi_1(\bar X)\to \GL_d(\bar A)$ is continuous, there is $m$ 
  such that if $\gamma\in \Gamma_m$, then $\rho(\gamma)\in 1+{\rm M}_d(\mathfrak m_A^n)$. 
  Here $\Gamma_m\subset \Gamma$ is the characteristic finite index subgroup of $\pi_1(\bar X)$ as before. Let $Y_m\to \bar X$ be the corresponding $\Gamma/\Gamma_m$-cover which is the base change of a $\Gamma/\Gamma_m$-cover $Y'_m\to X\otimes_k k^{\rm sep}$. We can write $ k(Y'_m)= k^{\rm sep}(X)(\alpha)$, for the extension of function fields, where $k^{\rm sep}( X)=k(X)\otimes_k k^{\rm sep}$. Then, we  have
$\Gamma/\Gamma_m\simeq {\rm Gal}(k^{\rm sep}(X)(\alpha)/k^{\rm sep}(X))$. Choose a finite Galois extension $k\subset k'\subset k^{\rm sep}$ which contains all the coefficients of the minimal polynomial of $\alpha$ over $k^{\rm sep}(X)$ and with  $X(k')\neq\emptyset $. Then, there is a continuous section $s: \Gk\to \pi_1(X)$ such that, if $\sigma\in G_{k'}$, then conjugation by $s(\sigma)$ is trivial on $\Gamma/\Gamma_m\simeq {\rm Gal}(k^{\rm sep}(X)(\alpha)/k^{\rm sep}(X))$.
 We now have
 \[
 \rho(s(\sigma)\gamma s(\gamma)^{-1})=\rho(\gamma\gamma_m)\equiv \rho(\gamma)\, {\rm mod}(\mathfrak m_A^n).
 \]
By the definition of $\phi(\sigma)$ and the universal property of $(\bar A, \bar\rho)$, this gives that $\phi(\sigma)\equiv\, {\rm Id}\, {\rm mod}\, (\frakm_A^n)$, so $\phi(\sigma)\in \calK_n$.
\end{proof}

\begin{prop}\label{fixedprop}
Suppose that $A\in\mathcal C_\O$ and let $\bar\rho:\pi_1(\bar X)\to \GL_d(A)$
be a deformation of $\bar\rho_0$ with determinant $\ep$ which corresponds to
$f: \bar A_{\rm un}\to A$. 

a) If $\bar\rho$ extends to a representation $\rho:\pi_1(X)\to \GL_d(A)$ with determinant $\ep$,
then  
$f\cdot \phi(\sigma)=f$, for all $\sigma\in \Gk$.
 
b) Conversely suppose $f\cdot \phi(\sigma)=f$,  for all $\sigma\in \Gk$, and ${\rm gcd}(\ell, d)=1$.
Then $\bar\rho:\pi_1(\bar X)\to \GL_d(A)$ extends to a representation $\rho:\pi_1(X)\to {\GL}_d(A)$  which deforms $\rho_0$ and has determinant $\ep$.
\end{prop}

\begin{proof}
(a) Suppose $\bar\rho$ extends
to   $\rho$. Then, we have
\[
\bar\rho^{\ti\sigma}(\gamma)=h\bar\rho(\ti\sigma \gamma\ti\sigma^{-1})h^{-1}=h\rho(\ti\sigma)\bar\rho( \gamma)\rho(\ti\sigma)^{-1}h^{-1}.
\]
This gives that $\bar\rho^{\ti\sigma}$ is equivalent to $\bar\rho$. The representability of the deformation problem now implies $f\cdot \phi(\sigma)=f$. 

(b) Conversely, suppose that $f\cdot \phi(\sigma)=f$, for all $\sigma\in \Gk$. Then, for $g\in \pi_1(X)$ which maps to $\sigma\in \Gk$, we have
\[
 \bar\rho(g\gamma g^{-1})=h_{g}\phi(\sigma)(\bar\rho(\gamma))h_{g}^{-1}=h_{g}\bar\rho(\gamma)h_{g}^{-1}.
\]
for some $h_g\in \GL_d(A)$. The result now follows from Proposition \ref{Mackeyprop}.
\end{proof}

\subsection{Volume for the universal deformation}

Start with $\rho_0: \pi_1(X)\to \GL_d(\BF)$ 
such that $\bar\rho_0: \pi_1(\bar X)\to \GL_d(\BF)$  is absolutely irreducible
and consider
\[
\bar\rho:=\bar\rho_{\rm un}: \pi_1(\bar X)\to \GL_d(\bar A_{\rm un})
\]
the universal deformation. Set 
\[
\calD={\rm Spf}(\bar A_{\rm un})[1/\ell]
\]
for the rigid analytic fiber of the formal scheme $\bar A_{\rm un}$ over $\O$. This is a rigid analytic space over $E$ which is non-canonically isomorphic to the open unit polydisk $D_1(m)$.

Apply the construction with $\Gamma=\pi_1(\bar X)$, $A=\bar A_{\rm un}$,
and $G=\Gk$ mapping to ${\rm Out}(\Gamma)$ via the exact sequence and acting on $\bar A_{\rm un}$ via $\phi$
as above. Take $\be_A$ to be given by the $\ell$-adic regulator which now takes values in $\sO(\calD)\simeq\sO$.    We obtain a continuous cohomology class 
\[
\CS(\bar\rho)\in \rH^1_{\rm cts}(k, \sO(\calD)(-1))
\]
where the action of $\Gk$ on $\sO(\calD)(-1)$ is via $\chi_{\rm cycl}^{-1}\cdot \phi$.

\subsection{The symplectic structure on the deformation space}

We continue with the above assumptions and notations. Set $\bar A=\bar A_{\rm un}$, $\bar A_n=\bar A_{\rm un}/\frak m_{\bar A}^n$. 

Consider the $d+1$-dimensional representation $\bar \rho_+:= \bar\rho\oplus \ep^{-1}$ of $\pi_1(\bar X)$ which has trivial determinant. The constructions of \S 2 apply to $\bar \rho_+$
and $\Gamma=\pi_1(\bar X)$. 
We obtain cohomology classes
\[
\kappa_n\in \rH^2(\pi_1(\bar X), \rK_2(\bar A_n)),\quad \omega_n\in \rH^2(\pi_1(\bar X), \Omega^2_{\bar A_n}).
\]
Recall that, for all $n\geq 1$, $\rK_2(\bar A_n)$ and $\Omega^2_{\bar A_n}$ are finite groups. 
They are both $\ell$-groups: This is visibly true for $\Omega^2_{\bar A_n}$. To show the same statement for $\rK_2(\bar A_n)$ observe that the kernel of $\rK_2(\bar A_n)\to \rK_2(\BF)$
is generated by Steinberg symbols of the form $\{1+\ell x, s\}$; these are $\ell$-power torsion since $(1+\ell x)^{\ell^N}=1$ in $\bar A_n$ for $N>n$, while $\rK_2(\BF)=(0)$
(\cite{Quillen}).

Assume now that $X$ is, in addition, projective.
Since $\bar X$ is ${\rm K}(\pi_1(\bar X), 1)$ (cf. \cite[Theorem 11]{Friedlander})  we have a canonical isomorphism
\begin{equation}\label{Hurewitz}
\calH: \rH^2(\pi_1(\bar X), \Omega^2_{\bar A_n})\xrightarrow{\simeq} \rH^2_{\et}( \bar X , \Omega^2_{\bar A_n}),
\end{equation}
By Poincare duality,
\[
{\rm Tr}: \rH^2_{\et}( \bar X , \Omega^2_{\bar A_n}) \xrightarrow{\simeq} \Omega^2_{\bar A_n}(-1), 
\]
and similarly for  $\rK_2(\bar A_n)$. Set 
\[
 \kappa_{a, n}:=({\rm Tr}\circ {\calH})(\kappa_n)\in \rK_2(\bar A_n)(-1),\quad \omega_{a, n}:
 =({\rm Tr}\circ {\calH})(\omega_n)\in \Omega^2_{\bar A_n}(-1).
\]
(compare \S \ref{sss:def}).
Also set
\[
\quad \kappa:=\varprojlim\nolimits_n \kappa_{a, n}\in \varprojlim\nolimits_n
\rK_2(\bar A_n)(-1),\quad
\omega:=\varprojlim\nolimits_n \omega_{a, n}\in \hat\Omega^2_{\bar A_{\rm un}/\O}(-1).
\]
Set
\[
T_{\bar A_{\rm un}/\O}={\rm Hom}_{\bar A_{\rm un} }(\hat\Omega_{\bar A_{\rm un}/\O}, \bar A_{\rm un} ),\quad 
T_{\bar A_n}=T_{\bar A_{\rm un}/\O}\otimes_{\bar A_{\rm un}}\bar A_n={\rm Hom}_{\bar A_n}(\hat\Omega_{\bar A_{\rm un}/\O}, \bar A_n).
\]
By \cite[\S 17, \S 21 Prop. 1, \S 24]{Mazur}, there are natural $\bar A_n$-isomorphisms
\[
T_{\bar A_n} \cong \rH^1(\pi_1(\bar X), {\rm Ad}^0_{\bar\rho}(\bar A_n)).
\]
For simplicity, set $W_n={\rm Ad}^0_{\bar\rho}(\bar A_n)$; this is a finite free $\bar A_n$-module given by trace zero matrices. The profinite group $\Gamma=\pi_1(\bar X)$ satisfies Poincare duality in dimension $2$ over the $\ell$-power torsion $\bar A_n$ as in \S \ref{poincaredual}. Then, cup product followed by $W_n\otimes_{\bar A_n} W_n\to \bar A_n$, $(X, Y)\mapsto {\rm Tr}(XY)$, and combined with Poincare duality gives the pairing
\[
\langle\ ,\ \rangle_n: T_{\bar A_n}\times T_{\bar A_n}=\rH^1(\pi_1(\bar X), W_n)\times \rH^1(\pi_1(\bar X), W_n)\to \bar A_n.
\]
Suppose $\ell$ does not divide $d$.
Then, this is a non-degenerate $\bar A_n$-linear pairing. Taking an inverse limit over $n$ gives
\[
\langle\ ,\ \rangle: T_{\bar A_{\rm un}/\O}\times T_{\bar A_{\rm un}/\O}\to \bar A_{\rm un} 
\]
 which is a non-degenerate (perfect) $\bar A_{\rm un} $-linear pairing.
 
\begin{thm}\label{prop:sympl}
1) The $2$-form $\omega\in \hat\Omega^2_{\bar A_{\rm un}/\O}(-1)$ is closed. 

2) Suppose $\ell$ does not divide $d$. For all   $v_1$, $v_2\in T_{\bar A_{\rm un}/\O} $, we have
\[
\langle v_1 , v_2\rangle =\omega (v_1, v_2)
\]
and $\omega$ is non-degenerate.
\end{thm}

\begin{proof} 
For simplicity, set $\bar A=\bar A_{\rm un}$. By construction $\omega=d\log(\kappa)$ and so $d\omega=0$, i.e. $\omega$ is closed which shows (1). Let us show (2). For every $n\geq 1$,  $v_i\in T_{\bar A_{\rm un}/\O} $ give deformations 
$\bar \rho_i$ of $ \bar\rho_0$ over $\bar A_n[\varepsilon]$, with determinant $\ep$.
These give representations
\[
 \bar \rho_{+, i}:\pi_1(\bar X)\to \SL_{d+1}(\bar A_n[\varepsilon]), \quad \bar \rho_{+, i}:=\bar \rho_i\oplus\ep^{-1}.
\]
Set $W^+_n={\rm Ad}^0_{\bar \rho_+}(\bar A_n)={\rm M}^0_{(d+1)\times (d+1)}(\bar A_n)$ which contains $W_n=
{\rm Ad}^0_{\bar \rho}(\bar A_n)$.  
By our construction of $\omega_{a, n}$ and (\ref{Goldman1}), $\omega_{a, n}(v_1, v_2)$ is the value of the pairing
\[
\rH^1(\pi_1(\bar X), W^+_n)\times \rH^1(\pi_1(\bar X), W^+_n)\to \rH^2(\pi_1(\bar X), \bar A_n)\cong \bar A_n
\]
given by cup product followed by $W^+_n\otimes_{\bar A_n} W^+_n\to \bar A_n$, $(X, Y)\mapsto {\rm Tr}(XY)$,
at $(v_1, v_2)$.  The cocycles of $\pi_1(\bar X)$ on $W^+_n$ corresponding to $\bar\rho_{+, i}$ factor through $W_n\subset W^+_n$. It follows that $\omega_{a, n}(v_1, v_2)$ is also the value of the pairing
\[
\rH^1(\pi_1(\bar X), W_n)\times \rH^1(\pi_1(\bar X), W_n)\to \rH^2(\pi_1(\bar X), \bar A_n)\cong \bar A_n
\]
at $(v_1, v_2)$. Part (2) now follows. \end{proof}
\smallskip

In what follows, we assume without further mention, that $\ell$ does not divide $d$.

The form $\omega$ gives, by definition, the {\bf canonical   symplectic structure} on the formal deformation space $\Spf(\bar A_{\rm un})$. 

\subsubsection{} Recall $\calD={\rm Spf}(\bar A_{\rm un})[1/\ell]$. The form $\omega$ gives a Poisson structure on $\sO(\calD)\simeq \sO $ as follows.
For $f\in \sO(\calD)$ set $X_f$ for the analytic vector field 
on $\calD$ defined by  
\[
X_f\intprod \omega=df.
\]
(Here and in what follows we denote by $i_X(\omega)$ or $X\intprod \omega$ for the contraction, or ``interior product",
of the vector field $X$ with the form $\omega$.). The analyticity of $X_f$ can be seen as follows: Choose an isomorphism 
$\bar A_{\rm un}\simeq \O\lps x_1,\ldots , x_m\rps$ and write
$\omega=\sum_{i<j}g_{ij}dx_i\wedge dx_j$. 
Then, in the basis $\partial/\partial x_i$,  $X_f$ is (formally) the image of the vector $-(f_1,\ldots, f_m)$ under the map given by the matrix $(g_{ij})$.  Since $g_{ij}\in 
\O\lps x_1,\ldots , x_m\rps$, we see that if $f\in \sO(\calD)$ all the components of $X_f$ converge on $||\x||<1$, i.e. they belong to $\sO$.

We now set
\begin{equation}\label{goldmanbracket}
\{ f, g\}=\omega(X_f, X_g).
\end{equation}
We can easily see that $\{f, g\}$ takes values in $\sO(\calD)$. Also
\[
\{\ ,\ \}: \sO(\calD)\times\sO(\calD)\to \sO(\calD)
\]
is a Lie bracket, i.e. satisfies $\{f, g\}=-\{g, f\}$ and
\[
\{f, \{g, h\}\}+\{g, \{h, f\}\}+\{h, \{f, g\}\}=0.
\]
It also satisfies the Leibniz rule $\{fg, h\}=f\{g, h\}+g\{f, h\}$.
Indeed, it is enough to show these identities in the ring of formal power series. There they 
are true by the standard arguments. (The Jacobi identity follows from the closedness
of the form $\omega$.)

\medskip

\section{The symplectic nature of the  Galois action}

We now return to the Galois action on the formal deformation space of a modular representation of the arithmetic \'etale fundamental group of a curve. 
We construct the $\ell$-adic Galois group flow and explain its interaction with the canonical symplectic form. Finally we show that the set of deformed representations that extend to a representation of the fundamental group of the curve over a finite extension of $k\cy$ is the intersection of the critical loci for a set of rigid analytic functions.

\subsection{Galois action and the symplectic form}  

We continue with the assumptions and notations of \S \ref{deform}, \S \ref{galoisactiondeform}.

\begin{prop} 
 We have
 \[
\phi(\sigma)(\omega )=\chi^{-1}_{\rm cycl}(\sigma)\cdot \omega,
\]
where $\phi(\sigma): \bar A_{\rm un}\to \bar A_{\rm un}$ is the automorphism induced by $\sigma\in G_k$
as in \S \ref{galoisactiondeform}.
 \end{prop}

\begin{proof} This also follows from  Theorem \ref{prop:sympl} 
which gives a description of $\omega$ using  cup product and Poincare duality.
Here is a more direct argument that also applies to the $\rK_2$ invariant. Consider the  endomorphism $[\ti\sigma]$ of $ \rH^2(\pi_1(\bar X), \Omega^2_{\bar A_n})$ induced by $\gamma\mapsto \ti\sigma\gamma \ti\sigma^{-1}$ on $\pi_1(\bar X)$. 
By the construction of $\omega_n$ and the definition of $\phi$, we have
\begin{equation} \label{freq1}
[\ti\sigma](\omega_n) =\phi(\sigma)(\omega_n)
\end{equation}
in $ \rH^2(\pi_1(\bar X), \Omega^2_{\bar A_n} )$, where on the right hand side $\phi(\sigma)$ is applied to the coefficients $\Omega^2_{\bar A_n}$.
Next observe that, by functoriality of $\calH$ (\ref{Hurewitz}), the endomorphism $[\ti\sigma]$ of 
$\rH^2(\pi_1(\bar X), \Omega^2_{\bar A_n})$ corresponds to   the  endomorphism $(\ti\sigma)^*$
on $\rH^2_{\et}( \bar X , \Omega^2_{\bar A_n})$, i.e.
\begin{equation}\label{freq2}
\calH\circ [\ti\sigma]=(\ti\sigma)^*\circ \calH.
\end{equation}
By Poincare duality for \'etale cohomology,  ${\rm Tr}\circ (\ti\sigma)^* $ is multiplication by $\chi^{-1}_{\rm cycl}(\sigma)$.
Combining the above gives
\begin{equation*}
\phi(\sigma)(\omega_{c, n})=\chi^{-1}_{\rm cycl}(\sigma)\cdot \omega_{c, n}. 
\end{equation*}
  This then implies $\phi(\sigma)(\omega )=\chi^{-1}_{\rm cycl}(\sigma)\cdot \omega$, as desired.
\end{proof}

\begin{Remark}
{\rm a) The stronger statement $\phi(\sigma)(\kappa )=\chi^{-1}_{\rm cycl}(\sigma)\cdot \kappa$ is also true. This can be seen by repeating the argument in the proof   but with the coefficient group $\Omega^2_{\bar A_n}$  replaced by $\rK_2({A_n})$.

b)  (suggested by D. Litt) Consider the ring $\Z_\ell\lps \lambda \rps$ (with $\lambda$ a formal variable),
on which the group $G_k$ acts by $\sigma(\lambda)= \chi_{\rm cycl}(\sigma)\cdot\lambda$. Set $\Q_\ell\{\lambda\}:=(\varprojlim_n \Z_\ell\llps \lambda\lrps/\ell^n)[1/\ell]$ for the $\ell$-adic completion
of the Laurent power series $\Q_\ell\llps \lambda\lrps$. Then $\lambda\omega$ is a non-degenerate $2$-form 
on the rigid $\Q_\ell\{\lambda\}$-analytic space $\bar\calD_{\Q_\ell\{\lambda\}}:=\bar\calD\hat\otimes_{\Q_\ell}\Q_\ell\{\lambda\}$ which is isomorphic to a unit polydisk over 
$\Q_\ell\{\lambda\}$. The form $\lambda\omega$ is closed relative to the base field $\Q_\ell\{\lambda\}$,
hence it gives a symplectic structure on $\bar\calD_{\Q_\ell\{\lambda\}}$, and is invariant
under the diagonal action of $G_k$ on $\bar\calD_{\Q_\ell\{\lambda\}}=\bar\calD\hat\otimes_{\Q_\ell}\Q_\ell\{\lambda\}$. Note however, that this action is not $\Q_\ell\{\lambda\}$-linear.

}
\end{Remark}

\subsection{The Galois flow}\label{ss:Flow}
Since $\bar A/\frakm_{\bar A}^2$ is a finite ring, there is an integer $N\geq 1$ such that the $N$-th iteration $\psi=\phi(\sigma)^N: \bar A\to \bar A$ satisfies $\psi\equiv {\rm Id}\, {\rm mod}\, \frakm^2_{\bar A}$. Assume also that $(\ell-1)|N$. Recall $\calD={\rm Spf}(\bar A)[1/\ell]\simeq D_1(m)$.

Since $\bar A\simeq \O\lps x_1,\ldots , x_m\rps$, the results of the Appendix, especially Proposition \ref{propConv2}, apply to $\psi$.  We obtain:
 
 \begin{thm}\label{thm:flow} We can write $\calD=\cup_{c\in {\mathbb N}} \bar\calD_c$ as an increasing open union of affinoids (each $\bar\calD_c$ isomorphic to a closed ball $\bar D_{r(c)}(m)$ of radius $r(c)$ increasing to $1$) such that:
 
For each $c\geq 1$, there is $\varepsilon(c)\in \Q_{>0}$ with the property that, for each $\sigma\in G_k$, there is a 
rigid analytic map (the ``flow")
 \[
  \{t\ |\ |t|_\ell\leq \varepsilon(c)\}\times \bar\calD_c\to \bar\calD_c,\quad (t,\x)\mapsto \psi^t(\x):=\sigma^{tN}(\x),
 \]
 which satisfies 
 \begin{itemize}
 \item $\psi^{t+t'}=\psi^t\cdot \psi^{t'}$, for all $|t|_\ell, |t'|_\ell\leq \varepsilon(c)$,
 \smallskip
 
 \item   $\psi^n: \bar\calD_c\to \bar\calD_c$ is
 given by the action of $\sigma^{nN}$, for all $n\in \Z$ with $|n|_\ell\leq \varepsilon(c)$.
 \end{itemize}
 \end{thm}
 
 (In fact, the  Appendix gives more precise information on the flow $\sigma^{tN}$.) 
 
The flow $\psi^t$ induces a rigid analytic vector field $X_{\sigma^N}=X_\psi$   on 
$\calD$. The vector field
\[
X_{\sigma}:=N^{-1}\cdot X_{\sigma^N}
\]
on $\calD$ is well-defined and independent of the choice of $N$. 
The contraction of the $2$-form $\omega$ with $X_\sigma$ gives a rigid analytic $1$-form
\[
\mu_\sigma:=X_\sigma\intprod \omega
\]
on $\calD$. 

Denote by $\log_\ell: \Z_\ell^*\to \Q_\ell$ the $\ell$-adic logarithm.

\begin{prop}\label{FrobId}
 $d\mu_\sigma=-\log_\ell(\chi_{\rm cycl}(\sigma))\omega$.
\end{prop}

\begin{proof}
It is enough to show the identity in $E\lps x_1,\ldots , x_m\rps$, i.e. to check that the germs at $(0, \ldots, 0)$ of both sides agree. We use a formal version of  ``Cartan's magic formula" 
\[
L_X=i_X\cdot d+d\cdot i_X.
\]
For completeness, we give the argument for the proof of this formula in our set-up. Set   $E\lps \x\rps=E\lps x_1,\ldots, x_m\rps$. Consider the graded commutative superalgebra 
\[
\Omega :=\bigoplus\nolimits_{i\in \Z} \Omega^i
\]
where, for $i\geq 0$, $\Omega^i=\wedge^i\hat\Omega^1_{E\lps \x\rps/E}$,
while $\Omega^{-i}=0$, with multiplication satisfying $ab=(-1)^{ij}ba$, for $a$, $b$ in degree $i$, $j$. A derivation $D$ of degree $\delta$ of $\Omega$ is an $E$-linear graded map $D: \Omega\to \Omega$ satisfying $D(ab)=(Da)b+(-1)^{i\delta}aD(b)$, for $a\in \Omega^i$. For example, the standard $d$ is a derivation. If $X=X_\psi$ is the vector field associated to the flow $\psi^t(\x)$, then 
the Lie derivative $L_X: \Omega\to \Omega$, which, by definition, is given by
\[
L_X \tau:=\lim_{t\mapsto 0}   \frac{1}{t} (\psi^t(\x)^*(\tau)-\tau)
\]
is a derivation of degree $0$. The contraction $i_X=X\intprod -: \Omega \to \Omega $ is a derivation of degree $-1$ and we can easily see that the  ``superbracket" $[i_X, d]=i_X\cdot d+d\cdot i_X$ 
is a derivation of degree $0$. Now notice that the two derivations $L_X$ and $[i_X, d]$ of degree $0$ agree 
on $\Omega^0=E\lps \x\rps$. Indeed, if $f\in E\lps \x\rps$, then 
\[
L_X(f)=i_X\cdot df= X \lrcorner\, df
\]
while $i_X(f)=0$.
 Also, both derivations $L_X$ and $[i_X, d]$ commute with $d$. Indeed, since $d\cdot d=0$,  we have
\[
d\cdot (i_X\cdot d+d\cdot i_X)=d\cdot i_X\cdot d=(i_X\cdot d+d\cdot i_X)\cdot d.
\]
Also, $d\cdot L_X=L_X\cdot d$ since pull-back by $\psi^t(\x)$ commutes with $d$.
The proof of Cartan's magic formula $L_X=[i_X, d]$ follows by observing that any two derivations of degree $0$ on $\Omega$ that commute with $d$ and agree on $\Omega^0$ 
have to agree.  
\smallskip

Now apply this to $\psi=\phi(\sigma^N)=\phi(\sigma)^N$ with $N$ as above.
Since $d\omega=0$, $X_\sigma=N^{-1}\cdot X_\psi$, $\mu_\sigma=N^{-1}\cdot i_{X_\psi}(\omega)$, we have
\[
L_{X_\psi}(\omega)=i_{X_\psi}\cdot d\omega+d\cdot i_{X_\psi}(\omega)=N\cdot d\mu_\phi,
\]
and is enough to show that $L_{X_\psi}(\omega)=-N\log_\ell(\chi_{\rm cycl}(\sigma))\omega$.  

We have
\[
L_{X_\psi}(\omega):=\lim_{t\mapsto 0}   \frac{1}{t} (\psi^t(\x)^*(\omega)-\omega)=
\lim_{n\mapsto +\infty}  \ell^{-n} (\psi^{l^n}(\x)^*(\omega)-\omega).
\]
Now $(\psi^{l^n}(\x))^*(\omega)=\psi^{l^n}(\omega)=\phi(\sigma)^{N\ell^n}(\omega)=\chi_{\rm cycl}(\sigma)^{-N\ell^n} \omega$,  with the last identity given by Proposition \ref{FrobId}. Therefore,
\[
L_{X_\psi}(\omega)=\lim_{n\mapsto +\infty} \ell^{-n} (\chi_{\rm cycl}(\sigma)^{-N\ell^n}-1)\omega=-N \log_\ell(\chi_{\rm cycl}(\sigma)) \omega.
\]
(Recall $\ell-1$ divides $N$.) This identity, combined with the above, completes the proof.
\end{proof}

\begin{thm}\label{criticalThm} Assume $\ell$ does not divide $d$.
The critical set 
\[
{\rm Crit}:= \{\x\in \calD\ |\ X_\sigma(\x)=0, \forall\sigma\in \Gk\}
\]
 is equal to the set of points $\x$ of $\calD$ for which there is a finite extension $k'/k $ such that the representation $\bar\rho_\x: \pi_1(\bar X)\to \GL_d(\bar\Q_\ell)$ extends to $\rho_\x: \pi_1(X\times_kk')\to \GL_d(\bar\Q_\ell)$ which deforms $(\rho_0)_{|\pi_1(X\times_kk')}$.
\end{thm}

\begin{proof}  
 Recall (see for example \cite{Mazur2}, Sect. 6, (1.3) d) that, for any finite extension $F$ of $E$, the base change $(\bar A_{\rm un}\hat\otimes_{\O}\O_F, \bar\rho_{\rm un}\hat\otimes_{\O}\O_F)$ represents the deformation functor
 ${\rm Def}(\pi_1(\bar X),  \bar\rho_0\otimes_{\O}\O_F,  \bar\ep\otimes_{\O}\O_F)$.
 Suppose now $\x\in \calD$ is such that $\bar\rho_{\x}$ extends to 
$\rho_\x: \pi_1(X)\to \GL_d(\bar\Q_\ell)$. There is a finite extension
$F$ of $E$ such that the images $\rho_\x(\pi_1(\bar X))$ and $\rho_\x(\sigma)$ both lie in a conjugate of $\GL_d(\O_F)$ in $\GL_d(\bar\Q_\ell)$. 
By Proposition \ref{fixedprop} (a) and its proof, we can see that the action of $\sigma$ on $\calD$  fixes the point $\x$. (We consider 
here the point $\x$ as giving a value for the deformation problem 
${\rm Def}(\pi_1(\bar X),  \bar\rho_0\otimes_{\O}\O_F,  \bar\ep\otimes_{\O}\O_F)$.)
Hence, the flow given by $\sigma$ on $\calD$ also fixes $\x$; it follows that $X_\sigma$ vanishes
at $\x$.  

Conversely, suppose $\x\in {\rm Crit}$. There is a finite extension $F$ of $E$ in $\bar\Q_\ell$ such that $F=E(\x)$ and
a conjugate of $\bar\rho_\x$ takes values  in $\GL_d(\O_F)$.
 Set $r=\max(||\x||, (1/\ell)^{1/e})=(1/\ell)^a$, $a\in \Q\cap(0,1/e]$. Pick $N>1/a(\ell-1)+1$. Then by 
 Proposition \ref{propConv1} and Proposition \ref{propConv3} and its proof, if $\phi(\sigma) \equiv {\rm id}\,{\rm mod}\, \frakm^N$, then 
 $\phi(\sigma)(\x)=\psi^1(\x)=\x$. Now by continuity, there is a finite index normal subgroup $U\subset {\rm Gal}(k^{\rm sep}/k)$, such that for $\sigma\in U$, $\phi(\sigma)\equiv {\rm id}\,{\rm mod}\, \frakm^N$. Hence, there is a finite index normal
 $U\subset {\rm Gal}(k^{\rm sep}/k)$ such that
 for all $\sigma\in U$, $\sigma(\x)=\x$. Apply Proposition \ref{fixedprop} (b) to 
 $A=\O_F$ and $\bar\rho=\bar\rho_\x$ to the base field $k'=(k^{\rm sep})^U$.
 We obtain that $\bar\rho_\x$ extends to a continuous representation of $\pi_1(X\times_k k')$ with determinant $\ep$ which deforms $(\rho_0)_{|\pi_1(X\times_kk')}$.
 \end{proof}

\begin{cor}\label{vanish2} Suppose $k$ is a finite field of order $q=p^f$, $p\neq \ell$, and assume $\ell$ is prime to $d$. Let $X$ be a smooth projective curve over $k$ and let 
\[
\rho_0: \pi_1(X,\bar x)\to \GL_d(\BF)
\]
be a representation with determinant $\ep$ such that $\bar\rho_0={\rho_0}_{|\pi_1(\bar X)}$ is geometrically irreducible.   
Suppose $\x$ is an $E$-valued point of the rigid analytic deformation space  ${\rm Spf}(\bar A_{\rm un})[1/\ell]$. as before, where $E$ is a finite extension of $W(\BF)[1/\ell]$ with integers $\O_E$. The lift $\bar\rho_\x:\pi_1(\bar X)\to \GL_d(\O_E) $ of $\bar\rho_0:  \pi_1(\bar X)\to \GL_d(\BF)$ that corresponds to $\x$ extends to a continuous representation 
\[
\rho_{\x}:\pi_1(X\times_{\BF_q}{\BF_{q^N}})\to \GL_d(\O_E) 
\]
 with determinant $\ep$, for some $N\geq 1$, if and only if the $1$-form $\mu_{{\rm Frob}_q}$ vanishes at $\x$.   
\end{cor}

\begin{proof}
Follows from Theorem \ref{criticalThm} by observing that $G_k$ is topologically generated 
by ${\rm Frob}_q$ and so the critical set ${\rm Crit}$ is the zero locus of $\mu_{{\rm Frob}_q}$.
\end{proof}

\subsection{Hamiltonian Galois flow}  Recall $G_{k\cy}={\rm Gal}(k^{\rm sep}/k(\zeta_{\ell^\infty}))$ 
and set again $\bar A=\bar A_{\rm un}$. Recall the group homomorphism
\[
\phi: G_{k\cy} \to {\rm Aut}_{\O}(\bar A).
\]
By Proposition \ref{contAction}, this is continuous when we equip ${\rm Aut}_{\O}(\bar A)$ with the profinite topology given by the normal subgroups $\calK_n=\ker({\rm Aut}_\O(\bar A)\to {\rm Aut}_\O(\bar A/\frakm^n))$.

For $\sigma\in G_{k\cy}$, we have $\chi_{\rm cycl}(\sigma)=1$ and by Proposition \ref{FrobId}, $d\mu_\sigma=0$. The Poincare lemma \ref{poincare} implies that there is a rigid analytic function  $V_\sigma\in \sO(\calD)$ such that $\mu_\sigma=dV_\sigma$; we can normalize $V_\sigma$ by requiring $V_\sigma(0,\ldots ,0)=0$. We can think of $V_\sigma$ as a ``Hamiltonian potential" for the flow $\sigma^t$.

\begin{thm}\label{thm:Hamiltonflow}
The map $\sigma\mapsto V_\sigma$ extends to a $\Z_\ell$-linear  map
\[
V: \Z_\ell\lps G_{k\cy}\rps\to  \sO(\calD) ; \quad \sum_\sigma z_\sigma \sigma\mapsto \sum_\sigma z_\sigma V_\sigma
\]
which is continuous for the Fr\'echet topology on $\sO(\calD)$ and satisfies:
\begin{itemize}
\item[1)] For $\gamma\in \Gk$, $\sigma\in G_{k\cy}$, $V_{\gamma\sigma\gamma^{-1}}=\phi(\gamma)(V_\sigma)$, 
\smallskip

\item[2)] For $\sigma$, $\tau\in G_{k\cy}$, 
$
-d\{V_\sigma, V_\tau\}=[X_\sigma, X_\tau]\intprod \omega.
$

\end{itemize}
\end{thm}

Define  
\[
J: \calD\to {\rm Hom}(\Z_\ell\lps G_{k\cy}\rps, \bar\Q_\ell),
\]  
by $J(\x)=(z\mapsto   V_z(\x))$.
 We may think of $J$ as describing a moment map for the symplectic (Hamiltonian) action of $G_{k\cy}$ on $\calD$.

\begin{proof} We choose an isomorphism $\bar A\simeq R=\O\lps x_1,\ldots , x_m\rps$ that will allow
us to use the explicit constructions of the previous sections. We first show

\begin{lemma}
Fix $r=(1/\ell)^a$, $a\in \Q\cap(0,1/e]$, and $\epsilon >0$. There exists a finite index open normal subgroup
$U\subset G_{k\cy}$ such that for all $\sigma\in U$, we have 
\[
||X_\sigma||_r=\sup_{\x\in \bar D_r(m)}||X_\sigma(\x)||<\epsilon.
\]
\end{lemma}

\begin{proof}
We first observe that there exists $n=n(\ep)$ such that  $\phi(\sigma)\equiv {\rm Id}\, {\rm mod}\, \frakm^n_{\bar A}$,
implies that $||X_\sigma||_{r}<\epsilon$.  This follows from the argument in the proof of (\ref{normlim}). The result now follows from the continuity of $\phi$ (Proposition \ref{contAction}).
\end{proof}

\begin{Remark}
{\rm Consider the analytic vector field $X_\sigma=\sum_{i=1}^m X_i(\sigma)\partial/\partial x_i$. The inequality
$||X_\sigma||_r<\epsilon$ is $\sup_{i}||X_i(\sigma)||_r<\epsilon$. Suppose we perform a coordinate base change $x_i=\psi_i({\bf y})$ by an $\O$-automorphism given by $\psi: R\to R$. Then, if $||{\bf y}||\leq r$, $||\psi({\bf y})||\leq r$ and so $||X_i(\sigma)(\psi({\bf y}))||_r<\epsilon$. Also, $\partial y_j/\partial x_i\in R$. Since
\[
X_\sigma=\sum_j(\sum_iX_i(\sigma)(\psi({\bf y}))\frac{\partial y_j}{\partial x_i})\frac{\partial }{\partial y_j}
\]
it follows that the validity of $||X_\sigma||_r<\epsilon$ is independent of the choice of identification $\bar A_{\rm un}\simeq \O\lps x_1,\ldots, x_m\rps$.
}
\end{Remark}

Now write $X_\sigma=\sum_{i=1}^m X_i(\sigma)\partial/\partial x_i$  and $\omega=\sum_{i<j}g_{ij}dx_i\wedge dx_j$
with $g_{ij}\in R$; then 
\[
\mu_\sigma=i_{X_\sigma}(\omega)=\sum\nolimits_{i, j} X_i(\sigma) g_{ij} dx_j=\sum\nolimits_j h_j(\sigma)dx_j,
\]
where $h_j(\sigma)=\sum\nolimits_{i} X_i (\sigma)g_{ij}$. By the above lemma, there is a finite index normal open subgroup
$U\subset G_{k\cy}$
such that $\sup_{i}||X_i(\sigma)||_r<\epsilon$, for all $\sigma\in U$. Since $||g_{ij}||\leq 1$, 
we also have $\sup_{j}||h_j(\sigma)||_r<\epsilon$. Since the Tate algebra $\sO(\bar D_r(m))$ is complete for the Gauss norm $||\cdot ||_r$,  we obtain that, for each $r\mapsto 1^-$, the map $\sigma\mapsto h_j(\sigma)$ extends to
\[
h_j: \Z_\ell\lps G_{k\cy}\rps\to \sO(\bar D_r(m)).
\]
These maps are compatible with the restrictions $\sO(\bar D_r(m))\to \sO(\bar D_{r'}(m))$, $r'<r$.
Therefore, they give the extension $h_j: \Z_\ell\lps G_{k\cy}\rps\to \sO(\calD)$ which, in fact, continuous
for the Fr\'echet topology on $\sO(\calD)$ given by the family of Gauss norms $\{||\cdot ||_r\}_r$.
For $z=\sum_\sigma z_\sigma\sigma$, now set
\[
\mu_z=\sum\nolimits_{j} h_j(z) dx_j
\]
with $h_j(z)\in \sO(\calD)$. This $1$-form is also closed and by Proposition \ref{poincare} (a), there is (a unique)
$V_z\in \sO(\calD)$ with $V_z(0,\ldots , 0)=0$ and $dV_z=\mu_z$. The map $z\mapsto V_z$ gives our extension.
The continuity follows from the construction together with the fact that taking
(partial) antiderivatives is continuous for the Fr\'echet topology on $\sO(\calD)$.
(In turn, this follows by some standard estimates using that
$\lim_{i\to\infty}\ell^i(r/r')^{\ell^i}\to 0$, for $0<r<r'<1$. )
 
 Property (1) follows from the definitions using the identity of flows 
 \[
 (\gamma\sigma\gamma^{-1})^t=\phi(\gamma)\sigma^t\phi(\gamma^{-1}),
 \]
 (which follows from interpolating using the identities $(\gamma\sigma\gamma^{-1})^{\ell^n}=\gamma\sigma^{\ell^n}\gamma^{-1}$ in
 $\Gk$). 
 
 Property (2) is formal (see \cite[Prop. 18.3]{daSilva}): 
 We have
\[
[X_\sigma, X_\tau]\intprod \omega=L_{X_\sigma}(X_\tau\intprod \omega)-X_\tau\intprod (L_{X_\sigma}\omega).
\]
(This comes from the standard formal identity $[X,Y]\intprod \alpha=L_X(Y\intprod \alpha)-X\intprod(L_X\alpha)$ which
can be shown by arguing as in our proof of Cartan's magic formula above.)
By Cartan's formula this is equal to
\[
d(X_\sigma\intprod(X_\tau\intprod\omega))+X_\sigma\intprod d(X_\tau\intprod \omega)-X_\tau\intprod d(X_\sigma\intprod \omega)-
X_\tau\intprod(X_\sigma\intprod d\omega).
\]
In this expression, the last three terms are trivial since $d(X_\sigma\intprod \omega)=d(X_\tau\intprod \omega)=0$, $d\omega=0$. By definition, $X_\sigma\intprod(X_\tau\intprod\omega)=-\omega(X_\sigma, X_\tau)=-\{V_\sigma, V_\tau\}$
and this completes the proof. \end{proof}

\begin{cor}\label{criticalCor} Assume $\ell$ does not divide $d$.
The critical locus set 
\[
J^{-1}(0)=\{\x\in \calD\ |\ dV_\sigma(\x)=0, \forall\sigma\in G_{k\cy}\}
\]
 is equal to the set of points $\x$ of $\calD$ for which there is a finite extension $k'/k\cy$ such that the representation $\bar\rho_\x: \pi_1(\bar X)\to \GL_d(\bar\Q_\ell)$ extends to $\rho_\x: \pi_1(X\times_kk')\to \GL_d(\bar\Q_\ell)$ which deforms $(\rho_0)_{|\pi_1(X\times_kk')}$.
\end{cor}

   \begin{proof}
   This follows from Theorem \ref{criticalThm} by replacing $k$ by $k\cy$ and noting that for $\sigma\in G_{k\cy}$,
we have $dV_\sigma=\mu_\sigma$ which vanishes at $\x$ if and only if $X_\sigma$ vanishes at $\x$.  
 \end{proof}
 
 \subsubsection{} In the above, suppose $\bar\rho_{\x}$ extends to
 a representation $\rho_{\x}$ of $\pi_1(X)$. Then  $\x$ is a critical point of $V_\sigma$,
 $\forall \sigma\in G_{k\cy}$. It is reasonable to ask the following question: 
Do we have
  \[
 V_\sigma(\x)= A(\sigma)\cdot \CS({\rho_{\x}})(\sigma)+B(\sigma),  
 \]
for all $\sigma\in G_{k\cy}$, where $A(\sigma)$, $B(\sigma)$ are constants independent of $\x$?

\subsection{Milnor fiber and vanishing cycles}\label{Milnor} 
Suppose that $\sF$ is a \'etale $\Z_\ell$-local system over $X$. Assume that the corresponding
representation $\rho_0$ is such that $\bar\rho_0: \pi_1(X\times_k \bar k)\to \GL_d(\BF_\ell)$ is geometrically irreducible and that $\ell$ does not divide $d$.
Then the representation of $\pi_1(X\times_k \bar k)$ given by $\sF$ corresponds to a point $\x$ of the deformation space $\calD$
which is a critical point of $V_\sigma$,
 $\forall \sigma\in G_{k\cy}$. 
 
 Let us consider the germ $\hat V_\sigma$ of $V_\sigma$ in the completion $\hat\O_{\calD, \bar x}$ of the local ring $\O_{\calD, \bar x}$ of the rigid analytic $\calD$ at $\x$. This completion is isomorphic (non-canonically) to $\Q_\ell\lps v_1,\ldots, v_m\rps$ (e.g. by taking $v_i=x_i-\x_i$) and the germ $\hat V_\sigma$
 defines a $\Q_\ell$-algebra homomorphism
 $
 \Q_\ell\lps u\rps \to \hat\O_{\calD,\x},   
 $ by   $u\mapsto \hat V_\sigma-V_\sigma(\bar x)$.
 Consider the corresponding morphism of formal schemes
 \[
 f_\sigma: \Spf(\hat\O_{\calD,\x})\to \Spf(\Q_\ell\lps u\rps).
 \]
 (Here we use the $u$-adic topology, the $\ell$-adic topology plays no role.) This makes  $\Spf(\hat\O_{\calD,\x})$ a special formal scheme over $\Spf(\Q_\ell\lps u\rps)$ in the sense of \cite[1]{Berk}.

In this situation, we can consider various local invariants of the critical point $\x$ of $V_\sigma$:

\begin{para} The {\sl analytic Milnor fiber} 
 \[
 M(\x, \sigma)=\Spf(\hat\O_{\calD,\x}\hat\otimes_{\Q_\ell}\bar\Q_\ell)[1/u].
 \]
 This is, by definition (cf. \cite{NiSe}), the generic fiber of $f_\sigma\hat\otimes_{\Q_\ell}\bar\Q_\ell$ considered as a $\bar\Q_\ell\llps u\lrps$-analytic space.
 \end{para}
 
\begin{para} The stacks of the nearby cycle sheaves (\cite{Berk}, \cite{Berk2})
 \[
 R^i\Psi_{f_\sigma\hat\otimes_{\Q_\ell}\bar\Q_\ell}(\Q_p)_{\x}:=(\varprojlim\nolimits_n(R^i\Psi_{f_\sigma\hat\otimes_{\Q_\ell}\bar\Q_\ell}(\Z/p^n\Z))_{\x})\otimes_{\Z_p}\Q_p
 \]
  at $\x$. 
  
  By \cite[Theorem 3.1.1, Corollary 3.1.2]{Berk2}, for each $n\geq 1$, $R^i\Psi_{f_\sigma\hat\otimes_{\Q_\ell}\bar\Q_\ell}(\Z/p^n\Z))_{\x}$
  are finite ${\rm Gal}(\overline{\Q_\ell\llps u\lrps}/\Q_\ell\llps u\lrps)\simeq \hat \Z$-modules.
  By \emph{loc. cit.}, these agree with the \'etale cohomology groups
 \[
 \rH^i(\x,\sigma, \Z/p^n\Z):=\rH^i_\et(M(\x, \sigma)\times_{\bar\Q_\ell\llps u\lrps}\overline{\Q_\ell\llps u\lrps}^\wedge, \Z/p^n\Z)
 \]
 (again in the sense of Berkovich). In fact, the proof of the above results in \cite{Berk2} also give that there is $i_0$ such that
 for $i>i_0$, $ \rH^i(\x,\sigma, \Z/p^n\Z)=(0)$, for all $n$. Hence, for each $n\geq 1$,   
$
  \rH^i(\x,\sigma, \Z/p^n\Z) $
   are the cohomology groups of a perfect complex $P^\bullet(\x,\sigma, \Z/p^n\Z)$ of $\Z/p^n\Z$-modules. A standard argument (e.g. \cite[VI 8.16]{Milne}) gives that  there us a perfect complex of $\Z_p$-modules $ P^\bullet(\x,\sigma ) $
   so that    
   \[
  P^\bullet(\x,\sigma )\otimes_{\Z_p}\Z/p^n\Z\simeq   P^\bullet(\x,\sigma, \Z/p^n\Z).    
  \]
  Then $\rH^i(  P^\bullet(\x,\sigma ))\simeq \varprojlim_n \rH^i(\x,\sigma, \Z/p^n\Z)$ and
 we can conclude that for each $i$, 
 \[
 \rH^i(\x,\sigma, \Q_p)=(\varprojlim\nolimits_n(\rH^i(\x,\sigma, \Z/p^n\Z))\otimes_{\Z_p}\Q_p
 \]
 is a finite dimensional $\Q_p$-vector space with an action of ${\rm Gal}(\overline{\Q_\ell\llps u\lrps}/\Q_\ell\llps u\lrps)\simeq \hat \Z$.
  \end{para}
  
\begin{para}
With notations as above, 
 we can consider the (``perversely" shifted)  Euler characteristic of the vanishing cycles
 \begin{align*}
 {\lambda}(\sF, \sigma)&=(-1)^m (1-\chi(M(\x, \sigma)\times_{\bar\Q_\ell\llps u\lrps}\overline{\Q_\ell\llps u\lrps}^\wedge))=
 \\
  &=(-1)^m(1-\sum\nolimits_i (-1)^{i} \dim_{\Q_p} \rH^i_\et(M(\x, \sigma)\times_{\bar\Q_\ell\llps u\lrps}\overline{\Q_\ell\llps u\lrps}^\wedge, \Q_p)).
 \end{align*}
Note that the integer ${\lambda}(\sF,\sigma)$ is analogous to (a local version of) the Casson-type invariant given in \cite{AM} or 
the Behrend invariant of \cite{Ber}. Calculating this number appears to be a hard problem.

    \end{para}

\bigskip
 
 \section{Appendix: Interpolation of iterates and flows}\label{App}
 
In this appendix, we elaborate on an idea  of Poonen \cite{Poonen} (inspired by \cite{Tucker}) about $\ell$-adic interpolation of iterates. A similar construction using this $\ell$-adic interpolation argument was also used by Litt \cite{Litt}. We need a little more information than what is given in these references. The proofs of Theorems \ref{thm:flow}, 
\ref{criticalThm}, and \ref{thm:Hamiltonflow} use some of the bounds and estimates of rates of convergence 
shown below.

We  assume that $\ell$ is an odd prime and $\O$ a totally ramified extension of $W(\BF)$ of degree $e$.
For $a\in \Q\cap (0, 1/e]$, set $r=(1/\ell)^a$, so that $(1/\ell)^{1/e}=|\fl|_\ell\leq r<1$. 
We  set $R=\O\lps x_1,\ldots, x_m\rps$, $\frakm=(\fl, x_1,\ldots , x_m)$.  
Consider a $\O$-algebra homomorphism $\psi:  R\to  R$ such that $\psi\equiv {\rm id}\,{\rm mod}\, \frakm^N$
for some $N\geq 2$. This is determined by $\psi(\x)=(\psi(x_1), \ldots, \psi(x_m))\in R^m$. Set $\psi_j=\psi(x_j)\in R$. 
We also set $||\psi(\x)||_r=\sup_i||\psi_i||_r$. For $||\x||\leq r$ we also have $||\psi(\x)||_r\leq r$. Therefore, $\psi(\x)$ gives a rigid analytic $  \bar D_r(m)\to \bar D_r(m)$, for any such $r$; these maps agree and they are the restriction of a rigid analytic map $\und\psi: D_1(m)\to D_1(m)$.  Since $\und\psi: D_1(m)\to D_1(m)$ is given by 
\[
{\bf a}=(a_1,\ldots, a_m)\mapsto  (\psi_1(a_1,\ldots, a_m), \ldots , \psi_m(a_1,\ldots, a_m))=\psi({\bf a})
\]
we will often also denote this map by $\psi(\x)$. Our goal is to $\ell$-adically interpolate the iterates
$ \und\psi\circ \cdots \circ \und\psi$ of $\und\psi $ and obtain various related estimates. For simplicity, we will often write
\[
D=D_1(m).
\]

\subsection{Difference operators} As in \cite{Poonen}, set  $\Delta_\psi$ for the operator 
that sends $h:R\to R$ to $\Delta_\psi(h): R\to R$ given by $\Delta_\psi(h)(\x)=h(\psi(\x))-h(\x)$.  
Similarly, if $f\in R$, we can consider
$\Delta_\psi(f)\in R$ given by the power series $\Delta_\psi(f)(\x)=f(\psi(\x))-f(\x)$.

For simplicity, set $\Delta=\Delta_\psi$ and suppose $g$, $h\in R$. We have
\begin{align*}
\Delta(gh)(\x)=\, &(gh)(\psi(\x))-(gh)(\x)\\
=\, & g(\psi(\x))h(\psi(\x))-g(\x)h(\x)\\
 =\, & h(\psi(\x))\Delta(g)(\x)+g(\x)\Delta(h)(\x)\\
 =\, & g(\psi(\x))\Delta(h)(\x)+h(\x)\Delta(g)(\x)\\
=\, & g(\x)\Delta(h)(\x)+h(\x)\Delta(g)(\x)+\Delta(g)(\x)\Delta(h)(\x).
\end{align*}
Hence,
\begin{equation}\label{derivation1}
\Delta(gh)(\x)=g(\x)\Delta(h)(\x)+h(\x)\Delta(g)(\x)+\Delta(g)(\x)\Delta(h)(\x).
\end{equation}

Consider the formal series
\[
\psi^t:=({\rm I}+\Delta)^t=\sum_{k\geq 0} {t\choose k}\Delta^k={\rm I}+t\Delta+\frac{t(t-1)}{2!}\Delta^2+\cdots
\]
\[
X_\psi:={\rm Log}(\psi)=\log(\rI+\Delta)=\Delta -\frac{1}{2}\Delta^2+\frac{1}{3}\Delta^3-\cdots .
\]
For $g$, $h\in R$, we can see using (\ref{derivation1}) that $X_\psi$ satisfies, at least formally, the identity
\begin{equation}\label{derivation2}
X_\psi(gh)=gX_\psi(h)+hX_\psi(g).
\end{equation}
Formally, we have
\[
\psi^t={\rm I}+tX_\psi+\frac{t^2}{2}X^2_\psi+\cdots .
\]
If $\phi: R\to R$ is another such map, then
\[
\phi^t\psi^s\phi^{-t}\psi^{-s}=(\rI+tX_\phi+\cdots )(\rI+sX_\psi+\cdots )(\rI-tX_\phi+\cdots )(\rI-sX_\psi+\cdots )=
\]
\[
=\rI+ts(X_\phi X_\psi-X_\psi X_\phi)+(\hbox{\rm degree $\geq 3$ in  $s$, $t$}) .
\]

\subsection{Interpolation} We continue with an $\O$-algebra automorphism $\psi: R\to R$ inducing the identity on $R/\frakm^N$, $N\geq 2$.

Apply the operators of the previous paragraph to the identity map ${\rm id}: \x\mapsto \x$, so $\Delta(\x)=\psi(\x)-\x \in (\frakm^{N})^{\oplus m}$, $\Delta^2(\x)=\psi^2(\x)-2\psi(\x)+\x$. By induction, we have
\[
({\rm I}+\Delta )^k(\x)=\psi^k(\x),
\] 
for all $k\geq 1$. 

Recall $(1/\ell)^e\leq r=(1/\ell)^a<1$. We have $||\Delta(\x)||_r\leq r^N$.
By induction:
\[
\Delta^k(\x)\equiv 0\,{\rm mod}\, \frakm^{k(N-1)+1}, \quad
||\Delta^k(\x)||_r\leq r^{k(N-1)+1}.
\]
Since $|k!|_\ell\geq (1/\ell)^{k/(\ell-1)}$, we obtain, for $k\geq 1$,
\[
||{t\choose k}\Delta^k(\x)||_r\leq |t|_\ell\cdot r^{k(N-1)+1}(1/\ell)^{-k/(\ell-1)}=|t|_\ell\cdot (1/\ell)^{k[a(N-1)-1/(\ell-1)]+a}.
\]

\begin{prop}\label{propConv1}
 Suppose $\psi\equiv {\rm id}\,{\rm mod}\, \frakm^N$, $N\geq 2$, and   $a\in \Q\cap(0, 1/e]$.

a) The power series giving $\psi^t(\x)$ converge when $|t|_\ell\leq 1$, $||\x||\leq r=(1/\ell)^a $ and  $a>1/(\ell-1)(N-1)$.
Then $||\psi^t(\x)||_r\leq r$, so these give  an analytic map
\[
\psi^t(\x) : \bar D_1(1)\times \bar D_r(m)\to \bar D_r(m).
\]

b) The power series giving $X_\psi(\x)$ converge on $||\x||<1$.  We have
\[
||X_\psi(\x)||_{(1/\ell)^a}\leq \ell^{N(1, a(N-1))}\ell^{-(1+a)}.
\]
\end{prop}

\begin{proof} Part (a) follows from the above estimates.
For part (b) notice that we have
\[
||\Delta^k_\psi(\x)/k||_{(1/\ell)^a}\leq |k|^{-1}_\ell(1/\ell)^{ka(N-1)+a}\leq 
\ell^{d_\ell(k)-a(N-1)k}\ell^{-(1+a)}.
\]
The result follows as in the proof of Proposition \ref{sequence} (a).
\end{proof}
\smallskip

Fix $a\in \Q\cap (0, 1/e]$, $r=(1/\ell)^a$. From Proposition \ref{propConv1} (b) and Lemma \ref{trivial} it follows (as in Proposition \ref{sequence} (b)) that if $(\psi_n)_n$ is a sequence of maps with $\psi_n\equiv {\rm id}\,{\rm mod}\, \frakm^n$ and $n\mapsto +\infty$, then
\begin{equation}\label{normlim}
||X_{\psi_n}(\x)||_r\mapsto 0.
\end{equation}
In fact, by estimating $N(1, a(n-1))$, we can see that $||X_{\psi_n}(\x)||_r\leq r^n$ if $n$ is large enough so that $a>1/(\ell-1)(n-1)$.
On the other hand, in general, for a fixed $\psi$, 
$||X_\psi(\x)||_r\mapsto +\infty$ as $r\mapsto 1^-$.

\begin{cor}
The map $X_\psi: R\to \sO(D)$ which sends $f$ to
\begin{align*}
X_\psi(f)(\x)=&\sum_{k=1}(-1)^{k-1}\frac{\Delta^k_\psi(f)(\x)}{k}=\\
=&f(\psi(\x))-f(\x)-\frac{f(\psi^2(\x))-2f(\psi(\x))+f(\x)}{2}+\cdots
\end{align*}
is an $\O$-linear continuous derivation. It extends naturally  to a continuous $\O$-linear derivation $X_\psi: \sO(D)\to \sO(D)$.
\end{cor}

\begin{proof}
Follows from the above convergence and (\ref{derivation2}).
\end{proof}

Note that the component $X_\psi(\x)_i$ of $X_\psi(\x)=(X_\psi(\x)_1,\ldots, X_\psi(\x)_m)$ is equal to $X_\psi(x_i)(\x)$ and so we can write
\[
X_\psi(\x)=\sum_{i=1}^mX_\psi(\x)_i\frac{\partial}{\partial x_i}.
\]

\begin{lemma}\label{lemmaLn}
 If $\psi\equiv {\rm id}\,{\rm mod}\, \frakm^2$ then $\psi^{\ell^n}\equiv {\rm id}\,{\rm mod}\, \frakm^{n+2}$, for all $n\geq 0$.
\end{lemma}

\begin{proof}
Set 
 \[
 \calA_{n}={\rm ker}({\rm Aut}_{\O}(R/\frakm^{n+1})\to {\rm Aut}_{\O}(R/\frakm^{n}))
 \]
 for the kernel of reduction. Any $\chi\in \calA_{n}$, for $n\geq 2$, is given by
 \[
 \chi(x_1)= x_1+A(\chi)_1,\ldots , \chi(x_m)=x_m+A(\chi)_m,
 \]
 where $A(\chi)=(A(\chi)_1,\ldots , A(\chi)_m)\in (\frakm^n/\frakm^{n+1})^m$.
 Using induction, we see that the $N$-iteration $\chi^N$ is given by the row $A(\chi^N)=N\cdot A({\chi})$
 and therefore $\chi^\ell={\rm Id}$.  By assumption, $\psi$ gives an element of $\calA_2$. By the above, $\psi^\ell\equiv {\rm Id}\,{\rm mod}\,\frakm^3$ and by induction 
$\psi^{\ell^n}\equiv {\rm Id}$
 modulo $\frakm^{n+2}$. 
 \end{proof}
\smallskip

For $a\in \Q\cap (0, 1/e]$, $r=(1/\ell)^a$, let us set $\varepsilon(r)=(1/\ell)^{\frac{1}{a(\ell-1)}}$ and
\[
 \bar D_{\varepsilon(r)}(1)\times \bar D_r(m)= \{(t, \x)\ |\ |t|_\ell\leq (1/\ell)^{\frac{1}{a(\ell-1)}},\ ||\x||\leq (1/\ell)^{a}\}\subset \bar D_1(1)\times \bar D_r(m).
\]

\begin{prop}\label{propConv2}
Suppose $\psi\equiv {\rm id}\,{\rm mod}\, \frakm^2$. 

a) The series giving $\psi^t(\x)$ converges for $(t, \x)\in \bar D_{\varepsilon(r)}(1)\times \bar D_r(m)$ and defines a rigid analytic map
\[
\psi^t(\x): \bar D_{\varepsilon(r)}(1)\times \bar D_r(m)\to \bar D_r(m).
\]

b)  For $(t, \x)$, $(t',\x)\in \bar D_{\varepsilon(r)}(1)\times \bar D_r(m)$ we have
 \[
 \psi^{t+t'}(\x)=\psi^t(\psi^{t'}(\x)).
 \]
\end{prop}

\begin{proof}
We have formally $\psi^{t\ell^n}(\x)=(\psi^{\ell^n})^t(\x)$ and 
 Lemma \ref{lemmaLn} implies that we can apply Proposition \ref{propConv1} to $\psi^{\ell^n}$ with $N=n+2$.
We obtain that
$
\psi^{t\ell^n}(\x)
$
converges for $||\x||\leq (1/\ell)^a$, $n>1/a(\ell-1)-1$ and $|t|\leq 1$ and part (a) follows. Part (b) follows since  it interpolates the identity $\psi^n(\psi^{n'}(\x))=\psi^{n+n'}(\x)$, which is true for infinite number of pairs $n$, $n'\in \Z$.
\end{proof}

Now formally as power series in $\x$, we have
\begin{equation}\label{der1}
\frac{d\psi^t(\x)}{dt}_{|t=0}=\log(1+\Delta_\psi)(\x)=X_\psi(\x).
\end{equation}

(Hence, $X_\psi(\x)$ can be thought of as the vector field associated to the flow $\psi^t$.)

\begin{prop}\label{propConv3}
a) For all $||\x||<1$, we have 
\[
\frac{d\psi^t(\x)}{dt}=X_\psi(\psi^t(\x)).
\]

b) If $X_\psi({\bf a})=0$  for some ${\bf a}\in \bar\Q_\ell^m$ with $||{\bf a}||<1$, then
$\psi^t({\bf a})={\bf a}$, for all $|t|$ sufficiently small, in particular $\psi^{\ell^n}({\bf a})={\bf a}$, for all $n>>0$.
\end{prop}
\begin{proof}
Using Proposition \ref{propConv2}   and (\ref{der1}) gives that for all $||\x||<1$,
\[
\frac{d\psi^t(\x)}{dt}=\lim_{h\to 0}\frac{\psi^{t+h}(\x)-\psi^t(\x)}{h}=\lim_{h\to 0}\frac{\psi^{h}(\psi^t(\x))-\psi^t(\x)}{h}=
X_\psi(\psi^t(\x)).
\]
Part (b) now follows: 
If $X_\psi({\bf a})=0$ then $ ({d\psi^t }/{dt})({\bf a})=0$, so
$\psi^t({\bf a})={\bf a}$, for all $|t|_\ell$ sufficiently small, so $\psi^{\ell^n}({\bf a})={\bf a}$,
for all $n>>0$.
\end{proof}

\subsection{Vector fields and flows}
Here we recall how an analytic vector field $X$ gives a flow.
Suppose that
\[
X= \sum_{{\bf n}\in {\mathbb N}^{m}}a_{{\bf n}}\x^{\bf n}= \left(X_{t, 1},\ldots , X_{t, m}\right) 
\]
is given by power series in $E\lps x_1,\ldots, x_m\rps$ 
that converge on $||\x||<1$.

\begin{prop}\label{flowExp}
Suppose that for all $||\x||\leq r=(1/\ell)^a\leq 1$, we have  $||X(\x)||_r\leq r$. Set $\ep=(1/\ell)^{1/(\ell-1)}$. Then, there is a unique rigid analytic
map 
\[
h_t: D_1(\epsilon)\times \bar D_r(m)\to \bar D_r(m)
\]
such that
\[
\frac{dh_t(\x)}{dt}=X(h_t(\x)), \quad h_0={\rm id}, \quad h_t(0,\ldots , 0)=(0,\ldots ,0).
\]
\end{prop}

\begin{proof}
We can reduce to the case $r=1$ by rescaling: Consider the (inverse) maps $\ell^a: \bar D(m)\to \bar D_r(m)$, $\ell^{-a}: \bar D_r(m)\to \bar D(m)$ given by scaling by $\ell^a$, resp. $\ell^{-a}$. Then, $\ell^{-a}\circ X\circ \ell^a$ gives a vector field on $\bar D_1(m)$ and $\ell^{-a}\circ h_t\circ \ell^{a}$ is a solution of the ODE above for $\ell^{-a}\circ X\circ \ell^a$ if and only if $h_t$ is a solution for $X$.
In what follows, we assume $r=1$. Now set
\[
h_t(x_1, \ldots, x_m)=\sum_{s\geq 0} \frac{c_{ s}}{s!}t^s
\]
where   $c_{s}\in E\lps x_1,\ldots ,x_m\rps^m$. As in the proof of
\cite{SerreLie}, Thm, p. 158 (see also \cite[\S 5.1, Prop. 8]{HY}), we can solve (uniquely) formally for $c_{s}$ from $c_0=\x$ and
\[
\sum_{s\geq 0} \frac{c_{ s+1}}{s!}t^s=\sum_{{\bf n}} a_{\bf n} (\sum_{s\geq 1} \frac{c_{ s}}{s!}t^s)^{\bf n}.
\]
We see that $c_s$ are given by polynomials 
(with integral coefficients) in the coefficients $a_{s', \bf n}$ of the power series giving $X$ with $s'<s$ and $|{\bf n}|\leq  s$.  Since $||X(\x)||\leq 1$, we have $||a_{s', \bf n}||\leq 1$. We obtain $||c_s||\leq 1$. Since $|s!|_\ell\geq \ep^{s}$, when $|t|_\ell<\ep$, $||c_st^s/s!||\mapsto 0$ for $s\mapsto +\infty$ and convergence follows.
\end{proof}

 \subsection{Poincare lemma}
The Poincare lemma holds for the rigid analytic polydisk $D_1(m)$.  Here, we are only going to use that closed $1$-forms are exact. For completeness, we give a simple proof of this fact. 

\begin{prop}\label{poincare}
Let  $\mu=\sum_{i=1}f_idx_i$ be a closed $1$-form with $f_i\in \sO(D_1(m))$,
for all $i$. Then there is $F\in \sO(D_1(m))$ such that $dF=\mu$. 
\end{prop}

\begin{proof} We follow a standard proof of the ``formal" Poincare lemma. 
First find $F_m\in \sO(D_1(m))$ with $\partial F_m/\partial x_m=f_m$ by formally integrating the variable $x_m$. (The power series $F_m$ converges on $||\x||<1$). Consider $\mu-dF_m=g_1dx_1+\cdots +g_{m-1}dx_{m-1}$
which is also closed. Closedness implies $\partial g_i/\partial x_m=0$, for all $1\leq i\leq m-1$, so the $g_i$ are power series in $x_1,\ldots, x_{m-1}$ only and we can argue inductively. 
\end{proof}

   \newcommand{\sk}{

\end{document}